\documentclass[10pt,nocut]{amsart}
\usepackage[a4paper,hmargin=2cm,vmargin=3cm]{geometry}
\usepackage{amsmath,amssymb,amsthm,mathtools}
\usepackage{mathrsfs,bm,color}
\usepackage{hyperref,fancyhdr,lastpage}
\usepackage{enumitem,calc}
\usepackage{graphicx}
\usepackage{tikz,tikz-cd}
\usetikzlibrary{calc,arrows}
\usepackage{pgfplots}
\usepackage{bbm}
\usepackage{graphicx}
\usepackage{caption}
\usepackage{subcaption}
\hypersetup{
	colorlinks,
	linkcolor={red!50!black},
	citecolor={green!50!black},
	urlcolor={blue!80!black}
}

\usepackage{algorithm}
\usepackage{algorithmic}

\hypersetup{unicode=true,
	colorlinks=true, citecolor=blue, linkcolor=blue,
}

\numberwithin{equation}{section}

\usepackage{thmtools}
\declaretheorem[thmbox=M,name=Theorem,numberwithin=section]{theo}
\declaretheorem[name=Proposition,thmbox=M,numberwithin=section,numberlike=theo]{prop}
\declaretheorem[name=Lemma,thmbox=S,numberwithin=section,numberlike=theo]{lem}
\declaretheorem[name=Corollary,thmbox=M,numberwithin=section,numberlike=theo]{cor}

\makeatletter
\renewenvironment{proof}[1][\proofname]{\par
	\pushQED{\qed}%
	\normalfont \topsep6\p@\@plus6\p@\relax
	\trivlist
	\item[\hskip\labelsep
	\sffamily\bfseries
	#1\@addpunct{.}]\ignorespaces
}{%
	\popQED\endtrivlist\@endpefalse
}
\makeatother

\theoremstyle{definition}

\declaretheorem[name=Remark,numberwithin=section,numberlike=theo]{rem}
\declaretheorem[name=Hypothesis,numberwithin=section,numberlike=theo]{hyp}

\theoremstyle{plain}

\newcommand\given[1][]{\:#1\vert\:}

\newcommand{\out}{\mathrm{out}}

\newcommand{\iii}{i_1,\dots,i_p}
\newcommand{\E}{\mathbb{E}}	

\newcommand{\pert}{\mathrm{pert}}

\newcommand{\sP}{\mathscr{P}}

\newcommand{\N}{\mathbb{N}}

\newcommand{\R}{\mathbb{R}}

\newcommand{\Tr}{\operatorname{Tr}}
\newcommand{\iid}{\text{i.i.d. }}

\newcommand{\eps}{\epsilon}

\DeclareMathOperator{\PP}{\mathbb{P}}
\newcommand{\pP}{P}

\DeclareMathOperator{\1}{\mathbbm{1}}
\DeclareMathOperator{\Cov}{\operatorname{Cov}}

\newcommand{\bx}{\bm{x}}

\newcommand{\bp}{\rho}

\newcommand{\bvx}{x}
\newcommand{\bvy}{y}

\newcommand{\vz}{z}
\newcommand{\vx}{x}
\newcommand{\tx}{\tilde{x}}
\newcommand{\vm}{m}
\newcommand{\vy}{y}

\newcommand{\vl}{\lambda}

\newcommand{\bR}{R}

\newcommand{\bQ}{Q}
\newcommand{\bM}{M}
\newcommand{\bS}{\bm{\Sigma}}
\newcommand{\bA}{A}

\newcommand{\bD}{\Delta}
\newcommand{\bd}{D}

\newcommand{\bY}{Y}

\newcommand{\bC}{\bm{C}}

\newcommand{\bB}{\bm{B}}

\newcommand{\bI}{I}

\newcommand{\diag}{\mathrm{diag}}

\DeclareMathOperator{\supp}{\mathrm{supp}}
\newcommand{\trans}{\mathsf{T}}

\newcommand{\Var}{\operatorname{Var}}

\renewcommand{\phi}{\varphi}
\renewcommand{\epsilon}{\varepsilon}

\begin{document}
	
	\title[Low-rank Matrix Estimation with Inhomogeneous Noise]
	{Low-rank Matrix Estimation with Inhomogeneous Noise}
	
	\author{Alice Guionnet*, Justin Ko*, Florent Krzakala, Lenka Zdeborov\'a}
	
	\thanks{*Supported in part by ERC Project LDRAM:  ERC-2019-ADG	Project 884584}
	
	\begin{abstract}
		%This article is a rigorous extension of previous work of the authors  to output channels that are inhomogeneous. 
		%
		We study low-rank matrix estimation for a generic inhomogeneous output channel through which the matrix is observed. This generalizes the commonly considered spiked matrix model with homogeneous noise to include for instance the dense degree-corrected stochastic block model. 
		We adapt techniques used to study multispecies spin glasses to derive and rigorously prove an expression for the free energy of the problem in the large size limit, providing a framework to study the signal detection thresholds. 
		We discuss an application of this framework to the degree corrected stochastic block models. %and propose a denoising algorithm that achieves the theoretical minimal mean squared error. 
	\end{abstract}
	
	\maketitle
	\section{Introduction}

	Heterogeneity is a fundamental part of real world problems, as opposed to simpler homogeneous modeling assumptions. In this paper, we investigate the effect of such inhomogeneity in spike models and community detection. 
	
	We focus on a mathematical formulation, common  in statistics, inference and machine learning, where the aim is to reconstruct a rank $\kappa$ vector $\bvx^0$ in
	$\R^{\kappa \times N} $ from noisy measurements $\bd$ of the inner product of the vector.
	A substantial amount of recent work focuses on this issue in high dimension, as $N \to \infty$, and under separable priors on $\bvx^0$. Applications of this setting range from the  Wigner and Wishart spiked models, to the stochastic block model, sparse PCA, or clustering mixtures of Gaussians (see, e.g. \cite{BBP,deshpande2017asymptotic,donoho2018optimal,lenkarmt,lesieur2016phase,lelargemiolanematrixestimation,barbier2019adaptive}).
	Here, we study this problem in the case of inhomogeneous noise. In this situation, both the law and the strength of the noisy measurements $\bd$ can be different for any pairs $i,j$,
	\begin{equation}\label{eq:lawD}
		D_{ij} \sim		\pP_{ij} \Big( D_{ij} \given[\Big] \frac{\vx_i^0 \cdot \vx_j^0 }{\sqrt{N}}  \Big). 
	\end{equation}
	One of the main motivation for this problem, on which we shall focus for the concrete application of our results, is a dense version of a well known model of community detection called the degree-corrected stochastic block model (DCSBM)~\cite{Blockmodel}. 
	%As we shall see, in the dense limit, the  degree-correction indeed yields an equivalent Gaussian noise with a structured homogeneous noise. 

	\subsection{Highlights of our main contributions}
	We now summarize our main results that are then described precisely in Section \ref{sec:main_results}.
	\begin{enumerate}
		\item[	(i)] We  prove a \textbf{Gaussian universality theorem} ---
		that generalizes the ``homogeneous" universality proven in \cite{krzakala2016mutual} ---
		that shows that for a matrix factorization problem with inhomogeneous noise distributions, a large class of noise models (including for instance sign flips, or additive non-Gaussian noise, and the DCSBM), one may transform the model into an equivalent Gaussian one. This means that we need only to consider the case where the distribution $P_{ij}$ is Gaussian. This transformation   amounts into working with the Fisher score matrix. This results is of crucial importance as it allows to study an entire, complex, and diverse family of statistical model just by focusing on an equivalent spike model problem \cite{donoho2013optimal} with a Gaussian noise. The universality at the level of free energy is stated in Theorem~\ref{prop:universality} and the stronger form at the level of the spectrum is stated in Theorem~\ref{theo:univspec}. 
		
		\item[(ii)] \textbf{Proof of the free energy} -- Focusing on the spike model with Gaussian noise,  we then study the  value of the likelihood ratio with the corresponding null model, (or equivalently the free energy in statistical physics terminology, or the mutual information in information theory).	We then prove a formula for the asymptotic value of the likelihood ratio in Theorem~\ref{thm:main} and Theorem~\ref{thm:main2}, generalizing the homogeneous results \cite{deshpande2017asymptotic,dia2016mutual,lelargemiolanematrixestimation}. This is achieved by using methods from mathematical physics of spin glasses.
		
		\item [(iii)] We study the free energy and determine the \textbf{phase boundary} (see Lemma~\ref{prop:recovery}) in terms of the signal-to-noise ratio, focusing in particular to the dense DCSBM model. The phase transition marks a transition from a phase where it is information theoretically impossible to reconstruct the community better than a random guess, from a phase where it is possible to do so.  This phase boundary is compared to the separation of the largest eigenvalue of the matrix after a naive transformation to transfer the noise profile to the signal. We show that the regime where such a transformation leads to a extremal eigenvalue is contained in the information theoretically to detect the signal in Proposition~\ref{prop:gap}.
		%
		%We prove in particular a sharp threshold in eq.~\eqref{prop:recovery}.
		
		%	We also show that that the standard spectral method, Principal-Component-Analysis \cite{BBP}, fails to identity the communities at the transition, and instead outputs an estimate better than a random guess for larger signal to noise ratio than is information theoretically needed. This non-optimality of standard PCA is at variance with the standard case of the standard SBM where the transition PCA is optimal (at least for $q=2$ communities) \cite{deshpande2017asymptotic,BBP}.
		%}
\end{enumerate}

\subsection{Relation with previous work}

The universality property has been used extensively in the homogeneous cases, see e.g. \cite{lesieur2015mmse,krzakala2016mutual,perry2016optimality}. It is in particular at the roots of the identification of the dense stochastic block model with an equivalent spike model \cite{lesieur2015mmse,deshpande2017asymptotic}. Such spike models \cite{donoho2013optimal} have been the subject of many studies over the last few years, in particular in random matrix theory \cite{BBP}, in statistical inference (e.g. without pretension of exhaustivity \cite{rangan2012iterative,deshpande2014information,lenkarmt,perry2016optimality,barbier2020information,Camilli2022AnIP}) with many different applications \cite{deshpande2014sparse,deshpande2015finding,lesieur2016phase}. 

In particular, the last decade has witnessed spectacular progress in the rigorous approach to the computation of the asymptotic mutual information for such problems \cite{krzakala2016mutual,dia2016mutual,lelargemiolanematrixestimation,el2018estimation,barbier2019adaptive,el2020fundamental}. We shall use these techniques, in particular the one of \cite{lelargemiolanematrixestimation}, in our approach.

Community detection of one of the most fundamental problem of graph theory. The connection between the low rank factorization problem and the SBM  was unveiled in  \cite{lesieur2015mmse,deshpande2017asymptotic}. Here we shall be instead focusing on the inhomogenous Degree-Corrected SBM \cite{Blockmodel}, a much more realistic model. On a side note, the inhomogenous setting that we shall be looking at has deep connection with the spatial coupling introduced in \cite{javanmard2013state,dia2016mutual}.

To compute the limit of the free energy, we use a modification of the synchronization mechanisms for multispecies \cite{PMulti} and vector spin glasses \cite{PPotts, PVS}. By adding some extra terms to the perturbations, we can regularize our posterior probability by introducing the Ghirlanda--Guerra identities while preserving the Nishimori identities and the concentration of the overlaps. The synchronization of spin glasses was also recently adapted in other contexts \cite{kosphere, multisphere1, multisphere2,quantumspinglass,Alberici_2021} to compute the free energy of various models. A different point of view to study multispecies models using the TAP approach was also applied recently in \cite{subagmulti1,subagmulti2}.

\section{Main results}
\label{sec:main_results}

\subsection{Setting}
We consider the following inference problem. Given a probability measure $\pP_0$ on $\R^\kappa$ consider a \emph{signal} consisting of $N$ independent copies $\bvx^0 = (\vx^0_1, \dots, \vx_N^0) \in \R^{\kappa \times N}$ generated from the product measure
\[
\bvx^0 \sim \prod_{1\le i \leq N} \pP_0(\vx^0_i).
\]
Given the inner products of the signal $w_{ij}^0 = \frac{\vx_i^0 \cdot \vx_j^0}{\sqrt{N}}$, we generate independently some \emph{observed data} ${D}_{ij}$
conditionally on $w_{ij}^0$ according to the  probability measure
\begin{equation}\label{eq:conddist}
	\pP_{ij} \Big( D_{ij} \given[\Big] \frac{\vx_i^0 \cdot \vx_j^0 }{\sqrt{N}}  \Big). 
\end{equation}

A critical distinction in these inhomogeneous models is the fact that while the data $D_{ij}$ are generated independently,  the conditional distributions change depending on the indices $i,j$. The problem we are interested in is the estimation of the signal given an observation of the {\it observed data} $D = (D_{ij})_{i,j \leq N}$.

The posterior  probability of the signal given $D$ can be expressed in the form of a inhomogeneous vector spin model with respect to an arbitrary function $g_{ij}$ and probability measure $\pP_X$
\[
d\pP(X \given D) = \frac{1}{Z_X^g(D)} \prod_{1 \leq i<j \leq N} e^{g_{ij} ( D_{ij}, \frac{\vx_i \cdot \vx_j}{\sqrt{N}}  )} \prod_{1 \leq i \leq N} d\pP_X(\vx_i)
\]
where 
\begin{equation}\label{eq:eq:partition}
	Z^g_X(D)= \int \prod_{1 \leq i<j \leq N} e^{g_{ij} ( D_{ij}, \frac{\vx_i \cdot \vx_j}{\sqrt{N}}  )} \prod_{1 \leq i \leq N} d\pP_X(\vx_i)
	\,.
\end{equation} 
Our study will be restricted to the so-called Bayes optimal setting which amounts to make the following hypothesis:
\begin{hyp}[Bayes-optimality]\label{hypbayes}
	Suppose that $\pP_0 = \pP_X$ and if $\pP_{ij}$ is the distribution of $D_{ij}$ in \eqref{eq:conddist}, the function $g_{ij}(D,w)$ is the log-likelihood: 
		$$ g_{ij}(D,w)=\ln \pP_{ij} \bigg(D \given[\Big] \frac{\vx_i^0 \cdot \vx_j^0}{\sqrt{N}}=w \bigg)\,.$$
		where  we use  in short $\pP_{ij}$ to denote  also the  probability density function of $\pP_{ij}$.
\end{hyp}

Our goal is to compute the normalized free energy 
\begin{equation}\label{eq:FEgrowingrank}
	F_N(g) = \frac{1}{N} \bigg( \E_D  \Big(\ln Z^g_X(D) - \sum_{i < j} g ( D_{ij}, 0  )
	\Big) \bigg)
\end{equation}
where  $\E_D$ is the expectation under $\otimes \pP_{ij}(D_{ij}|\frac{ \vx_i^0 \cdot \vx_j^0}{\sqrt{N}})\pP^{\otimes N}_0(\vx_i^0)$ and the function $Z_X^g$ was defined in \eqref{eq:eq:partition}. Notice that, with this definition, the free energy is nothing but the expected log-likelihood ratio (under data generated by the model) between the likelihood that data are generated by the present model with the likelihood that they are generated from the null model (where there is no signal at all):
\begin{equation}\label{eq:FEgrowingrank_ratio}
	F_N(g) = \frac{1}{N} \bigg( \E_D \log \frac {\pP_D}{\pP_{D|\bvx^0=0}}  \bigg).
\end{equation}

%We will denote $\E_X$  the average with respect to $\pP^{\otimes N}_X(\vx_i)$. 

The limit of the free energy will depend on the following ``Fisher Information" matrix, defined as the expectation of the Fisher score:
\begin{equation}\label{eq:fisherscore}
	\frac{1}{\Delta_{ij} } = \E_{P_{ij}(D|w=0)} (\partial_w g_{ij}(D,0) )^2, \quad 1\le i<j\le N.
\end{equation}
We first assume that this matrix of variances is piece-wise constant.
\begin{hyp}[Block-constant Noise Profile]\label{hypDelta} Given $n \geq 1$, there exists a partition  of $[N]$ 
	\[
	[N] = \bigsqcup_{s =1}^n I_s
	\]
	such that the $\Delta_{i,j}$ are constant in the groups $I_s \times I_t$ for $s,t \in \{1,\dots,n\}$
	\begin{equation}\label{defD}
		\Delta_{ij} = \Delta_{st}, \quad \text{ for $i \in I_s$, $j \in I_t$}
	\end{equation}
	and $(\Delta_{st})_{s,t\leq n}$ are independent of $N$. 
	Furthermore, the proportions of configurations in each group converges in the limit
	\begin{equation}\label{eq:nt}
		\frac{|I_s|}{N} \to \rho_s \in (0,1) \quad \text{for all} \quad s \leq n.
	\end{equation}
	We will also assume that $\Delta_{s,t}$ belongs to $(0,+\infty)$ for each $s,t$ and the $n\times n$ symmetric  matrix $\frac{1}{\Delta}$  with entries
	$(\frac{1}{\Delta_{s,t}})_{s,t \leq n}$  is positive semidefinite.
\end{hyp} 

We finally describe our technical hypotheses. We first need to assume that the signal is compactly supported.
\begin{hyp}[Compact Support]\label{hypcompact}
	$\pP_0$ and $\pP_X$ are compactly supported so that $\vx$ and $\vx^0$ take values in $[-C,C]^\kappa$ for some finite~$C$. We also assume that $\kappa$ is independent of $N$. 
\end{hyp}
This hypothesis  implies that, uniformly, we have 
\begin{equation}
	|w_{ij}| = \Big| \frac{ \vx_i \cdot \vx_j}{\sqrt{N}} \Big| \leq \frac{C^2\kappa}{\sqrt{N}} .
\end{equation}
This uniform bound will allow to expand the functions $g_{ij}$ in the variables $w_{ij}$. To do so, we need to assume sufficient regularity of the $g_{ij}$, namely that, if $\|\cdot \|$ denotes the supremum norm: 
\begin{hyp}[Regularity of Log Likelihood]\label{hypg}
	$\|\partial_w g_{ij}(\cdot ,0)\|, \|\partial_w^2 g_{ij}(\cdot ,0)\|, \|\partial_w^3g_{ij}\|$ are bounded, uniformly in  $1< i<j\le N$ and $N\in \mathbb N$.
\end{hyp}
\begin{rem}
	We can weaken the condition on the first derivative of $\partial_w g_{ij}$. For the proof of universality in disorder in Lemma~\ref{lem:univ3}, we require the third moment $\E_{\pP_{ij}(D\given w^0) } [\partial_w g_{ij}(D,w^0) ]^3$ to be bounded for all $D$ and all $w^0$ in the support. 
\end{rem}

There assumptions will be used to prove a universality result which states that a class of statistical inference problems are equivalent to a low rank matrix factorization problem where the noise matrix has a variance profile. In particular, we shall use as an application the degree-corrected stochastic block model.

%generalization of a theorem for homogenous problem in *cite*

%Random matrices with variance profiles were studied in **cite**. These are closely related with the multispecies spin glass models, which were studied in **cite**.

\subsection{ Gaussian Estimation Problems with Covariance Profiles}
The key point in our approach is that the general inhomogeneous vector spin model can be reduced to a model spiked Gaussian matrix model with a variance profile.

\subsubsection{Effective data matrix}\label{subsec:effective}
Our first observation says that the derivatives up to the second order derivatives of the function $g$ encode asymptotically all its information. This holds even without the Bayes optimality Hypothesis~\ref{hypbayes}.

\begin{lem}[Free Energy Universality 1]
	Suppose that Hypothesis~\ref{hypg} and  Hypothesis~\ref{hypcompact} are satisfied. 
	Let $B$ be a $\sigma$ algebra such that
	the $D_{ij}$ are independent conditionally to $B$.
	Then 
	$$F_N(g)=  F_N(\bar g) +O\Big(\frac{\kappa^2}{\sqrt{N}}\Big)$$
	with \[
	\bar  
	g_{ij}(D,w) = g_{ij}( D_{ij},0) + \partial_w g_{ij}( D,0)  w + \frac{1}{2} \E_D[ \partial_w^2 g_{ij}(D,0)\given B] w^2 \,.
	\] 
	%where $\E_Y[ \cdot \given B]$ is any conditional expectation so that $\E_Y ( \partial_w^2 g_{ij}(Y,0)- \E_Y[ \partial_w^2 g_{ij}(Y,0)\given B] )=0$.
\end{lem}
This is a first fundamental universality result that motivates our study. Informally, this means that in order to perform (in high-dimensions)  estimation when we observed the data  $D$ given by the likelihood
$$ g_{ij}(D,w)=\ln  \pP_{ij} \bigg( D_{ij} \given[\Big] \frac{\vx_i^0 \cdot \vx_j^0}{\sqrt{N}}=w \bigg)\,,$$ we can simply create an {\it effective} data matrix $Y$ based on the Fisher score and Fisher information as 
\begin{equation}\label{eq:spike}
	Y_{ij} = \Delta_{ij} \partial_w g_{ij}(D,0), \,{\rm with} \,		\frac{1}{\Delta_{ij}} = \E_{P_{ij}(D|w=0)} (\partial_w g_{ij}(D,0) )^2\,
\end{equation}
then the free energy (or the likelihood ratio) depends only on this new matrix. 

\subsubsection{Equivalent spike model}
Under the Bayes optimality assumption Hypothesis~\ref{hypbayes}, we can further simplify the free energy by connecting with a Gaussian spike model. i.e. a model where the effective matrix $Y$ is indeed sampled from a Gaussian spike model with inhomogeneous noise. 

For a $N \times N$  Gaussian matrix $W$, a deterministic $N \times N$ standard deviation matrix $\bD^{\odot \frac{1}{2}}$, and a $N \times \kappa$ matrix  $\bvx^0$ with column vectors $x_i^0 \in \R^\kappa$ we want to estimate
\begin{equation}\label{eq:inhomospiked}
	Y = Y^\Delta  = \bD^{\odot \frac{1}{2}} \odot W + \sqrt{\frac{1}{N}} \bvx^0 (\bvx^0)^\trans 
\end{equation}
in other words, for $1 \leq i < j \leq N$
\[
Y_{ij} = \sqrt{\Delta_{ij}} W_{ij} + \frac{1}{\sqrt{N}}\vx^0_i \cdot \vx^0_{j}.
\]
Note that the $\Delta_{ij}$ are non negative by \eqref{eq:spike}. 
We also only care about the off-diagonal terms because the diagonals are negligible.

The main difference in this setting, in contrast to the standard spiked matrix models, is that the coordinates $i,j$ play an important role in the behavior of our model. %To add some more structure to our problem, we will assume that $\bm{\Delta}$ takes finitely many values. {\alice:erased the $\sqrt{\Delta}$ below}
Observe that
\[
\Delta_{ij}^{-\frac{1}{2}} \Big( Y_{ij} - \frac{1}{\sqrt{ N}} \vx^0_i \cdot \vx^0_{j} \Big) 
\]
follows a standard Gaussian law.
Then the random posterior distribution of $X = (\vx_1, \dots, \vx_N)$ is
\[
d\pP( X \given Y) = \frac{1}{Z} \exp \biggl( - \sum_{i < j} \frac{1}{2 \Delta_{ij}}  \bigg( Y_{ij} - \frac{1}{\sqrt{N}} \vx_i \cdot \vx_{j} \bigg)^2 \biggr) \, d\pP^{\otimes N}_0(X).
\]
After absorbing the terms that do not depend on $X$ into the normalization, the density is encoded by the Hamiltonian given by
\begin{align}
	H_N(\bvx) &= \sum_{i < j} \frac{1}{ \sqrt{\Delta_{ij} N} } Y_{ij} (\vx_i \cdot \vx_{j} ) -   \frac{1}{2 \Delta_{ij} N} (\vx_i \cdot \vx_{j} )^2 \notag
	\\&= \sum_{i < j} \frac{1}{ \sqrt{ N \Delta_{ij}} } g_{ij} (\vx_i \cdot \vx_{j} ) + \frac{1}{\Delta_{ij} N} (\vx^0_i \cdot \vx^0_{j} )(\vx_i \cdot \vx_{j} )   -   \frac{1}{2 \Delta_{ij} N} (\vx_i \cdot \vx_{j} )^2. \label{eq:HamilGauss}
\end{align}
We define
\begin{equation}\label{eq:FEGauss}
	F_N(\bD) = \frac{1}{N} \E_Y \log \int e^{ H_N(\bvx) } \, d \pP_0^{\otimes N}(x)
\end{equation}
where $H_N(\bvx)$ is the Hamiltonian defined in \eqref{eq:HamilGauss} and $\bD$ is the variance profile. 

We will prove that solving the free energy of Gaussian estimation problems are equivalent to solving the free energy of general inhomogeneous vector spin models with a specific choice of parameters in the Bayes optimal setting:
\begin{theo}[Free Energy Universality 2] \label{prop:universality}
	Suppose we are in the Bayes optimal setting of Hypothesis~\ref{hypbayes},  $g$ satisfies Hypothesis~\ref{hypg}  and the signal space is compact as in Hypothesis~\ref{hypcompact}. If we define 
	\[
	\frac{1}{\Delta_{ij}} = \E_{P_{ij}(Y|w=0)} (\partial_w g_{ij}(Y,0) )^2 = \int (\partial_w g_{ij}(Y,0) )^2 e^{g_{ij}(y,0)} \, dy,
	\]
	then the free energy $F_N(g)$ of the inhomogeneous vector spin models \eqref{eq:FEgrowingrank}
	satisfies
	\[
	\big| F_N(g) - F_N(\bD) \big| =  O( \kappa^3 N^{-1/2} ).
	\]
\end{theo}

\begin{rem} We will need that the
	matrix $1/\bD$ satisfy Hypothesis \ref{hypDelta} or \ref{hypDeltac} to get the limit of the free energy. Similarly, $\kappa$ could go to infinity with $N$ provided $\kappa^3 N^{-1/2}$ goes to zero but then one would need to understand the asymptotics of $\sup_{\bQ} \phi( \bQ )$. 
\end{rem}

We will also derive a stronger form of universality at the level of the spectrum of random matrices, instead of at the level of the free energy. Consider the transformed data matrix
\[
\frac{\tilde Y_{ij}}{\sqrt{N}} =  \frac{1}{\sqrt{N}} \bigg( \partial_w g_{ij} (D,w) \bigg) \bigg|_{D = D_{ij}, w = 0} \qquad i,j \leq N
\]
where $D_{ij}$ is the random variable with law given by \eqref{eq:lawD} and the spiked matrix with variance profile $Y^\Delta$
\[
Y^\Delta =   \bD^{\odot \frac{1}{2}} \odot W + \frac{\bvx^0 (\bvx^0)^\trans }{\sqrt{N} }
\]
defined in \eqref{eq:inhomospiked}. Under some conditions on the smallest entries of the Fisher information matrix $\frac{1}{\Delta}$ (see Hypothesis~\ref{hypQVE}), we have the following universality result for the spectrum. 

\begin{theo}[Universality of the Spectrum]\label{theo:univspec}
	If $g$ satisfies Hypothesis~\ref{hypDelta} and the corresponding Fisher information matrix \eqref{eq:fisherscore} satisfies Hypothesis~\ref{hypQVE}, then
	\begin{enumerate}
		\item Conditionally on $\bvx^0$, the empirical distribution $\mu_1$ of the eigenvalues of $\frac{\tilde Y_{ij}}{\sqrt{N}}$ and the empirical distribution $\mu_2$ of the eigenvalues of $\frac{1}{\sqrt{N} \bD} \odot \bY^\Delta$ satisfy
		\[
		\lim_{N \to \infty} d( \mu_1, \mu_2) \to 0
		\]
		in probability, where $d$ is a distance compatible with the weak topology. 
		\item Conditionally on $\bvx^0$, when the dimension goes to infinity, $\frac{1}{\sqrt{N} \bD} \odot \bY^\Delta$ has an extremal eigenvalue away from the bulk if and only if 
		$\frac{\tilde Y_{ij}}{\sqrt{N}}$ does, for almost all 
		$\bD$ and $\bp$. 
	\end{enumerate}
\end{theo}

\subsection{The Limit of the Free Energy and Consequences}

Given a sequence  $\bQ = (\bQ_s)_{s \leq n}$ of symmetric $\kappa\times\kappa$ positive semidefinite matrices, the replica symmetric free energy is given by 
\begin{align}
	\phi(\bQ) &= - \sum_{s,t =1}^n \frac{\rho_s \rho_t}{4 \Delta_{st}}\Tr( \bQ_s \bQ_t) +  \sum_{s =1}^n \rho_s \E_{\vz,\vx^0} \ln \bigg[ \int e^{ ( \tilde \bQ_s\vx^0 + \sqrt{\tilde \bQ_s} \vz )^\trans \vx - \frac{\vx^\trans \tilde \bQ_s \vx}{2}} \, d \pP_0(\vx) \bigg] \label{eq:RSformula}
\end{align}
where 
\[
\tilde \bQ_s = \sum_{t = 1}^n \frac{1}{\Delta_{st}} \rho_t \bQ_t,
\]
and $\vx^0 \sim \pP_0$ and $\vz \sim N(0,\bI_r)$ are independent. We have the following limit of the free energy.

\begin{theo}[Bayes Optimal Free Energy] \label{thm:main}
	Suppose that Hypotheses  \ref{hypbayes} and \ref{hypDelta} are verified, as well as  technical Hypotheses \ref{hypcompact} and  \ref{hypg}. Then, 
	\[
	\lim_{N\rightarrow\infty} F_N(g) = \sup_{\bQ} \phi( \bQ ).
	\]
\end{theo}

We can generalize Theorem \ref{thm:main} to the case where $\Delta$ is not piecewise constant, but we then need to approximate it by piecewise constant matrices.  We then replace Hypothesis \ref{hypDelta} by the following: 
\begin{hyp}[Kernel Regularity]\label{hypDeltac}
	Assume that there exists a non-negative measurable function $\Delta(s,t)$ such that
	$$\lim_{N\rightarrow \infty} \sup_{s,t\in [0,1]}\left|\frac{1}{\Delta_{\lfloor sN\rfloor ,\lfloor tN\rfloor }}-\frac{1}{ \Delta(s,t)}\right|=0.$$
	We will also assume that there exists $\epsilon>0$ such that $\Delta(s,t)$ belongs to $(\eps,1/\eps)$ for each $s,t\in [0,1]$ and the  symmetric  operator $\frac{1}{\Delta}$  on $L^2([0,1])$ given by 
	$$\frac{1}{\Delta}f(t)=\int_0^1 \frac{1}{\Delta(t,s)}f(s) ds$$
	is non-negative. 
\end{hyp} 

\begin{rem}
	This hypothesis is satisfied if, for instance, $\Delta$ is $C^0$ and bounded below.
\end{rem}
Under this hypothesis, Theorem \ref{thm:main} generalizes as follows. Let $\bQ_s : [0,1] \mapsto \bS_\kappa^+$ be a measurable function with values in the set of  $\kappa\times \kappa $ symmetric definite matrices, and define
\begin{align*}
	\tilde \phi(\bQ) &:= - \int_{0}^1 \frac{1}{4 \Delta(s,t)}\Tr( \bQ_s \bQ_t) ds dt+  \int_0^1 ds \E_{\vz,\vx^0} \ln \bigg[ \int e^{  ( \tilde \bQ_s\vx^0 + \sqrt{\tilde \bQ_s} \vz )^\trans \vx - \frac{\vx^\trans \tilde \bQ_s \vx}{2} } \, d \pP_0(\vx) \bigg]  
\end{align*}
where 
\[
\tilde \bQ_s = \int_0^1\frac{1}{\Delta({s,t})} \bQ_t dt.
\]
Then, we have:
\begin{theo}[Bayes Optimal Free Energy with General Kernel] \label{thm:main2}
	Suppose that Hypotheses  \ref{hypbayes} and \ref{hypDeltac} are verified, as well as the technical Hypotheses \ref{hypcompact} and  \ref{hypg}. Then, 
	\[
	\lim_{N\rightarrow\infty} F_N(g) = \sup_{\bQ } \tilde \phi( \bQ ).
	\]
\end{theo}

When the variance profile is discrete and if the probability measure $\pP_0$ is centered, the maximizers of $\phi$ defined in \eqref{eq:RSformula} satisfy the following fixed point equation
\begin{equation}
	\tilde \bQ_s = \sum_{t =1 }^n \frac{\rho_t}{ \Delta_{s,t} } \E \langle \vx  \rangle_{\tilde \bQ_t} \langle \vx \rangle_{\tilde \bQ_t}^\trans
\end{equation}
where $\langle \cdot \rangle_{\tilde \bQ_t}$ denotes the average
\[
\langle f(x) \rangle_{\tilde \bQ_t} = \frac{\int f(x) e^{  ( \tilde \bQ_s\vx^0 + \sqrt{\tilde \bQ_s} \vz )^\trans \vx - \frac{\vx^\trans \tilde \bQ_s \vx}{2}   } \, d \pP_0(\vx)}{ \int e^{  ( \tilde \bQ_s\vx^0 + \sqrt{\tilde \bQ_s} \vz )^\trans \vx - \frac{\vx^\trans \tilde \bQ_s \vx}{2}   } \, d \pP_0(\vx) }.
\]
This fact allows us to compute the informationally theoretical optimal thresholds of the inhomogeneous estimation problems.

We define the matrix minimal means square estimator of our signal $\bx^0 (\bx^0)^\trans$ by
\[
\mathrm{MMSE}(N) = \min_\theta \frac{2}{N(N - 1)} \sum_{i < j} \E( x^0_i \cdot x^0_j - \theta_{i,j}(\bY)  )^2
\]
where the minimum is over all possible estimators $\theta$ that only depend on the data $\bY$. 
We have the following result for the limit of the minimal mean squared error.
\begin{cor}[Limiting MMSE]\label{prop:limitingMMSE}
	Suppose that Hypotheses  \ref{hypbayes} and \ref{hypDelta} are verified, as well as  technical Hypotheses \ref{hypcompact} and  \ref{hypg}. Then for almost all $\bD$ and $\bp$, for any  maximizer $(\bQ_1, \dots, \bQ_n)$ of \eqref{eq:RSformula} and
	\[
	\lim_{N \to \infty} \mathrm{MMSE}(N) = \E_{\pP_0} \| x x^\trans \|_2^2 - \sum_{s,t = 1}^n \rho_s \rho_t \Tr( \bQ_s \bQ_t ).
	\]
\end{cor}
It follows that the maximizers of \eqref{eq:RSformula} and the fact that they vanish or not  are essential to quantify the information theoretic thresholds for these inhomogenous factorization problems. Notice that we do not prove the uniqueness of the maximizers but the uniqueness of the values of $\left(\Tr( \bQ_s \bQ_t )\right)_{s,t}$ regardless of the choice of the maximizer, which is enough to guarantee the above statement.

We now classify detectability phase transitions with respect to the size of the variance profiles. Consider the noise parameter matrices
\[
\bp = \diag( \rho_1, \dots, \rho_{n} ) \in \R^{n \times n} \qquad \frac{1}{\bD} = \bigg( \frac{1}{\Delta_{s,t}} \bigg) \in \R^{n \times n}.
\]
The size of the noise of the inhomogeneous models can be encoded by the largest eigenvalue of the matrix $\sqrt{\bp} \frac{1}{\bD} \sqrt{\bp}$. Let $\| \cdot \|_{op}$ denote the operator norm of the matrix, which is equivalent to the largest eigenvalue when the matrix is symmetric and positive semidefinite. We have the following thresholds on recovery.

\begin{lem}[Recovery Transitions]\label{prop:recovery} Suppose that Hypotheses  \ref{hypbayes} and \ref{hypDelta} are verified, as well as  technical Hypotheses \ref{hypcompact} and  \ref{hypg}. Furthermore, suppose that $\pP_0$ is symmetric. We have that
	\begin{enumerate}
		\item If
		\[
		\bigg\| \sqrt{\bp} \frac{1}{\bD} \sqrt{\bp}  \bigg\|_{op}  < \frac{1}{9 \kappa^4 C^6}
		\]
		then $\lim_{N \to \infty} F_N(g) = 0$ and $\lim_{N \to \infty} \mathrm{MMSE}(N)  =  \E_{\pP_0} \| x x^\trans \|_2^2$.
		\item If
		\[
		\bigg\| \sqrt{\bp} \frac{1}{\bD} \sqrt{\bp}  \bigg\|_{op} > \frac{1}{\| \Cov(x) \|_{op}^2}
		\]
		then $\lim_{N \to \infty} F_N(g) > 0$ and $	\lim_{N \to \infty} \mathrm{MMSE}(N)   < \E_{\pP_0} \| x x^\trans \|_2^2$.
	\end{enumerate}
\end{lem} 

\begin{rem}
	If $\pP_0$ is a standard Gaussian, then we can improve the first bound to conclude that if
	\[
	\bigg\| \sqrt{\bp} \frac{1}{\bD} \sqrt{\bp}  \bigg\|_{op}  < 1
	\]
	then the signal is not recoverable, which gives us a sharp phase transition when the priors are Gaussian. Technically $\pP_0$ does not have compact support, so it will violate Hypothesis~\ref{hypcompact}, so universality of models with Gaussian signals will require a bit more work to show. However, the computation of the free energy of the inhomogenous spiked matrix model \eqref{eq:inhomospiked} does not require the compact support Hypothesis, so we do have a rigorous proof of the sharp transition in this case. 
\end{rem}

There is a gap between the two thresholds in Lemma~\ref{prop:recovery}. To explore the phase transition, we can numerically solve for the maximizers of $\phi$ defined in \eqref{eq:RSformula} to test if the phase transition is tight. 

We first check the behavior for $\kappa = 1$, $n = 2$, and a non-sparse symmetric prior
\[
\pP_0(\pm 1) = \frac{1}{4} \quad \pP_0(0) = \frac{1}{2}.
\]
Under these assumptions, it follows that $\|  \Cov(x)\|_{op}^2 = \frac{1}{4}$, so the phase transition in part 2 of Lemma~\ref{prop:recovery} happens at $\bigg\| \sqrt{\bp} \frac{1}{\bD} \sqrt{\bp}  \bigg\|_{op} 
	=4$.  For matrix paths $\bp(t)$ and matrix $\bD(t)$ we plot the free energy $F_N(g)$ corresponding to the parameters $\bD(t)$ and $\bp(t)$ versus the operator norm of $\sqrt{\bp(t)} \frac{1}{\bD(t)} \sqrt{\bp(t)}$ where $t \in \R^+$. We consider the following noise profiles
\[
\frac{1}{\Delta_1(t)} = \begin{bmatrix}
	\frac{1}{2} & 0\\
	0 & t 
\end{bmatrix} \qquad \frac{1}{\Delta_2(t)} = \begin{bmatrix}
	t & \frac{t}{2}\\
	\frac{t}{2} & t 
\end{bmatrix} \qquad \frac{1}{\Delta_3(t)} = \begin{bmatrix}
	\frac{t}{3} & \frac{t}{4}\\
	\frac{t}{4} & t 
\end{bmatrix}
\]
with proportion $\bp = \diag(0.4,0.6)$. Figure~\ref{fig:dense_general} shows the relationship between the free energy and the inhomogeneous noises at the $3$ different choices of the matrices $\bD$ and~$\bp$.  We see in the figure that the transition derived in part 2 of Lemma \ref{prop:recovery} appears tight in this case when the prior is not sparse.

%We can fix matrices $\bp_k$ and matrix $\bD_k$ and plot the free energy $F_N(g)$ corresponding to the parameters $t \bD$ and $\bp$ versus the operator norm of $\frac{1}{t} \sqrt{\bp_1} \frac{1}{\bD} \sqrt{\bp_1}$ where $t \in \R^+$. Figure \ref{fig:dense_general} shows the relationship between the free energy and the homogeneous noises at $3$ different choices of the matrices $\bD$ and~$\bp$. We see in the figure that the transition derived in part 2 of Lemma \ref{prop:recovery} is tight in this case when the prior is not sparse.
\begin{figure}
	\begin{center}
		\includegraphics[width=13cm]{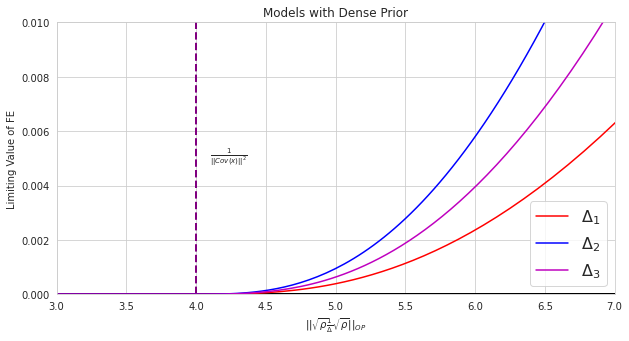}
		\includegraphics[width=13cm]{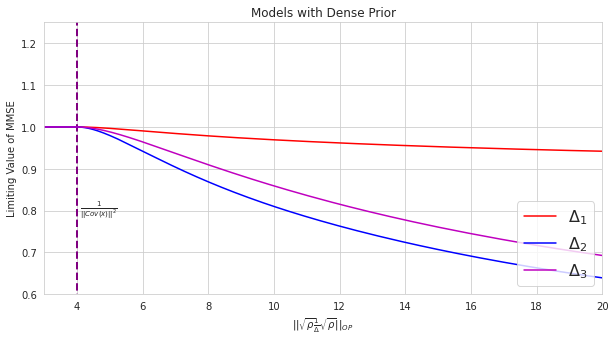}
		\caption{The free energy as a function of the noise. A case with a continuous phase transition (dashed line) separating the undetectable regime from one where detection of the signal is possible.}
		\label{fig:dense_general}
	\end{center}
\end{figure}

Another example, where the transition in part 2 of Lemma \ref{prop:recovery} does not appear to be tight is when estimating a sparse signal. We thus now examine the behavior for $\kappa = 1$, $n = 2$, and symmetric prior
\[
\pP_0(\pm 1) = 0.03 \quad \pP_0(0) = 0.94.
\]
Under these assumptions, it follows that $\frac{1}{\|  \Cov(x)\|_{op}^2} \approx 278$.  In Figure \ref{fig:sparse_general} we plot the limit of the free energy of the matrix $$\frac{1}{\bD(t)} = \begin{bmatrix}
		t & \frac{t}{2} \\
		\frac{t}{2} & 2t
	\end{bmatrix}$$ and proportion $\bp = \diag(0.4,0.6)$ versus the operator norm of $\sqrt{\bp} \frac{1}{\bD(t)} \sqrt{\bp }$ where $t \in \R^+$,  
\begin{figure}
	\begin{center}
		\includegraphics[width=13cm]{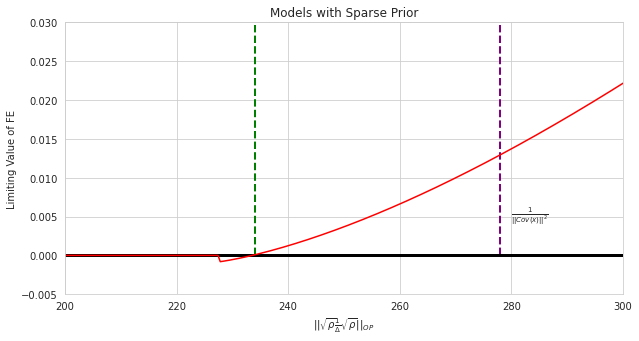}
		\includegraphics[width=13cm]{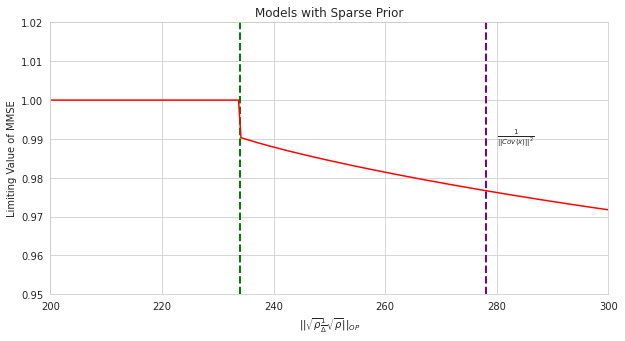}
		\caption{The free energy as a function of the noise. A case with a discontinuous phase transition (orange dashed line) separating the undetectable regime from one where detection of the signal is possible. We plots the free energy of two local maximizers (in blue and red), it is the larger one that provides the final result. The purple dashed line is a position of the threshold from part 2 of Lemma \ref{prop:recovery} that is not tight in this case.} 
		\label{fig:sparse_general}
	\end{center}
\end{figure}
In this scenario, there is a clear gap between the upper bound on the transition $\frac{1}{\|\Cov(x)\|^2}$ (denoted by the purple line) and the actual transition (denoted by the green line) when the free energy becomes positive. In this paper we do not focus on algorithms but in analogy with the homogeneous case \cite{lesieur2015mmse,lesieur2016phase} we anticipate that the transition at $\frac{1}{\|\Cov(x)\|^2}$ in fact has an algorithmic meaning of a threshold beyond which a corresponding message passing algorithm is able to find a vector correlated with the signal and below which it is not. The region in between these two threshold is then conjectured to be algorithmically hard. 
% -----------

	Next, we examine the behavior of the limiting free energy as the proportions of the two blocks vary in the same setting of the  sparse  Rademacher prior. We consider the case when $\bp = (p,1-p)$. We consider fixed
	\[
	\frac{1}{\bD(t)} = \begin{bmatrix}
		t & 0\\
		0 & 2t
	\end{bmatrix}.
	\]
	For each choice of $p \in \{0.1,0.2,0.3,0.4 \}$, we again plot in Figure \ref{fig:small_diagonal}  the limit of the free energy of the matrix $\bD(t)$ and proportion $\bp$ versus the operator norm of $ \sqrt{\bp} \frac{1}{\bD} \sqrt{\bp}$ where $t \in \R^+$. In this case, the phase transition is independent of the proportions $p$.
\begin{figure}
	\begin{center}
		\includegraphics[width=13cm]{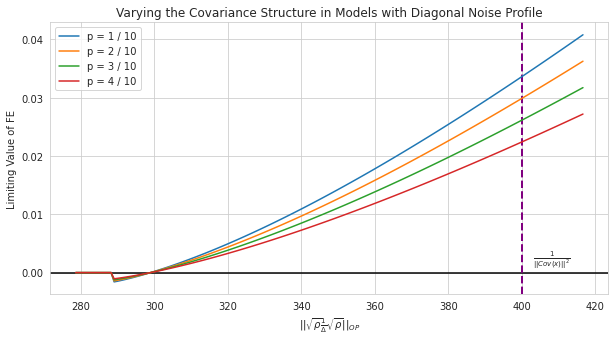}
		\caption{The free energy as a function of the noise parameter, for groups of different sizes $p$, and diagonal inhomogeneity $\bD$.}
		\label{fig:small_diagonal}
	\end{center}
\end{figure}

 If we now consider a non-diagonal inhomogeneity matrix 
	\[
	\frac{1}{\bD} = \begin{bmatrix}
		t & t\\
		t & 4t
	\end{bmatrix}
	\]
	the phase transition is dependent on the proportions $p$ as depicted in Figure \ref{fig:small_nondiagonal}.
	\begin{figure}
		\begin{center}
			\includegraphics[width=13cm]{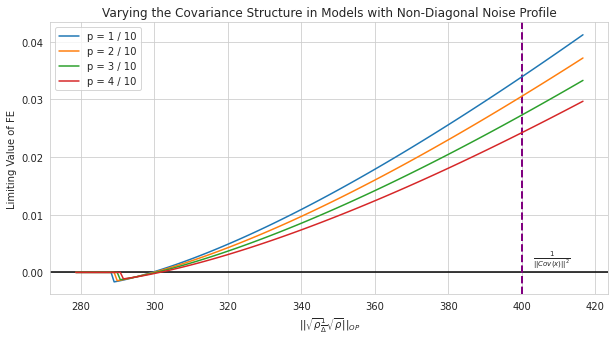}
			\caption{The free energy as a function of the noise parameter, for groups of different sizes $p$, and nondiagonal inhomogeneity $\bD$.}
			\label{fig:small_nondiagonal}
		\end{center}
	\end{figure}
	It remains to study the dependence of the limiting free energies on the structure of the noise more rigorously. In particular, it appears that phase transition at $ \frac{1}{\| \Cov(x) \|_{op}^2}$ can be improved for certain choices of sparse symmetric priors. However, this will likely be a difficult problem to solve, because it will depend on both the covariance structure $\bD$ and the proportions $\bp$.

\subsection{The BBP Transition of a Transformed Homogeneous Noise Model}

Recovery transitions can also be studied from the point of view of the BBP transition \cite{BBP} in random matrix theory. In the following, we compute the transition for a normalization of the spiked matrix that will move the noise profile onto the signal,
\[
\frac{1}{\sqrt{N}}\frac{1}{\bD^{\odot \frac{1}{2} }}\odot  Y^\Delta  = \frac{1}{\sqrt{N}}  W + \frac{1}{N}  \frac{1}{\bD^{\odot \frac{1}{2} }}  \odot  \bvx^0 (\bvx^0)^\trans.
\]
The normailzation by $\frac{1}{\sqrt{N}}$ was to ensure the eigenvalue distribution is supported almost surely on a compact set, and the normalization by $\frac{1}{\bD^{\odot \frac{1}{2} }} $ moved the noise profile onto the signal. Classical random matrix theory arguments can be used to compute when the largest eigenvalue separates from the bulk of the transformed matrix.

\begin{lem}[BBP Transition]\label{lem:BBPtrans}
	The largest eigenvalue of the matrix $\frac{1}{\sqrt{N}} Y \odot \frac{1}{\bD^{\odot \frac{1}{2} }}$ separates from the bulk when
	\[
	\bigg\| \sqrt{\bp} \frac{1}{\bD^{\odot \frac{1}{2}}} \sqrt{\bp} \bigg\|_{op} > \frac{1}{\| \Cov(x) \|_{op}}.
	\]
\end{lem}

When the noise profile $\bD$ is not identically equal to one, the region of noise described in Lemma~\eqref{lem:BBPtrans} where the top eigenvalue separates from the bulk is contained in the recovery regime Lemma~\ref{prop:recovery}. This means that a naive reduction of the inhomogeneous problem to a homogeneous problem does not yield the sharpest transition for these inhomogeneous noise models. This behavior is different from the classical homogeneous noise models.

\begin{prop}[Gap in Thresholds]\label{prop:gap}
	We have
	\[
	\bigg\| \sqrt{\bp} \frac{1}{\bD^{\odot \frac{1}{2}}} \sqrt{\bp} \bigg\|^2_{op} \leq \bigg\| \sqrt{\bp} \frac{1}{\bD} \sqrt{\bp} \bigg\|_{op},
	\]
	with equality holding if and only if $\bD = c$ for some constant $c$. In particular, when $\bD \neq c$
	\[
	\bigg\{ \bigg\| \sqrt{\bp} \frac{1}{\bD^{\odot \frac{1}{2}}} \sqrt{\bp} \bigg\|_{op} \geq \frac{1}{ \| \Cov(x) \|_{op} } \bigg\} \subset \bigg\{ \bigg\| \sqrt{\bp} \frac{1}{\bD} \sqrt{\bp} \bigg\|_{op} \geq \frac{1}{\| \Cov(x) \|_{op}^2} \bigg\}
	\]
	so the information theoretic detectability regime strictly contains the BBP transition.
\end{prop}

\subsection{Application to the Degree Corrected Stochastic Block Model}
We now apply our results to the degree corrected stochastic block models (DCSBM) introduced in \cite{Blockmodel}. 
%We first explain the application to the DCSBM with $\kappa$ symmetric equally sized groups. 

We consider a dense community detection problem with $\kappa \geq 1$ communities in a network of $N$ nodes. If the vertex $i$ belongs to the community $t_i \in \{1, \dots, \kappa\}$, then its membership can be encoded by $x_i = e_{t_i} \in \R^{\kappa}$ where $e_{t_i}$ is the $t_i$th basis element in $\R^\kappa$. In particular, for any two nodes $i$ and $j$, $x_i \cdot x_j = 1$ if they belong to the same community and $x_i \cdot x_j = 0$ otherwise. If we choose the prior $\pP_0$ to be uniform on all basis elements, then we are in the setting with equally sized groups. 

We now build a site dependent adjacency matrix that encodes different probability of attaching an edge depending on if the vertices are in the same group or not. In contrast to the standard stochastic block model, we consider a vector of site specific parameters $(\theta_1, \dots, \theta_N) \in (0,1)^N$ which controls the global probability of any vertex being attached to node $i$. For each $i,j$ we also define a $\kappa \times \kappa$ connectivity matrix
\[
\bC_{ij} =  \begin{bmatrix}
\theta_i \theta_j & \theta_i \theta_j &\dots & \theta_i \theta_j\\
\theta_i \theta_j&\ddots &\ddots &\vdots \\
\vdots &\ddots &\ddots & \theta_i \theta_j \\
\theta_i \theta_j &\dots &\theta_i \theta_j &  \theta_i \theta_j
\end{bmatrix} + \frac{\lambda}{\sqrt{N}} \begin{bmatrix}
1 &0 &\dots &0\\
0&\ddots &\ddots &\vdots \\
\vdots &\ddots &\ddots &0 \\
0 &\dots &0 &1
\end{bmatrix}.
\]
In particular, this implies that the probability to attach an edge between vertex $i$ and vertex $j$ is $\theta_i \theta_j$ if they are in different communities and if they are in the same community, then the probability is different by an additive factor of $\frac{\lambda}{\sqrt{N}}$. The precise output channel for the adjacency matrix of the entire graph is given by
\begin{equation}\label{defy}
\pP_\out(D_{ij} = 1 \given \bvx^0, \bm\theta) = \theta_i \theta_j + \frac{\lambda}{\sqrt{N}} \vx^0_i \cdot \vx_j^0.
\end{equation}

We will generalize the framework slightly and consider parameters $(\theta_i)_{i \leq N}$ with each $\theta_i$ sampled independently from some distribution $\pP_\theta$ on $(0,1) \subset \R^+$ (the precise interval doesn't matter because we can scale the $p_\out > 0$). As usual, we also generate a vector $\vx^0 \in \R^{\kappa}$ according to $\pP_0$. Finally, we generate the adjacency matrix $D_{ij}$ according to \eqref{defy}. 
The parameter $\lambda$ is the signal to noise ratio in these models. 

\begin{rem}
If $\pP_\theta(\theta = 1) = 1$, then this model reduces to the classical stochastic block model. 
\end{rem}

\subsubsection{The Limit of the Free Energy}

We now state the limit of the free energies in these models. Suppose that $\pP_\theta$ is supported on finitely many points $(\theta^{(1)}, \dots, \theta^{(n)}) \in (0,1)$ with probabilities $\rho_1, \dots, \rho_n$. We define $\bD = (\bD_{s,t})_{1\leq s,t \leq n}$ by
\begin{equation}\label{eq:DCSBM1}
\frac{1}{\Delta_{st}} = \bigg( \frac{\lambda^2}{ \theta^{(s)} \theta^{(t)} } +  \frac{\lambda^2}{ (1 - \theta^{(s)} \theta^{(t)}) } \bigg) 
\end{equation}
and consider the conditional probabilities
\begin{equation}\label{eq:DCSBM2}
d\pP(X|D, \Theta) = \frac{1}{Z_X(D,\Theta)}  \prod_{1 \leq i<j \leq N} e^{\ln \pP_\out(D_{ij} \given \bvx , \bm\theta)} \prod_{1 \leq i \leq N} d\pP_0(\vx_i)
\end{equation}
where
\begin{equation}\label{eq:DCSBM3}
\pP_\out(D_{ij} = 1 \given \bvx, \bm\theta) = \theta_i \theta_j + \frac{\lambda}{\sqrt{N}} \vx_i \cdot \vx_j \quad\text{and}\quad \pP_\out(D_{ij} = 0 \given \bvx^0, \bm\theta) = 1 - \theta_i \theta_j - \frac{\lambda}{\sqrt{N}} \vx_i \cdot \vx_j.
\end{equation}
We define the collection $\bQ = (\bQ_1, \dots, \bQ_n)$ of symmetric positive semidefinite $\kappa \times \kappa$ matrices and the functional
\begin{align}\label{eq:DCSBM4}
\phi(\bQ) &= - \sum_{s,t = 1}^n \frac{\rho_s \rho_t}{4 \Delta_{s,t}}\Tr( \bQ_s \bQ_t) +  \sum_{s =1}^n \rho_s \E_{\vz,\vx^0} \ln \bigg[ \int e^{( \tilde \bQ_s\vx^0 + \sqrt{\tilde \bQ_s} \vz )^\trans \vx - \frac{\vx^\trans \tilde \bQ_s \vx}{2} } \, d \pP_0(\vx) \bigg]
\end{align}
where for $s\in \{1,\ldots,n\}$, 
\[
\tilde \bQ_s = \sum_{t =1}^n \frac{1}{\Delta_{s,t}} \rho_t \bQ_t
\]
$\vx^0 \sim \pP_0$ and $\vz \sim N(0,\bI_\kappa)$ are independent. The limit of the free energy is given by maximizing this functional.
\begin{cor}[Free Energy for Discrete Degree Corrected Stochastic Block Models] \label{prop:FESBM}
If $\pP_0=\sum_{i=1}^n \rho_i\delta_{\theta^{(i)}}$ with $\theta^{(1)}, \dots, \theta^{(n)} \in (0,1)$ then the free energy of the degree corrected stochastic block model is
\[
\lim_{N \to \infty}  \frac{1}{N}  \E \left(\ln Z_X(D,\Theta)- \sum_{i<j}\ln  \pP_\out(D_{ij} \given \bvx=0 , \bm\theta)\right)
= \sup_{\bQ} \phi( \bQ).
\]
where the parameters of the model  are given in \eqref{eq:DCSBM1}, \eqref{eq:DCSBM2}, \eqref{eq:DCSBM3} and \eqref{eq:DCSBM4}.
\end{cor}

\begin{proof}
We essentially check that conditional probabilities in \eqref{eq:DCSBM3} satisfy the conditions for the inhomogeneous vector spin models required in Theorem \ref{thm:main}. Hypotheses \ref{hypbayes} and \ref{hypcompact} are clearly verified. We next check the two other assumptions. 
\\\\\textit{Hypothesis \ref{hypg} (Bounds on the derivatives):} The corresponding function $g_{ij}$ is the log likelihood expressed as a function of $D_{ij}$ and $ \frac{1}{\sqrt{N}} \vx^0_i \cdot \vx_j^0$,
\[
g_{ij}(1,w) = \ln\pP_\out(D_{ij} = 1 \given w, \bm\theta) =  \ln ( \theta_i \theta_j  + \lambda w)
\]
and
\[
g_{ij}(0,w) = \ln\pP_\out(D_{ij} = 0 \given w, \bm\theta) =  \ln (1 -  \theta_i \theta_j  - \lambda w ).
\]
Notice that our function $g$ is smooth in $w$, and well defined for $w$ such that $1 -  \theta_i \theta_j - \lambda w \in (0,1)$. Furthermore,
\begin{align}
	\partial_w g_{ij} (D,w) &= D \frac{\lambda}{ \theta_i \theta_j  + \lambda w }  - (1 - D) \frac{\lambda}{1 -  \theta_i \theta_j  - \lambda w }
	\\\partial^2_w g_{ij} (D,w) &=  D \frac{-\lambda^2}{ (\theta_i \theta_j  + \lambda w)^2 }  - (1 - D)\frac{\lambda^2}{( 1 -  \theta_i \theta_j  - \lambda w)^2 }
	\\\partial^3_w g_{ij} (D,w) &=  D \frac{2\lambda^3}{ (\theta_i \theta_j + \lambda w)^3 }  - (1 - D)\frac{2\lambda^3}{( 1 -  \theta_i \theta_j  - \lambda w)^3 }
\end{align}
which are all uniformly bounded provided that $0 < \theta_i \theta_j + \lambda w < 1$ for all $i$ and $j$ and $w$ in the domain. Since $w$ goes to zero, for $N$ large enough,  it is enough that the $\theta^{(i)}$ belong to $(0,1)$  independently of $N$. %These conditions will be satisfied if we restrict the support of $\pP_{0}$ and $\pP_\theta$ and take $\lambda$ sufficiently small.  
\\\\\textit{Hypothesis \ref{hypDelta} (Positive Definiteness of Noise Profile):} By definition, we have 
\begin{align}
	\frac{1}{\Delta_{ij}} = \E_{P_{ij}(D|w=0)} \bigg[ (\partial_w g_{ij}(D,0) )^2 \bigg] &= \bigg( \pP(D_{ij} = 1\given w_{ij}^0 = 0,\theta) \frac{\lambda^2}{ \theta^2_i \theta_j^2 } +  \pP(D_{ij} = 0\given w_{ij}^0 = 0,\theta) \frac{\lambda^2}{ (1 - \theta_i \theta_j)^2 } \bigg) \notag
	\\&=\bigg( \frac{\lambda^2}{ \theta_i \theta_j } +  \frac{\lambda^2}{ (1 - \theta_i \theta_j) } \bigg)  \label{eq:fishinfo0}
\end{align} 
because $\pP_{ij}(D_{ij} = 1 \given 0, \theta) = \theta_i \theta_j = 1 - \pP(D_{ij} = 0 \given 0,\theta)$.
Observe that $\Delta_{ij}$ is piecewise constant since the $\theta_i$'s are piecewise constant. We denote as well $\Delta$ the $n\times n$ symmetric defined in Hypothesis \ref{hypDelta}. 
We next show that the matrix $\frac{1}{\Delta^2}$ is non-negative. 
To see this, notice that
\begin{equation}\label{eq:fishinfo1}
	\bigg[ \frac{\lambda^2}{ \theta^{((s)} \theta^{(t)}}  \bigg]_{1 \leq s,t \leq n} = \lambda^2 \bigg[ \frac{1}{\theta^{(s)}} \cdot \frac{1}{\theta^{(t)}}  \bigg]_{1 \leq i,j \leq n}
\end{equation}
is positive semidefinite because it is a Gram matrix. Furthermore, for $|\theta^{(s)} \theta^{(t)}| < 1$, we have %https://www.wolframalpha.com/input/?i=sum_%28n%3D0%29%5E%E2%88%9E+x%5En+%281+%2B+n%29+y%5En
\[
\frac{\lambda^2}{ (1 - \theta^{(s)} \theta^{(t)}) } = \lambda^2 \sum_{k = 0}^\infty (\theta^{(s)})^k (\theta^{(t)})^k,
\]
by its Taylor expansion, so 
\begin{equation}\label{eq:fishinfo2}
	\bigg[ 	\frac{\lambda^2}{ (1 - \theta^{(s)} \theta^{(t)}) } \bigg]_{1 \leq s,t \leq n} =  \lambda^2 \sum_{k = 0}^\infty  \Theta^{\odot k},
\end{equation}
where $\Theta = [\theta^{(s)} \theta^{(t)}]_{1 \leq s,t \leq n}$. The matrix $\Theta$ is positive semidefinite because it is a Gram matrix, and the Schur product theorem implies that the Hadamard powers are also positive semidefinite, so is also \eqref{eq:fishinfo2} positive semidefinite. This concludes the proof.\end{proof}

We can also apply Theorem \ref{thm:main2} to this setting. To this end we assume that
\begin{hyp}\label{hypthetac}
Assume that $\theta_{\lfloor Ns\rfloor}$ converges uniformly to a measurable function $(\theta^{(s)})_{s\in [0,1]}$ with values in a compact subset $[a,b]$ of $(0,1)$.

\end{hyp}
Let $\bQ(\theta): [a,b] \to \mathbb{S}_\kappa^+$ be a measurable matrix valued function. Let $\theta$ and $\theta'$ be independent samples from $\tilde \pP_\theta$. We define
\begin{equation}\label{eq:DCSBM7}
\frac{1}{\Delta(\theta,\theta')} = \bigg( \frac{\lambda^2}{ \theta \theta' } +  \frac{\lambda^2}{ (1 - \theta \theta' ) } \bigg) 
\end{equation}
We have
\begin{equation}\label{eq:DCSBM8}
\phi(\bQ) = - \int_0^1 \frac{\Tr(\bQ(\theta^{(s)}), \bQ(\theta^{(t)}) ) }{4 \Delta(\theta^{(s)},\theta^{(t)})} dsdt + \int_0^1 ds \ln \bigg[ \int e^{( \tilde \bQ(\theta^{(s)})\vx^0 + \sqrt{\tilde \bQ(\theta^{(s)})} \vz )^\trans \vx - \frac{\vx^\trans \tilde \bQ(\theta^{(s)}) \vx}{2} } \, d \pP_0(\vx)
\end{equation}
where the expected value is with respect to $\theta, \theta', \vx^0, \vz$ and
\begin{equation}\label{eq:DCSBM9}
\tilde\bQ(\theta) = \int_0^1 \frac{1}{\Delta(\theta, \theta^{(s)}) } \bQ(\theta^{(s)}) ds  \,.
\end{equation}
It follows by approximation that Corollary~\ref{prop:FESBM} holds in the limit. 

\begin{cor}[Free Energy for Degree Corrected Stochastic Block Models]
Under  Hypothesis \ref{hypthetac}, the free energy of the degree corrected stochastic block model is given by 
\[
\lim_{N \to \infty}  \frac{1}{N} \E \ln\left( Z_X(Y,\tilde \Theta)- \sum_{i<j}\ln  \pP_\out(Y_{ij} \given \bvx=0 , \bm\theta)\right) = \sup_{\bQ}  \phi( \bQ).
\]
where the parameters of the model  are given in \eqref{eq:DCSBM2} \eqref{eq:DCSBM7} and \eqref{eq:DCSBM8} and \eqref{eq:DCSBM9}.
\end{cor}

\subsubsection{Rademacher Prior}

We now consider the case when $\pP_0$ takes values in $\{-1,0,1\}$ such that for $p \in [0,\frac{1}{2}]$
\begin{equation}\label{eq:sparse}
\pP_0(-1) = \pP_0(1) = p \qquad \pP_0(0) = 1 - 2p.
\end{equation}
If $x \sim \pP_0$, then $\Var(x) = 2p$. We define the $n \times n$ matrices
\[
\frac{1}{\Delta} = \Big( \frac{1}{\Delta_{s,t} } \Big)_{s,t \leq n} \quad \rho = \diag( \rho_1, \dots, \rho_n ).
\]
In the degree corrected stochastic block model, the parameters were encoded by $\pP_\theta$ supported on finitely many points $(\theta^{(1)}, \dots, \theta^{(n)}) \in (0,1)$ with probabilities $\rho_1, \dots, \rho_n$,
\begin{equation}\label{eq:dcsbmvar}
\frac{1}{\Delta_{st}} = \bigg( \frac{\lambda^2}{ \theta^{(s)} \theta^{(t)} } +  \frac{\lambda^2}{ (1 - \theta^{(s)} \theta^{(t)}) } \bigg) 
\end{equation}
and for \iid $\theta_i \sim \pP_\theta$
\begin{equation}
\pP_\out(Y_{ij} = 1 \given \bvx, \bm\theta) = \theta_i \theta_j + \frac{\lambda}{\sqrt{N}} \vx_i \cdot \vx_j \quad\text{and}\quad \pP_\out(Y_{ij} = 0 \given \bvx^0, \bm\theta) = 1 - \theta_i \theta_j - \frac{\lambda}{\sqrt{N}} \vx_i \cdot \vx_j.
\end{equation}
Recall that the replica symmetric functional is
\begin{equation}
\phi(\bQ) = - \sum_{s,t =1}^n \frac{\rho_s \rho_t}{4 \Delta_{s,t}}\Tr( \bQ_s \bQ_t) +  \sum_{s =1}^n \rho_s \E_{\vz,\vx^0} \ln \bigg[ \int \exp \bigg( \bigg( \tilde \bQ^s\vx^0 + \sqrt{\tilde \bQ^s} \vz \bigg)^\trans \vx - \frac{\vx^\trans \tilde \bQ^s \vx}{2}  \bigg) \, d \pP^\kappa_0(\vx) \bigg]
\end{equation}
where 
\[
\tilde \bQ^s = \sum_{t =1}^n \frac{1}{\Delta_{s,t}} \rho_t \bQ_t.
\]
The replica free energy functional is identical to the one computed by a spiked matrix with covariance profile given by \eqref{eq:dcsbmvar}. By Lemma~\ref{prop:gentransition1} and Lemma~\ref{prop:gentransition2}, we can conclude that
\begin{enumerate}
\item $\phi$ has a unique maximum at a sequence of matrices with all entries equal to $0$ when
\[
\bigg\| \sqrt{\rho} \frac{1}{\bm\Delta} \sqrt{\rho} \bigg\|_{op} < \frac{1}{9}
\]
\item $\phi$ has a maximum at a sequence of matrices where at least one matrix has a non-zero entry when
\[
\bigg\| \sqrt{\rho} \frac{1}{\bm\Delta} \sqrt{\rho} \bigg\|_{op} > \frac{1}{4p^2}.
\]
\end{enumerate}
Since $p \leq \frac{1}{2}$, this transition is clearly not tight. 

We can numerically solve the free energies to see how the phase transitions depends on the inhomogeneity. In the $n = 2$ case, a perfectly homogeneous model is when all entries of $\frac{1}{\bD}$ are identical the phase transitions agree with known results \cite{lesieur2015mmse}. We analyze the effects of changing the homogeneity on the locations of the maximizers. To do this, we can fix $\theta_1$ and $\theta_2$ and vary the parameter $\lambda$. If we consider the sparse case when $p = 0.025$ in \eqref{eq:sparse}, we plot in Figure \ref{fig:DCSBM_sparse} the limit of the free energy versus the operator norm of $\| \sqrt{\rho} \frac{1}{\bm\Delta} \sqrt{\rho} \|_{op}$.
\begin{figure}
\begin{center}
	\includegraphics[width=15cm]{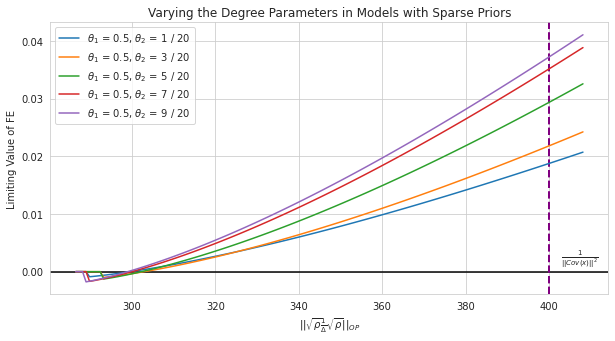}
	\caption{Free energy as the function of noise for the DCSBM with a small size of one of the groups. We see a discontinuous phase transition (purple line).}
	\label{fig:DCSBM_sparse}
\end{center}
\end{figure}
In contrast, we consider the dense case when $p = 0.25$ in \eqref{eq:sparse} and plot again limit of the free energy versus the operator norm of $\| \sqrt{\rho} \frac{1}{\bm\Delta} \sqrt{\rho} \|_{op}$ in Figure \ref{fig:DCSBM_dense}.
The second bound appears to be tight when the average degree encoded by $p$ is high.
\begin{figure}
\begin{center}
	\includegraphics[width=15cm]{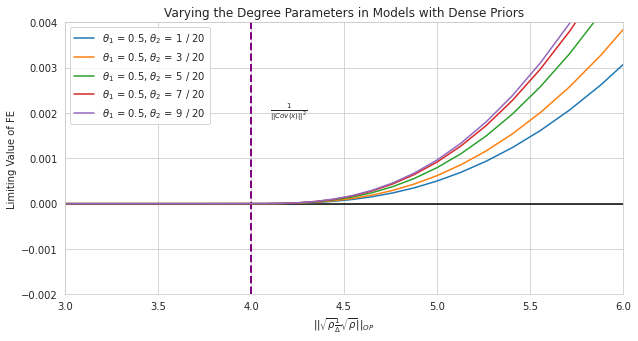}
	\caption{Free energy as the function of noise for the DCSBM with comparable sizes of one of the groups. We see a continuous phase transition (purple line).}
	\label{fig:DCSBM_dense}
\end{center}
\end{figure}

\subsection{Outline of the Paper.} We first start with a proof of the universality result explained in Lemma~\ref{prop:universality}. In Section~\ref{sec:univ}, we move onto proving the limit of the free energy in Theorem~\ref{thm:main} by proving the lower bound using interpolation in Section~\ref{sec:lwbd}. We prove the matching upper bound using concentration of the overlaps and cavity computations in Section~\ref{sec:upbd}. In Section~\ref{sec:phase}, we study the maximizers of the free energy, and explore its consequences on the information theoretic thresholds.

\section{Universality in Inhomogenous Vector Spin Models} \label{sec:univ}

In this section, we proceed to several approximations to arrive to Theorem \ref{prop:universality}. We start by showing only the second order Taylor expansion of $g$ matters in the computation of the free energy (recall the  definition of $F_N(g)$ in \eqref{eq:FEgrowingrank}). 

\begin{lem}[Independence of Third Order Expansions] \label{lem:univ1}
If $\sup_{i,j} \|\partial_w^3 g_{ij}\|_\infty < \infty$, then 
\[
F_N(g) = F_N(\tilde g) + O\Big( \frac{\kappa^3}{\sqrt{N}}\Big) 
\]
where
\[
\tilde g_{ij}(D,w) = g_{ij}( D_{ij},0) + \partial_w g_{ij}( D,0)  w + \frac{1}{2} \partial_w^2 g_{ij}( D,  0) w^2.
\]
\end{lem}

\begin{proof}
By Taylor's theorem, for all $i,j$,
$$g_{ij}(  D_{ij},w_{ij})-g_{ij}( D_{ij},0)=  \partial_w g_{ij}( D_{ij},0) w_{ij} + \frac{1}{2} \partial_w^2 g_{ij}(  D_{ij},0)  w_{ij}^2  +\frac{w_{ij}^3}{3!} \partial_w^3g_{ij}(D_{ij},\theta_{ij}w_{ij})$$
for some $\theta_{ij}\in [0,1]$. Since our hypothesis implies that $|w_{ij}|_\infty \le C^2 \kappa/\sqrt{N}$, our assumption that $\sup_{i,j} \|\partial_w^3 g_{ij}\|_\infty < \infty$ implies
\[
\bigg\|\frac{1}{N} \sum_{i < j} \frac{w_{ij}^3}{3!}g^{(3)}_w(Y_{ij},\theta_{ij}w_{ij} ) \bigg\| =  O\Big( \frac{\kappa^3}{N^{1/2}}\Big) 
\]
from which the result follows.
\end{proof}

The next step in the reduction is to prove that the coefficient of the second derivative term can be replaced by its conditional average. We let $B$ be a  sub 
$\sigma$-algebra of $\mathbb R^{N^2}\times (\mathbb R^\kappa)^N$ and denote by $\mathbb P_B$ (resp. $\E_Y[.|B]$) its associated conditional probability under $\pP_Y$
(resp. its conditional expectation).  Later on we will simply take $B$ to be the $\sigma$-algebra generated by $\bvx^0$, and therefore $\pP_B$ will just be the distribution of $Y$ knowing the $\bvx^0$. 

\begin{lem}[Concentration of Second Order Terms] \label{lem:univ2} 
Assume $\sup_{i,j} \|\partial_w^2 g_{ij}(\cdot ,0)\|_\infty < \infty$ and   $\pP_0$ compactly supported.  
Let $B$ be a $\sigma$ algebra  such that
the $Y_{ij}$ are independent conditionally to $B$.
Then 
$$F_N(\tilde g)=  F_N(\bar g) +O\Big(\frac{\kappa^2}{\sqrt{N}}\Big)$$
with \[
\bar  
g_{ij}(D,w) = g_{ij}( D_{ij},0) + \partial_w g_{ij}( D,0)  w + \frac{1}{2} \E_D[ \partial_w^2 g_{ij}(D,0)\given B] w^2 \,.
\] 
%where $\E_Y[ \cdot \given B]$ is any conditional expectation so that $\E_Y ( \partial_w^2 g_{ij}(Y,0)- \E_Y[ \partial_w^2 g_{ij}(Y,0)\given B] )=0$.
\end{lem}
\begin{proof}
Notice that
$$F_N(\tilde g)-F_N(\bar g)=\E_D\frac{1}{N}\ln \Big\langle e^{\frac{1}{2\sqrt N}\sum_{i< j} \frac{1}{\sqrt{N}}(\partial_w^2 g_{ij}(D_{ij},0)-\E_D[ \partial_w^2 g_{ij}(D_{ij},0)|B])(\vx_i^\trans \vx_j)^2)}\Big\rangle$$
where $$\langle f\rangle= \frac{\int f e^{\sum_{i< j} \bar g_{ij}(D_{ij},w_{ij})} d\pP_0^{\otimes N}(\vx)}{ \int e^{ \sum_{i< j} \bar g_{ij}(D_{ij},w_{ij})} d\pP_0^{\otimes N}(\vx)}\,.$$
Let $Z$ be the $N\times N$ symmetric  matrix with entries so that $Z_{ii}=0$ and for $i\neq j$
$$Z_{ij}=\frac{1}{4\sqrt{N}}(\partial_w^2 g_{ij}(\Sigma_{ij}^{-1} D_{ij},0)-\E_D[ \partial_w^2 g_{ij}(\Sigma_{ij}^{-1} D_{ij},0)|B])\,.$$
As a consequence
$$\sum_{i< j} \frac{1}{2\sqrt{N}}(g_{ww}(D_{ij},0)-\E_D[ g_{ww}(D_{ij},0)|B])(\vx_i^\trans \vx_j)^2=\Tr\left (Z (\vx^\trans \vx)^2\right).$$
$Z$ is a random  symmetric matrix under $\pP_B$ which  has centered independent entries with covariance bounded by $C/N$ and where $(\vx^\trans \vx)^2$ is the matrix with entries $(\vx_i^\trans \vx_j)^2$.
Because $\sqrt{N} Z$  has  bounded entries, we can use concentration inequalities due to Talagrand (see \cite[Theorem 2.3.5]{AGZ}  and \cite[Lemma 5.6]{HuGu}) to see that there exists some finite $L_0$ such that, uniformly, 
\begin{equation}\label{conc2}\PP_B\left(\|Z\|_\infty\ge L\right)\le e^{-N(L-L_0)}\,.\end{equation}
On $\{\|Z\|_\infty\le L\}$, we have the bound
$$\left|\Tr\bigl(Z (\vx^\trans \vx)^2\big)\right|=\left|\sum_{k,k'=1}^{\kappa} \sum_{i,j} Z_{ij} x_i(k)x_i(k') x_j(k)x_j(k')\right| \le L \sum_{k,k'=1}^{\kappa} \sum_{i=1}^N (x_i(k)x_i(k'))^2\le C L \kappa^2 N $$
for some finite constant $C$ depending only on the bound on the support of $\pP_0$.
Hence, we deduce
$$F_N(\tilde g)-F_N(\bar g)=\E_D 1_{\|Z\||\ge L} 
\frac{1}{N}\ln \langle e^{\frac{1}{\sqrt N}\sum_{i\le j} \frac{1}{2\sqrt{N}}(\partial_w^2 g_{ij}(Y_{ij},0)-\E_D[ \partial_w^2 g_{ij}(D_{ij},0)|B])(\vx_i^\trans \vx_j)^2)}\rangle + O\Big( \frac{\kappa^2}{\sqrt{N}}\Big).$$
Moreover as we assumed that  $\partial_w^2 g_{ij}(D,0)$ is  uniformly bounded, the term in the above expectation is uniformly bounded and therefore the first term is going to zero exponentially fast by \eqref{conc2} for $L$ large enough. 
\end{proof}

We finally compare our free energy to those of a spin glass model. It will depend on three matrices with entries:
\begin{equation}\label{eq:defM}
\gamma_{ij}= \E_D[ \partial_{w}^2 g_{ij}(D_{ij},0)\given \bvx^0], \quad
\mu_{ij} =\E_D [\partial_w g_{ij}(D_{ij}, 0) \given \bvx^0] ,\quad \sigma^2_{ij} = \E_D[ (\partial_w g_{ij}(D_{ij}, 0)  -\mu_{ij})^2 \given \bvx^0].
\end{equation}
By universality, we will prove that we can replace $\partial_w g_{ij}( D_{ij},0)$ by $\sigma_{ij} W_{ij} + \mu_{ij}$ where $W_{ij}$ are \iid standard Gaussian variables (under the assumption that $\sqrt{N} \mu_{ij} = O(1)$).  Let
$$F_N(\sigma,\mu,\gamma)=\E_{W,x^0}\left[\frac{1}{N}\ln\E_{x}\left[ \exp( H_N(\bvx))\right]\right]$$
with 
\begin{align}
H_N(\bvx) &= \bigg( \sqrt{\frac{1}{N}} \sum_{i<j} \sigma_{ij} W_{ij} \vx_i^\trans \vx_j + \mu_{ij} \vx_i^\trans \vx_j \bigg) + \frac{1}{2N} \sum_{i<j} \gamma_{ij} ( \vx_i^\trans  \vx_j )^2 \notag
\\&= \sum_{k = 1}^{\kappa} \bigg( \sqrt{\frac{1}{N}} \sum_{i<j} \sigma_{ij} W_{ij} \vx_i(k) \vx_j(k) +  \mu_{ij} \vx_i(k) \vx_j(k) \bigg) + \frac{1}{2N} \sum_{k,\ell = 1}^{\kappa} \sum_{i<j} \gamma_{ij} \vx_i(k)  \vx_j(k) \vx_i(\ell)  \vx_j(\ell) \label{eq:genHamil}.
\end{align}
Then we shall prove that :
\begin{lem}[Universality in Disorder] \label{lem:univ3}
Assume that  
\[
\sup_{i,j} \|\mu_{ij} \|_\infty= O(N^{-1/2}),~\sup_{ij} \|\sigma_{ij}^2 \|_\infty< \infty,~\sup_{i,j}  \bigg\| \frac{ \E_{D}[ |\partial_w g_{ij}( D_{ij},0) - \mu_{ij} |^3|\bvx^0] }{ \sigma_{ij}^3 } \bigg\|_\infty < \infty
\]
%\begin{itemize}
%	\item $\sup_{i,j} \|\mu_{ij} \|_\infty= O(N^{-1/2})$
%	\item $\sup_{ij} \|\sigma_{ij}^2 \|_\infty< \infty$
%	\item  $\sup_{i,j}  \bigg\| \frac{ \E_{Y}[ |\partial_w g_{ij}( Y_{ij},0) - \mu_{ij} |^3|\bvx^0] }{ \sigma_{ij}^3 } \bigg\|_\infty < \infty$
%\end{itemize}
then
$$F_N(\bar g)=F_N(\sigma,\mu,\gamma)+O\Big( \frac{ \kappa^3}{N^{1/2}} \Big) 
$$
where $F_N(\sigma,\mu,\gamma)$ is the free energy with respect to the Hamiltonian defined in \eqref{eq:genHamil}.
\end{lem}

The proof follows from the following  approximate integration by  parts lemma \cite[Lemma 3,7]{Pbook}.
\begin{lem}\label{lem:approxIBP}
Suppose $x$ is a random variable that satisfies $\E x = 0$, $\E |x|^3 < \infty$. If $f: \R \to \R$ is twice continuously differentiable and $\|f''\|_\infty < \infty$, then
\[
| \E x f(x) - \E x^2 \E f'(x)| \leq \frac{3}{2} \| f''\|_\infty \E |x|^3.
\]
\end{lem}

\begin{proof} 
We follow the proof of Carmona--Hu \cite{UnivCARMONAHU} presented in \cite[Theorem 3.9]{Pbook}. To compare the free energies $F_N(\bar g)$ and $F_N(\sigma, \mu, \kappa)$ we use an interpolation argument. Conditionally on $\vx^0$, consider the interpolating Hamiltonian defined for $t \in [0,1]$ by
\begin{align*}
	H_{N}(\bvx,t) &= \frac{1}{\sqrt{N} } \sum_{i < j} \Big( \sqrt{t} ( \partial_{ij}g(D_{ij},0) - \mu_{ij}) + \sqrt{1 - t} \sigma_{ij} W_{ij} \Big) \vx_i^\trans \vx_j + \frac{1}{\sqrt{N}} \sum_{i < j} \mu_{ij}  \vx_i^\trans \vx_j  + \frac{1}{2N} \sum_{i < j} \gamma_{ij}  (\vx_i^\trans \vx_j)^2
	\\&= \frac{1}{\sqrt{N} } \sum_{i < j} \sigma_{ij} \Big( \sqrt{t} \tilde W_{ij}+ \sqrt{1 - t} W_{ij} \Big)  \vx_i^\trans \vx_j + \frac{1}{\sqrt{N}} \sum_{i < j} \mu_{ij}  \vx_i^\trans \vx_j  + \frac{1}{2N} \sum_{i < j} \gamma_{ij}  (\vx_i^\trans \vx_j)^2
\end{align*}
where we defined $\tilde W_{ij} = \sigma_{ij}^{-1} ( \partial_w g_{ij}( D_{ij},0) - \mu_{ij})$ to simplify notation. Notice that 
\[
\E_{D}[ \tilde W_{ij}^2\given \bvx^0] = \sigma_{ij}^{-2} \E_{D}[ ( \partial_w g_{ij}( D_{ij}, 0) - \mu_{ij})^2|\bvx^0] = 1
\]
and
\[
\E_{D} [\tilde W_{ij} \given \bvx^0] = \sigma_{ij}^{-1} \E_{D} [ ( \partial_wg_{ij}(D_{ij}, 0) - \mu_{ij} )  \given \bvx^0] = 0
\]
so both $W$ and $\tilde W$ have mean zero and variance 1. Also $\E_D |\tilde W_{ij}^3| \leq \frac{1}{\sigma_{ij}^3} (\| \partial_w g \|_\infty)^3$ is uniformly bounded.

We define the interpolating free energy 
\[
\phi(t) = \frac{1}{N} \E_{D}[ \ln \E_{X} \exp(  H_N(\bvx,t) )|\bvx^0], \langle f\rangle_t=\frac{\int f(\bvx)  \exp(  H_N(\bvx,t) )d\pP_0(\bvx)}{\int \exp(  H_N(\bvx,t) )d\pP_0(\bvx)}, \,
\]
and  notice that 
\[
\phi'(t) =\E_Y\left[ \frac{1}{2 \sqrt{t} N^{3/2}} \sum_{i < j} \sigma_{ij}\tilde W_{ij} \langle \vx_i^\trans \vx_j  \rangle_t - \frac{1}{2 \sqrt{1-t} N^{3/2}} \sum_{i < j} \sigma_{ij} W_{ij} \langle \vx_i^\trans \vx_j  \rangle_t \given[\bigg]\bvx^0\right].
\]
Let $f(\tilde W_{ij}) = \langle \vx_i^\trans \vx_j  \rangle_t$ 
(the dependence on $\tilde W$ is in the numerator and denominator in the Gibbs measure). We find that
\[
\frac{\partial f}{ \partial \tilde W_{ij} } = \frac{ \sqrt{t} \sigma_{ij}}{ \sqrt{N} } \bigg( \langle (\vx^{1}_i \cdot \vx^1_j)^2 \rangle_t - \langle ( \vx^{1}_i \cdot \vx^1_j )(\vx^{2}_i \cdot \vx^2_j) \rangle_t \bigg)
\]
and
\[
\frac{\partial^2 f}{ \partial \tilde W^2_{ij} } =  \frac{t \sigma^2_{ij}}{N}   \bigg(\langle (\vx^{1}_i \cdot \vx^1_j)^3 \rangle_t -2 \langle (\vx^{1}_i \cdot \vx^1_j )^2(\vx^{2}_i \cdot \vx^2_j) \rangle_t - \langle (\vx^{1}_i \cdot \vx^1_j)(\vx^{2}_i \cdot \vx^2_j)^2 \rangle_t + 2\langle( \vx^{1}_i \cdot \vx^1_j )(\vx^{2}_i \cdot \vx^2_j)(\vx^{3}_i \cdot \vx^3_j ) \rangle_t \bigg)\,.
\]
Therefore the second derivative is bounded by 
\[
\|\frac{\partial^2 f}{ \partial \tilde W^2_{ij} }\|\le
\frac{\sup_{ij} \|\E_{D}[ |\partial_w g_{ij}(D_{ij},0) - \mu_{ij} |^2 |\bvx^0]\|_\infty 6 C^6 \kappa^3}{N}
\]
where $C$ is such that   $\vx\in [-C,C]^\kappa$ almost surely. Applying the approximate integration by parts lemma to $\tilde W$ stated in Lemma~\ref{lem:approxIBP} applied conditionally on $\bvx^0$ implies
\begin{align*}
	&\bigg| \E_D[\frac{1}{2 \sqrt{t} N^{3/2}} \sigma_{ij}\tilde W_{ij} \langle \vx_i^\trans \vx_j  \rangle_t\given \bvx^0] - \frac{\sigma_{ij}^2}{2 N^2} \bigg( \E_D[ \langle (\vx^{1}_i \cdot \vx^1_j)^2 \rangle_t\given \bvx^0] - \E [\langle( \vx^{1}_i \cdot \vx^1_j )(\vx^{2}_i \cdot \vx^2_j) \rangle_t\given\bvx^0] \bigg) \bigg| 
	\\&\le  \frac{\sup_{ij} \|\E_{D}[ |\partial_w g_{ij}(D_{ij},0) - \mu_{ij} |^2 |\bvx^0]\|_\infty 6 C^6 \kappa^3}{N} \cdot \frac{3 \sup_{i,j} \E [ |\tilde W_{i,j}|^3 \given \bvx^0 ] }{4N^{3/2}} = O\Big( \frac{ \kappa^3}{N^{5/2}} \Big)
\end{align*}
by our assumption on the uniform bounds on the conditional expectation of $\tilde W$. The classical integration by parts lemma for Gaussian variables implies
\[
\E_W \bigg[\frac{1}{2 \sqrt{1-t} N^{3/2}} \sigma_{ij} W_{ij} \langle \vx_i^\trans \vx_j  \rangle_t \bigg]=\frac{ \sigma_{ij}^2}{2 N^2}\E_W \bigg[ \bigg(\langle (\vx^{1}_i \cdot \vx^1_j)^2 \rangle_t - \langle( \vx^{1}_i \cdot \vx^1_j )(\vx^{2}_i \cdot \vx^2_j) \rangle_t \bigg) \bigg].
\]
Summing over $i < j$ gives us the bound
\[
|\phi'(t)| \leq O\Big( \frac{ \kappa^3}{N^{1/2}} \Big)
\]
so that
\[
|\phi(1) - \phi(0)| = |  F_N(\bar g)-F_N(\sigma,\mu,\kappa) | \leq O\Big( \frac{ \kappa^3}{N^{1/2}} \Big) .
\]
\end{proof}

\begin{rem}
We use the notation $\vx^1_i, \vx^2_i, \dots $ to denote the replicas. They are independent copies of $\vx$ under the corresponding Gibbs measures.
\end{rem}

We can now further reduce our problem to a Gaussian estimation problem under the Bayes optimal assumption of Hypothesis \ref{hypbayes}. Indeed, recalling the definition of the Fisher score matrix \eqref{eq:fisherscore}, it results in special relation between $\sigma,\mu$ and $\gamma$ which state as follows.

\begin{lem}\label{lembayes}Assume Hypotheses \ref{hypbayes} and \ref{hypcompact}. Then for all $i,j\in [N]$ the terms defined in \eqref{eq:defM} satisfy
$$ \mu_{ij} =\frac{\vx_i^0 \cdot \vx_j^0}{\Delta_{ij} \sqrt{N}} + O(\kappa^2 N^{-1}) , \quad \sigma_{ij}^2 =  -\gamma_{ij} + O(\kappa N^{-1/2}), \quad \gamma_{ij}=- \frac{1}{\Delta_{ij}} + O(\kappa N^{-1/2})$$
\end{lem}

\begin{proof} 
The fact that $g_{ij}(D,w)$ is a log-likelihood is very important in this section, and the results in this final decomposition do not apply for general $g$. %Since we know all prior distributions, we can compute the exact posterior probability
%\[
Firstly, using the fact that $\pP_{ij}(\cdot|w)$ is a probability measure for all $w$ and $i < j$,
\begin{equation}\label{eq:out_probability}
	\int \pP_{ij} (y|w) \, dy  = \int e^{ g_{ij}(y,w)} dy = 1.
\end{equation}
By differentiating \eqref{eq:out_probability}, it follows that for all $w$, 
\begin{equation}\label{eq:center}
	\E_{\pP_{ij}(D|w)} \partial_w g_{ij}(D,w) =  \partial_w \int  e^{ g_{ij}(y,w)} \, dy  = 0
\end{equation}
and
\begin{equation}\label{eq:center2}
	\E_{\pP_{ij}(D|w)} \Big( (\partial_w g_{ij}(D,w))^2 + \partial_w^2 g_{ij}(D,w) \Big) = \partial_w^2 \int  e^{ g_{ij}(y,w)} \, dy  = 0.
\end{equation}
The conditional laws of $D_{ij}$ given the signal $\bvx^0$ are independent, so we have the following reductions of the parameters $\mu_{ij}$ and $\sigma_{ij}$. By \eqref{eq:center} we get 
\begin{eqnarray*}
	\mu_{ij}&=&\int \partial_w g_{ij}(y,0) e^{g_{ij}(y,w^0_{ij})}dy\\
	&=&\int (\partial_w  g_{ij}(y,0) -\partial_w  g_{ij}(y,w^0_{ij}) )e^{g_{ij}(y,w^0_{ij})}dy\\
	%&=& w_{ij}^0\int \frac{(\partial_w  g_{ij}(y,0) -\partial_w  g_{ij}(y,w^0_{ij}) )}{w_{ij}^0}( e^{g_{ij}(y,0)}  + O(w_{ij}^0) )dy\\
	&=&-w^0_{ij} \int \partial_w^2 g_{ij}(y,0) e^{g_{ij}(y,0)}dy+ O((w^0_{ij})^2)\\
	&=&-\Big( \frac{\vx_i^0 \cdot \vx_j^0}{\sqrt{N}} \Big) \E_{\pP_{ij}(D|w = 0)} \partial_w^2 g_{ij}(D,0) + O(\kappa^2N^{-1})%\\
	%&=&\Big( \frac{\vx_i^0 \cdot \vx_j^0}{\sqrt{N}} \Big) \E_{\pP_\out(Y|w = 0)} g_{w^2}(Y,0) + O(N^{-1})
\end{eqnarray*}
where we applied Taylor's theorem in the fourth equality. Applying \eqref{eq:center2} one more time implies that
\[
\mu_{ij} =  \frac{\vx_i^0 \cdot \vx_j^0}{\sqrt{N}} \E_{\pP_{ij}(D|w = 0)} (\partial_w g_{ij}(D,0))^2  + O(\kappa^2 N^{-1}) =  \frac{\vx_i^0 \cdot \vx_j^0}{\Delta_{ij} \sqrt{N}} + O(\kappa^2 N^{-1}) 
\]
where we recall that \eqref{eq:fisherscore}
\[
\frac{1}{\Delta_{ij}} = \E_{P_{ij}(D|w=0)} (\partial_w g_{ij}(D,0) )^2.
\]
Similarly, \eqref{eq:center2} implies 
\begin{eqnarray*}
	\sigma_{ij}^2&=& \int (\partial_w g_{ij}(y,0))^2 e^{g(y,w^0_{ij}) }dy -\mu_{ij}^2\\
	&=& \int(\partial_w g_{ij}(y,0))^2 -(\partial_w g_{ij}(y,w_{ij}^0))^2 )e^{g_{ij}(y,w_{ij}^0)} dy -\int \partial_w^2 g_{ij}(y,w_{ij}^0) e^{g_{ij}(y,w_{ij}^0)} dy  -\mu_{ij}^2\\
	&=&-\gamma_{ij}  +O(w_{ij})= -\gamma_{ij} + O(\kappa N^{-1/2}).
	\\
	%&=& -2 \Big( \frac{\vx_i^0 \cdot \vx_j^0}{\sqrt{N}} \Big) \E_{\pP_\out(Y|w = 0)} \partial_w g_{ij}(Y,0) \partial_w^2 g_{ij} (Y,0) -\gamma_{ij}  +O(\kappa^2 N^{-1}).
\end{eqnarray*}
We can do this trick one more time and apply \eqref{eq:center2} to see that 
\begin{eqnarray*}
	\gamma_{ij}&=& \int \partial_w^2 g_{ij}(y,0) e^{g_ij(y,w^0_{ij}) }dy \\
	%&=& \int ( \partial_w^2 g_{ij}(y,0) - \partial_w^2 g_{ij}(y,w^0_{ij}) ) e^{g_{ij}(y,w^0_{ij}) }dy - \int (\partial_w g_{ij}(y,w) )^2 e^{g_{ij}(y,w^0_{ij}) }dy\\
	&=& -w_{ij} \int \partial_w^3 g_{ij}(y,0) e^{g_{ij}(y,w^0_{ij}) }dy - \int ( (\partial_w g_{w}(y,w))^2 - (\partial_w g_{ij}(y,0) ) ^2 ) e^{g_{ij}(y,0) }dy - \frac{1}{\Delta_{ij}} + O(\kappa N^{-1/2})\\
	&=& - \frac{1}{\Delta_{ij}} + O(\kappa N^{-1/2}).
\end{eqnarray*}
\end{proof}
With this in mind, an interpolation argument and Gaussian integration by parts will prove that the Hamiltonian associated with the free energy $F(\sigma,\mu,\gamma)$ in the Bayes optimal case
\[
H_{N,\sigma,\mu,\gamma}(\bvx) = \sum_{i<j} \frac{\sigma_{ij} W_{ij}}{\sqrt{N}} (\vx_i \cdot \vx_j) + \frac{\mu_{ij}}{\sqrt{N}} (\vx_i \cdot \vx_j) + \frac{1}{2N} \gamma_{ij} ( \vx_i \cdot \vx_j )^2
\]
can be replaced with the following Hamiltonian
\[
H_{N,\Delta}(\bvx) = \sum_{i<j} \frac{W_{ij}}{   \sqrt{\Delta_{ij} N }} (\vx_i \cdot \vx_j) + \frac{( \vx_i^0 \cdot \vx^0_j)}{\Delta_{ij} N} (\vx_i \cdot \vx_j) - \frac{1}{2 \Delta_{ij} N} ( \vx_i \cdot \vx_j )^2
\]
without changing the limit of the free energy. We
let
\begin{equation}\label{defFD}
F_N(\Delta)=\frac{1}{N} \E_{W,\vx^0} \ln \E_X e^{H_{N,\Delta}(\bvx)}\,.
\end{equation}
The form of this Hamiltonian is identical to the free energy in the low rank matrix estimation with Hadamard covariance profile, which we introduced in Subsection~\ref{subsec:effective}. 
\begin{lem}[Reduction to Low Rank Hamiltonian] \label{lem:univ4}
If \label{lem:lowrankoptimal} \eqref{eq:out_probability} holds, then
$$F_N(\sigma,\mu,\kappa)=F_N(\Delta) +O(\kappa^3 N^{-1/2})\,.$$
\end{lem}

\begin{proof}
Consider the interpolating Hamiltonian,
\begin{align*}
	H_N(t,\bvx) &= \sum_{i<j} \frac{\sqrt{t} \sigma_{ij} W_{ij}}{\sqrt{N}} (\vx_i \cdot \vx_j) + \frac{t\mu_{ij}}{\sqrt{N}} (\vx_i \cdot \vx_j) + \frac{t}{2N} \kappa_{ij} ( \vx_i \cdot \vx_j )^2
	\\&\quad + \sum_{i<j} \frac{\sqrt{1-t} \tilde W_{ij}}{ \sqrt{N \Delta_{ij}}} (\vx_i \cdot \vx_j) + \frac{t(\vx_i^0 \cdot \vx_j^0)}{\Delta_{ij} N} (\vx_i \cdot \vx_j) - \frac{t}{2 \Delta_{ij} N} ( \vx_i \cdot \vx_j )^2.
\end{align*}
where $W$ and $\tilde W$ are independent standard Gaussians. If we define
\[
\phi(t) = \frac{1}{N} \E_{W,\tilde W, x^0} \ln \E_x e^{H_N(t,\bvx)}
\]
then 
\begin{align*}
	\phi'(t) &= \frac{1}{N} \E \bigg\langle \partial_t H_N(t,\bvx) \bigg\rangle_t
	\\&=\frac{1}{N}\E \bigg(\sum_{i<j} \frac{\sigma_{ij} W_{ij}}{2 \sqrt{t} \sqrt{N}} \langle \vx_i \cdot \vx_j \rangle_t  + \frac{\mu_{ij}}{\sqrt{N}} \langle \vx_i \cdot \vx_j \rangle_t + \frac{1}{2N} \kappa_{ij} \langle (\vx_i \cdot \vx_j)^2 \rangle_t \bigg)
	\\&\quad - \frac{1}{N} \E\bigg( \sum_{i<j} \frac{\tilde W_{ij}}{2 \sqrt{1-t}  \sqrt{\Delta_{ij} N }} \langle \vx_i \cdot \vx_j\rangle_t + \frac{(\vx_i^0 \cdot \vx_j^0)}{\Delta_{ij} N} \langle \vx_i \cdot \vx_j \rangle_t - \frac{1}{2 \Delta_{ij} N} \langle (\vx_i \cdot \vx_j )^2 \rangle_t \bigg)
\end{align*}
where $\langle \cdot \rangle_t$ is the average with respect to the Gibbs measure $G_t \propto e^{H_N(t,\bvx)}$. Recall that \eqref{eq:center} and \eqref{eq:center2} imply
\[
\mu_{ij}  =  \frac{\vx_i^0 \cdot \vx_j^0}{\sqrt{\Delta_{ij} N}} + O(\kappa^2N^{-1}) \quad\text{and}\quad \sigma_{ij}^2 =  -\gamma_{ij} + O(\kappa N^{-1/2}) = \frac{1}{\Delta_{ij}} + O(\kappa N^{-1/2})
\]
This implies that the terms without the Gaussian $W,\tilde W$ cancel each other. For the first terms of each expectation,  we integrate by parts to find 
\[
\frac{1}{N}\E \sum_{i<j} \frac{\sigma_{ij} W_{ij}}{2 \sqrt{t} \sqrt{N}} \langle \vx_i \cdot \vx_j \rangle_t = \frac{1}{2N^2} \sum_{i < j} \sigma_{ij}^2\big( \langle (\vx_i^1 \cdot \vx_j^1)^2 \rangle_t - \langle (\vx_i^1 \cdot \vx_j^1)(\vx_i^2 \cdot \vx_j^2) \rangle_t\big)
\]
and
\[
\frac{1}{N}\E \sum_{i<j} \frac{\tilde W_{ij}}{2 \sqrt{1 - t}  \sqrt{ \Delta_{ij} N}} \langle \vx_i \cdot \vx_j \rangle_t = \frac{1}{2 \Delta_{ij} N^2} \sum_{i < j} \big( \langle (\vx_i^1 \cdot \vx_j^1)^2 \rangle_t - \langle (\vx_i^1 \cdot \vx_j^1)(\vx_i^2 \cdot \vx_j^2) \rangle_t\big)
\]
so the difference of the Gaussian terms are also of order  $O(\kappa N^{-1/2})$. Therefore,
\[
|\phi'(t)| = O(\kappa N^{-1/2}),
\]
which completes the proof after integrating $\phi$ on $[0,1]$. 
\end{proof}

\begin{rem}
Notice that $\mu_{ij} = O( \frac{1}{\sqrt{N}} )$, so one of the hypothesis in Lemma~\ref{lem:univ3} is automatically satisfied. 
\end{rem}

\begin{rem}
In the spin glass setting,
\[
\int e^{g(D,w)} \, dy = \int e^{Dw} \, dy
\]
does not equal to $1$, so \eqref{eq:out_probability} is not satisfied and Lemma~\ref{lem:lowrankoptimal} doesn't apply.
\end{rem}

The proof of the main universality theorem is now immediate.

\begin{proof}[Proof of Lemma~\ref{prop:universality}]
The result follows by combining Lemma~\ref{lem:univ1}, Lemma~\ref{lem:univ2}, Lemma~\ref{lem:univ3}, and Lemma~\ref{lem:univ4}.
\end{proof}

Lastly, the computations will be much simpler if $\pP_0(\vx^0)$ was supported on finitely many points. We will now prove that general $\pP_0$ with bounded support can be approximated by $\pP_0$ with finite support.  Suppose that $\pP_0$ is supported on $[-C, C]^\kappa$. We can discretize the support in $K$ blocks by defining the function $f: [-C,C]^\kappa \to [-C,C]^\kappa$ by
\[
f( x_1, \dots, x_\kappa ) = \bigg( \frac{C}{K} \Big\lfloor \frac{x_1K}{C} \Big\rfloor , \dots, \frac{C}{K} \Big\lfloor \frac{x_\kappa K}{C} \Big\rfloor  \bigg).
\]
Notice that each coordinate of $f(\vx)$ is supported on $2K$ points. We define $\pP_d = \pP \circ f^{-1}$, and set
\[
F_N(\Delta) = \E_{W,x^0\sim \pP_0}\Big[\frac{1}{N}\ln\E_{x \sim \pP_0}\left[ \exp\{ H_{N,\Delta}(\bvx)\}\right] \Big] ~ \text{and} ~ F^d_N(\Delta)=\E_{W,x^0 \sim \pP_d}\Big[\frac{1}{N}\ln\E_{x\sim \pP_d}\left[ \exp\{ H_{N,\Delta}(\bvx)\}\right] \Big] .
\]
\begin{lem}[Approximations of References Measures with Finite Support]\label{lem:finitesupport}
For any $N \geq 1$, we have 
\[
|F_N(\Delta) - F_N^d(\Delta)| \leq \frac{2 C^4 \kappa^2 }{\| \Delta \|_\infty K}.
\]
\end{lem}

\begin{proof}
To simplify notation, we define $\tilde x_i = f(\vx_i)$ to be the discretization of $\vx_i$ and $F_N(\Delta)=F_N(\pP_0)$. Consider the interpolating Hamiltonian
\begin{align*}
	H_N(\bvx;t) &=  \sum_{i<j} \frac{ \sqrt{t} W_{ij}}{   \sqrt{ \Delta_{ij} N}} (\vx_i \cdot \vx_j) + \frac{t  ( \vx_i^0 \cdot \vx^0_j)}{\Delta_{ij} N} (\vx_i \cdot \vx_j) - \frac{t}{2 \Delta_{ij} N} ( \vx_i \cdot \vx_j )^2\\
	&+  \sum_{i<j} \frac{ \sqrt{1-t} \tilde W_{ij}}{ \Delta_{ij}  \sqrt{N}} (\tx_i \cdot \tx_j) + \frac{(1-t)  ( \tx_i^0 \cdot \tx^0_j)}{\Delta_{ij} N} (\tx_i \cdot \tx_j) - \frac{(1 -t)}{2 \Delta_{ij} N} ( \tx_i \cdot \tx_j )^2
\end{align*}
where $\tilde W_{ij}$ and $W_{ij}$ are independent standard Gaussians. We consider the usual interpolation Hamtiltonian
\[
\phi(t) = \frac{1}{N} \E\log \E_x e^{H_N(\bvx;t)}.
\]
From integration by parts and the Nishimori property (see equation~\ref{nishimori}) to simplify the Gaussian terms, we see that
\[
\phi'(t) = \sum_{i < j}  \frac{1}{2 \Delta_{ij} N^2} \Big( \E \langle ( \vx_i^0 \cdot \vx^0_j)(\vx_i \cdot \vx_j)   \rangle - \E \langle  ( \tx_i^0 \cdot \tx^0_j)(\tx_i \cdot \tx_j) \rangle  \Big).
\]
Since the $\frac{1}{\Delta_{ij}^2}$ are uniformly bounded, we have the upper bound
\begin{align*}
	|\phi'(t)| &= \frac{1}{4\| \Delta\|_\infty }  \Big| \E \langle ( \vx_i^0 \cdot \vx^0_j)(\vx_i \cdot \vx_j) - ( \tx_i^0 \cdot \tx^0_j)(\tx_i \cdot \tx_j) \rangle  \Big|
	\\&\leq   \frac{1}{4\| \Delta\|_\infty }  \Big| \E \langle ( \vx_i \cdot \vx_j - \tx_i \cdot \tx_j)(\vx^0_i \cdot \vx^0_j)   \rangle \bigg| + \frac{1}{4\| \Delta\|_\infty } \Big| \E \langle  ( \vx_i^0 \cdot \vx^0_j - \tx_i^0 \cdot \tx^0_j)(\tx_i \cdot \tx_j) \rangle  \Big|
	\\&\leq \frac{1}{4 \| \Delta \|_\infty}  \Big( \big( \E \langle ( \vx_i \cdot \vx_j - \tx_i \cdot \tx_j)^2 \rangle \E \langle (\vx^0_i \cdot \vx^0_j)^2   \rangle \big)^{1/2} +  \big( \E \langle  ( \vx_i^0 \cdot \vx^0_j - \tx_i^0 \cdot \tx^0_j)^2 \rangle \E \langle (\tx_i \cdot \tx_j)^2 \rangle \big)^{1/2}  \Big|  \Big)
	\\&\leq  \frac{C^2 \kappa}{\| \Delta \|_\infty}  \Big(  \E \langle ( \vx_i \cdot \vx_j - \tx_i \cdot \tx_j)^2 \rangle  \Big)^{1/2}
	\\&\leq  \frac{C^2 \kappa}{\| \Delta \|_\infty}  \Big(  \E \langle ( \vx_i \cdot ( \vx_j - \tx_j )  + \tx_j \cdot ( \vx_i - \tx_i ) )^2 \rangle  \Big)^{1/2} .
\end{align*}
Since $\| \vx_i - \tx_i\|_\infty \leq \frac{C}{K}$ and $\| \vx_j - \tx_j \|_\infty \leq \frac{C}{K}$ and $\pP_0$ has compact support, we get the rough bound
\[
|\phi'(t)| \leq \frac{2C^4 \kappa^2 }{\| \Delta \|_\infty K} \implies |\phi(1) - \phi(0)| \leq \frac{2 C^4 \kappa^2 }{\| \Delta \|_\infty K},
\]
so the statement follows since $\phi(1) = F_N(\pP_0)$ and $\phi(0) = F_N(\pP_d)$.
\end{proof}
\begin{rem} We can also modify this proof so that it holds as long the probability is compactly supported if we assume that the tails do not grow too much. We just have to be more careful since instead of using the uniform bound on the measure, we can get a bound in terms of $\E \| \vx\|_2^2$ like in the proof in \cite{lelargemiolanematrixestimation}. The same argument (interpolating among the $\Delta$'s rather than the $\vx$) shows how to deduce Theorem \ref{thm:main2} from Theorem \ref{thm:main}.
\end{rem}

Therefore, without loss of generality, we may assume that $\pP_0$ is supported on finitely many points, because we can always approximate general $\pP_0$ with compact support up to arbitrary accuracy with a probability measure supported on only finitely many points. Similarly we can assume that $\Delta$ is stepwise constant. 

\subsection{Spectrum Universality}

In this section, we state a stronger form of universality than Lemma~\ref{prop:universality}. Consider the transformed data matrix
\[
\frac{\tilde Y_{ij}}{\sqrt{N}} =  \frac{1}{\sqrt{N}}  \partial_w g_{ij} (D_{ij},0)  \qquad i,j \leq N
\]
where $(D_{ij})_{i,j}$ are independent and distributed according to $P_{ij}$ conditionally on $\bx^{0}$ as in \eqref{eq:conddist}, 
and the normalized spiked matrix with variance profile $\frac{1}{\sqrt{N}} Y^\Delta$ defined in \eqref{eq:inhomospiked} 
\[
Y^\Delta =   \bD^{\odot \frac{1}{2}} \odot W + \frac{(\bvx^0) (\bvx^0)^{T}}{\sqrt{N} }.
\]
In the Bayes optimal setting the first two moments of the matrix $\tilde Y$ are equivalent to those of $\frac{1}{\Delta} \odot Y^\Delta$ up to a $O(\frac{1}{\sqrt{N}})$ term. Define
\[ 
\tilde \mu_{ij} = \E_Y [\tilde Y_{ij} \given \bx^0] ,\quad \tilde \sigma^2_{ij}  = \E_Y[ (\tilde Y_{ij}  -\tilde \mu_{ij})^2 \given \bx^0].
\]
and
\[
\mu_{ij} = \E_Y \bigg[\frac{1}{\bD} \odot Y^\Delta \given[\bigg] \bx^0\bigg] ,\quad  \sigma^2_{ij}  = \E_Y\bigg[ (\frac{1}{\bD} \odot Y^\Delta  - \mu_{ij})^2 \given[\bigg] \bx^0\bigg]
\]
By Lemma~\ref{lembayes}, it follows that
\begin{cor}Assume that $\pP_X \in \sP(\R^k)$ has compact support. Then for all $i,j\in [N]$, we have
$$ \tilde \mu_{ij} =\frac{x_i^0 \cdot x_j^0}{\Delta_{ij} \sqrt{N}} + O(k^2 N^{-1}) ,~ \tilde \sigma_{ij}^2 = \frac{1}{\Delta_{ij}} + O(k N^{-1/2})$$
and 
$$ \mu_{ij} =\frac{x_i^0 \cdot x_j^0}{\Delta_{ij} \sqrt{N}},~ \sigma_{ij}^2 = \frac{1}{\Delta_{ij}}.$$
\end{cor}
Furthermore, we assume that the Fisher information matrix \eqref{eq:fisherscore} satisfies the following assumption
\begin{hyp}[Quadratic Vector Equation Conditions]\label{hypQVE}
Assume that there exists parameters $p,q,P > 0$ and $L \in \N$ such that
\begin{enumerate}
	\item For all $N$, 
	\[
	\frac{1}{\Delta_{ij}} \leq q \qquad i,j \leq N
	\]
	\item For all $N$, 
	\[
	\bigg( \frac{1}{N \Delta} \bigg)^L_{ij} \geq \frac{p}{N} \qquad i,j \leq N.
	\]
	\item The unique solution $(m_i(z))_{i \leq N}$  of  vector of analytic functions on $\mathbb C^{+}=\{\Im z>0\}$ to the following quadratic vector equation,
	\[
	-\frac{1}{m_i(z)} = z + \sum_{j = 1}^N \frac{1}{\Delta_{ij}} m_j(z) \qquad \Im(z) > 0.
	\] going to zero when $\Im z$ goes to infinity, 
	satisfies
	\[
	|m_i(z)| \leq P,  \qquad i,j \leq N,  \Im(z) > 0.
	\]
\end{enumerate}
\end{hyp}

We have
\begin{theo}[Universality of the Spectrum]\label{univspec}
If $g$ satisfies Hypothesis~\ref{hypDelta} and the corresponding Fisher information matrix \eqref{eq:fisherscore} satisfies Hypothesis~\ref{hypQVE}, then
\begin{enumerate}
	\item Conditionally on $\bvx^0$, the empirical distribution $\mu_1$ of the eigenvalues of $\frac{\tilde Y_{ij}}{\sqrt{N}}$ and the empirical distribution $\mu_1$ of the eigenvalues of $\frac{1}{\sqrt{N} \Delta} \odot Y^\Delta$ satisfy
	\begin{equation}\label{convmeas}
		\lim_{N \to \infty} d( \mu_1, \mu_2) \to 0
	\end{equation}
	in probability. 
	\item Conditionally on $\bvx^0$, when the dimension goes to infinity, $\frac{1}{\sqrt{N} \Delta} \odot Y^\Delta$ has an extremal eigenvalue away from the bulk iff 
	$\frac{\tilde Y_{ij}}{\sqrt{N}}$ does, for almost all 
	$\Delta$ and $\rho$. 
\end{enumerate}
\end{theo}

\begin{proof}
We  fix the realization of $\bvx^0$ and assume that $\bvx^0 \in \R^n$ to simplify notation (see Remark~\ref{rem:finiterank} for the generalization to higher rank).  We first show that the spectrum of $\frac{1}{\sqrt{N} }  \tilde Y^\Delta$ and the spectrum of 
$$Z= \frac{1}{\sqrt{N} } \left( \tilde Y -\tilde\mu+\mu\right)$$
differ  by a matrix with operator norm bounded by $O(1/\sqrt{N})$.
Indeed, bounding the operator norm by the Hilbert-Schmidt norm, we get, %assuming $\Delta_{st}>c>0$ for all $s,t$, 
$$\|Z-\frac{1}{\sqrt{N} } \tilde Y\|_{op} \le\left(\frac{1}{N } \sum_{ij}(\mu_{ij}-\tilde\mu_{ij})^{2}\right)^{1/2}\le \sqrt{O\Big( \frac{k^{4} N^{2}}{N^{3}}\Big)}=O\Big(\frac{k^{2}}{\sqrt{N}} \Big)\,.$$
Next, we check that the empirical distribution of the eigenvalues of $Z$ and  of $\frac{1}{\sqrt{N} \Delta} \odot Y^\Delta$ are close, as well as the largest eigenvalue. We 
first consider the recentered matrices 

$$\tilde W= \frac{1}{\sqrt{N} } \left(\tilde Y -\tilde\mu\right), \qquad W_{ij}=\frac{1}{\sqrt{N} }(\Delta_{ij}^{-1}Y^\Delta_{ij}-\mu_{ij})\,.$$
We use \cite{universalitywigner2017} and check that all the conditions  of this paper are satisfied by these matrices $W$ and $\tilde W$.
We therefore can use \cite[Theorem~1.7 and (1.22)]{universalitywigner2017}, to conclude that the Stieltjes transform  $G_{\tilde W}(z)$ of $\tilde W$  and  the Stilejes transform $G_{W}(z)$ of
$W$ are close to their deterministic limits. Namely, under these  hypotheses, if $X$ is a matrix with centered independent entries with variance $s_{ij}$  bounded by $c/N$ then for any deterministic vector $w$ so that $\|w\|_{\infty}\le 1$
\begin{equation}\label{conv}\left|\frac{1}{N} \sum_{i=1}^{N}w_{i}((z-X)^{-1}_{ii}-m^{s}_{i}(z))\right|\le C\frac{1}{\sqrt{N\Im z}}
\end{equation}
where ${\mathbf m}^{s}=\{m^{s}_{i}\}_{1\le i\le N}$ is the unique solution of the vector equation
$$-\frac{1}{m^{s}_{i}(z)}=z+\sum_{i=1}^{N}s_{ij}m^{s}_{j}(z).$$
We therefore only need to check that $m^{\tilde\sigma/N}$ and $m^{\sigma/N}$ are close and apply \eqref{conv} with $w_{i}=1$ for all $i$ to conclude that the Stieltjes transform of both matrices are close to each other, yielding the conclusion by \eqref{conv}. This follows from  \cite[Corollary 3.4]{universalitywigner2017} which asserts that if $\Im z >\delta$,
\begin{equation}\label{an} \|{\mathbf m}^{\sigma/N}(z)-{\mathbf m}^{\tilde \sigma/N}(z)\|_{\infty}\le\frac{1}{\delta} \max_{i}\frac{1}{N}\sum_{j}|\sigma_{ij}^{2}-\tilde\sigma^{2}_{ij}|\le O \Big( \frac{k}{\sqrt{N} \delta} \Big)\,.
\end{equation}
Hence, combining with \eqref{conv} (with $w_{i}=1$), we conclude that  the empirical measures $\mu_{W}$ and $\mu_{\tilde W}$ converge vaguely to the same deterministic limit. Since moreover $\frac{1}{N}\Tr(W^{2})$ and $\frac{1}{N}\Tr(\tilde W^{2})$ are uniformly bounded with overwhelming probability, we deduce that $d(\mu_{1},\mu_{2})$ goes to zero in probability. 

Observe that $  \mu$ has finite rank because $x^{0}(x^{0})^{T}$ has finite rank and $\Delta$ is piecewise constant. Therefore, the empirical measure of the eigenvalues of $\frac{1}{\sqrt{N}}\tilde Y$  and  $\frac{1}{\sqrt{N} \Delta} \odot Y^\Delta$  are approximately the same as those of 
$W$ and $\tilde W$ by Weyl's interlacing property. This shows \eqref{convmeas}.

Moreover,  \cite[Corollary 1.10]{universalitywigner2017} show that the eigenvalues of $W$ and $\tilde W$   stick to the bulk, namely the extreme eigenvalues converge towards the boundary of the support of the measure with Stieljes transform $\sum \rho_{s}m_{s}^{\sigma/N}(z)$ and  $\sum \rho_{s}m_{s}^{\tilde\sigma/N}(z)$. Because of \eqref{an}, these boundaries are very close to each other.

We next study the BBP transition  and show that the top eigenvalues of the matrices
$$ Z= \tilde W+  \frac{1}{\sqrt{N}}  \mu \mbox{ and } \frac{1}{\sqrt{N}\Delta}\odot Y^\Delta= W+    \frac{1}{\sqrt{N}}  \mu $$
have the same limits.
Recall that in general, if $X$ is a self-adjoint matrix and $R= \sum_{i=1}^{r} \theta_{i}v_{i}v_{i}^{T}$,$\theta_{i}\neq 0$, is a finite rank matrix then $\lambda$ is an eigenvalue of $X+R$ iff 
$\det( \lambda-X-R)$ vanishes, and therefore if $\lambda$ does not belong to the spectrum of $X$, this is also equivalent to:
\begin{equation}\label{BBP}
	0=\det( I-(\lambda-X)^{-1}R)=\prod \theta_{i}\det (\mbox{diag}(\theta_{j}^{-1})- \left( \langle v_{i},(\lambda-X)^{-1} v_{j}\rangle\right)_{1\le i,j\le r}).
\end{equation}
Therefore, to prove that an eigenvalue pops out of the bulk, it is  necessary and sufficient to prove that the above  right hand side vanishes for some $\lambda$ outside of the bulk of $X$.
We will show that 
the matrix $\left( \langle v_{i},(\lambda-X)^{-1} v_{j}\rangle\right)_{1\le i,j\le r}$ converges  for 	$X=Z$ and  $ \frac{1}{\sqrt{N} \Delta} \odot Y^\Delta$  and that the limiting equation for the outliers  has a unique and stable solution.
In our case,  $$R=   \frac{1}{\sqrt{N}}  \mu=
\frac{1}{N} \frac{1}{\Delta} \odot x x^\trans
$$
and $\Delta$ is piecewise constant. Let $x(s) = (x_i \1(i \in I_s))_{i \leq N}$. We have the following decomposition
\begin{equation}\label{eq:decomspecies}
	R = \frac{1}{N} \sum_{s,t = 1}^n  \frac{\|x(s)\| \|x(t)\|}{\Delta_{s,t}}  \frac{x(s)}{\|x(s)\|} \bigg( \frac{x(t)}{\|x(t)\|} \bigg)^\trans.
\end{equation}
For fixed $x,s, t$ we denote $M_{st}=\frac{\|x(s)\|\|x(t)\|}{\Delta_{st}}$. $M$ is symmetric and if $(\gamma_{i},w_{i})$ are its eigenvectors and eigenvalues, we find that 
\[
R = \frac{1}{N} \sum_{i = 1}^n\gamma_{i}  v_{i}v_{i}^{T}
\]
with $v_{i}=\sum_{s=1}^{n} \frac{x(s)}{\|x(s)\|} w_{i}(s)$ an orthonormal family of eigenvectors of $R$. 
\cite[Theorem~1.13]{universalitywigner2017} implies that for any $\gamma>0$ 
$$\left \|  \left( \langle v_{i},(\lambda+i N^{-1+\gamma}-Z)^{-1} v_{j}\rangle\right)_{1\le i,j\le n} -\left( \sum_{k=1}^{N}m^{\tilde \sigma}_{k}(\lambda +i N^{-1+\gamma}) v_{i}(k) v_{j}(k)\right)_{1\le i,j\le n}\right\| $$
and
$$\left \| 
\left( \langle v_{i},(\lambda +i N^{-1+\gamma}-\frac{1}{\sqrt{N} \Delta^{2}} \odot Y^\Delta)^{-1} v_{j}\right)_{1\le i,j\le n} 
-\left( \sum_{k=1}^{N}m^{\sigma}_{k}(\lambda +i N^{-1+\gamma} ) v_{i}(k) v_{j}(k)\rangle\right)_{1\le i,j\le n}\right\| $$
go to zero with overwhelming probability. If $\lambda$ is outside of the support of the limiting distribution then we can remove the small complex number $i N^{-1+\gamma}$, and since we have seen that the support of the eigenvalues of both centered matrices converge to the same limit this is fine for any $\lambda$ at a positive distance of this limiting support.
Moreover, by the stability property \eqref{an} we know that 
$$\left \| \left( \sum_{k=1}^{N}m^{\tilde \sigma/N}_{k}(\lambda+iN^{-1+\gamma}) v_{i}(k) v_{j}(k)\right)_{1\le i,j\le n}
-\left( \sum_{k=1}^{N}m^{\sigma/N}_{k}(
\lambda+i N^{-1+\gamma}) v_{i}(k) v_{j}(k)\rangle\right)_{1\le i,j\le n}\right\| $$
goes to zero if $\lambda$ is away from the support of the limiting measure. Hence the only thing to verify is that the largest solution $\lambda$ to \eqref{BBP} does not change much under these small perturbations. To that end, first notice that because $\Delta$ is piecewise constant, so is $\sigma$ and therefore ${\mathbf m}^{\sigma}$ is piecewise constant, and $m^{\sigma}_{i}(z)$ equals to $m^{\sigma/N}_{s}(z)$ for $i\in I_{s}$. We can therefore sum over the indices inside each $I_{s}$ and find
$$\left( \sum_{k=1}^{N}m^{\sigma/N}_{k}(z) v_{i}(k) v_{j}(k)\right)_{1\le i,j\le n}= \left( \sum_{s=1}^{n}m^{\sigma/N}_{s}(z) w_{i}(s) w_{j}(s)\right)_{1\le i,j\le n}= w\mbox{diag} (m(z)) w^{T}$$
Therefore the outliers of $X= \frac{1}{\sqrt{N} \Delta} \odot Y^\Delta$  and $X=Z$ satisfy 
$$\det \left(\mbox{diag}(\gamma_{j}^{-1})- w\mbox{diag} (m^{\sigma/N}(\lambda)) w^{T}+\epsilon(X) \right)=0$$ 
with $\epsilon(X)$ a matrix with operator norm going to zero. 
The last thing to check is that the largest solution to this equation are arbitrarily close to each others when $\|\epsilon(X)\|_{op}$ go to zero. But the above equations characterize the outliers as zeroes of the analytic function (outside of the support of the limiting measure) 
$$ F(M,\lambda)=\det\left(M -w\mbox{diag} (m^{\sigma/N}(\lambda) ) w^{T} \right)=0$$
where $M$ belongs to a neighborhood of $\mbox{diag}(\gamma_{j}^{-1})$. As long as the derivative of $F$ in $\lambda$ does not vanish, its solution is smooth. This is true for almost all $\gamma_{i}$'s, namely almost all $\Delta$ and $\rho$.

Finally, observe that if two $N\times N$  matrices $X$ and $Y$ are such that their largest eigenvalues are close and their empirical measures are close (with  atomless limits) then $X$ and $Y$ are close in operator norms in the sense that if the eigenvalues $\lambda_{i}(X)$ and $\lambda_{i}(Y)$ are increasing in $i$, 
$$\limsup_{N\rightarrow \infty}\max_{i}|\lambda_{i}(X)-\lambda_{i}(Y)|=0$$ in probability. This applies to $X=Z$ and $Y=\frac{1}{\sqrt{N} } \tilde Y$ by the previous arguments. Indeed,  for any self-adjoint matrix $Z$

$$\hat x^{i-1}_{Z} \le \lambda_{i}(Z)\le \hat x^{i}_{Z}$$
where $\hat x_{Z}^{i}=\inf\{ x:\hat \mu_{Z}([x,\lambda_{max}(Z)])\ge (N-i)/N \}$ and $\hat\mu_{Z}$ is the empirical measure of the eigenvalues of $Z$. But, because the empirical measure of the eigenvalues converge towards the same limit and the limit correspond to an atomless measure, together with the convergence towards the same limit of $\lambda_{max}(X)$ and $\lambda_{max}(Y)$, we find that
for each $\delta>0$, for $N$ large enough $\max_{i}|\hat x^{i}_{X}-\hat x^{i}_{Y}|\le \delta$. 	
\end{proof}

\begin{rem}\label{rem:finiterank}
If $\bvx^0 \in \R^{N \times \kappa}$, then we can write decompose
\[
(\bvx^0) (\bvx^0)^\trans = \sum_{j = 1}^\kappa \theta^j u^j (u^j)^\trans,
\]
where $\theta^1 \geq \theta^2 \geq \dots \geq \theta_\kappa$. 
Then repeating the computation following \eqref{eq:decomspecies} with $u^j(s) = (\sqrt{\theta^j} u^j_i \1(i \in I_s) )_{i \leq N}$ and $R$ of the form
\[
R = \frac{1}{N} \sum_{j = 1}^\kappa \sum_{s,t = 1}^n  \frac{\|v^j(s)\| \|v^j(t)\|}{\Delta_{s,t}}  \frac{v^j(s)}{\|v^j(s)\|} \bigg( \frac{v^j(t)}{\|v^j(t)\|} \bigg)^\trans.
\]
The rest of the proof remains unchanged, since we only examine the behavior of $v^1$ .
\end{rem}

\section{Lower Bound - Gaussian Interpolation} \label{sec:lwbd}

Given a sequence of symmetric matrix $\kappa \times \kappa$ matrix $\bQ_s$ for each $s \leq n$, we want to derive the replica symmetric formula. Let
\[
\tilde \bQ_s = \sum_{t \leq n} \frac{1}{\Delta_{s,t}} \rho_t \bQ_t .
\]
and define 
\begin{align*}
\phi(\bQ) &= - \sum_{s,t =1}^n \frac{\rho_s \rho_t}{4 \Delta_{s,t}}\Tr( (\bQ_s)^\trans \bQ_t) +  \sum_{s =1}^n \rho_s \E \ln \bigg[ \int \exp \bigg( \bigg( \tilde \bQ_s\vx^0 + \sqrt{\tilde \bQ_s} \vz \bigg)^\trans \vx - \frac{\vx^\trans \tilde \bQ_s \vx}{2}  \bigg) \, d \pP_X(\vx) \bigg]
\end{align*}
where $\vx^0 \sim \pP_0$ and $\vz \sim N(0,\bI_r)$. Recall that, in the Bayes optimal case we defined in \eqref{eq:fisherscore}
\[
\frac{1}{\Delta_{ij}} = \E_{P_\out(D|w=0)} ( \partial_w g_{ij}(D,0) )^2
\]
which takes $n^2$ different values by Hypothesis~\ref{hypg} on $\Delta$. 
The goal of this section is to prove that $\phi$ is a lower bound for the free energy.

\begin{theo}[Bayes Optimal Lower Bound of the Free Energy] Assume Hypotheses \ref{hypcompact}, \ref{hypDelta}, \ref{hypbayes}. 
Then, for any $\bQ = (\bQ_1, \dots, \bQ_n) \in (\mathbb{S}_\kappa^+)^n$,
\[
F_N(\Delta) \geq \phi( \bQ ) - O(\kappa N^{-1/2}).
\]
\end{theo}

\begin{proof}
We follow the standard interpolation proof.  Fix a sequence $(\bQ_s)_{s \leq n}$ of positive semidefinite matrices and for each $i \in I_s \subset N$, we set $\tilde Q_i = \tilde Q_s$.
Let $z_i$ be \iid standard Gaussians independent of all other sources of randomness, and consider the interpolating Hamiltonian
\begin{align*}
	H_N(t,\bvx) &= \sum_{i<j} \frac{\sqrt{t} W_{ij}}{ \sqrt{\Delta_{ij} N}} (\vx_i \cdot \vx_j) + \frac{t}{\Delta_{ij} N} (\vx_i \cdot \vx_j)(\vx_i^{0} \cdot \vx_j^{0}) - \frac{t}{2 \Delta_{ij} N} ( \vx_i \cdot \vx_j )^2
	\\&\quad + \sum_{i \leq N}  \sqrt{1 - t} \big(  (\tilde\bQ_i^{1/2}\vz_i)\cdot   \vx_i\big) + (1-t) \big( (\tilde \bQ_i\vx^0_i) \cdot \vx_i \big) - \frac{(1-t)}{2} \vx_i^\trans \tilde\bQ_i \vx_i.
\end{align*}
The corresponding interpolating free energy is given by
\[
\phi(t) = \frac{1}{N} \E \ln \int e^{H_N(t,\bvx)} \, d \pP_0^{\otimes N}(\vx) .
\]
It follows that	\begin{equation}\label{eq:interpolationboundsOpt}
	\phi(1) = F_N(\Delta) \quad\text{and}\quad \phi(0) =  \sum_{s \leq n} \rho_s \E_{\vz,\vx^0} \ln \bigg[ \int \exp \bigg( \bigg( \tilde \bQ_s \vx^0 + \sqrt{ \tilde \bQ_s} \vz \bigg)^\trans \vx - \frac{\vx^\trans \tilde \bQ_s \vx}{2}  \bigg) \, d \pP_X(\vx) \bigg].
\end{equation}
We now control the derivative
\begin{align}
	\phi'(t) %&= \frac{1}{N} \E \bigg\langle \partial_t H_N(t,\bvx) \bigg\rangle_t\notag
	&=\frac{1}{N}\E \bigg(\sum_{i<j} \frac{W_{ij}}{2 \sqrt{t} \sqrt{ \Delta_{ij} N}} \langle \vx_i \cdot \vx_j \rangle_t  + \frac{1}{\Delta_{ij} N} \vx_i^0 \cdot \vx_j^0\langle \vx_i \cdot \vx_j \rangle_t - \frac{1}{2\Delta_{ij} N} \langle (\vx_i \cdot \vx_j)^2 \rangle_t \bigg)\notag
	\\&\quad - \frac{1}{N} \E\bigg( \sum_{i = 1}^N  \frac{1}{2 \sqrt{1 - t}} \big\langle  (\tilde \bQ_i^{1/2}\vz_i)\cdot   \vx_i\big\rangle_t +  \big\langle (\tilde\bQ_i\vx^0_i) \cdot \vx_i \big\rangle_t - \frac{1}{2} \langle \vx_i^\trans \tilde\bQ_i  \vx_i \rangle_t \bigg) \label{eq:lowboundderiv}
\end{align}
where the inner average is with respect to the Gibbs measure associated with $H_N(t,\bvx)$,
\[
\langle f \rangle_t = \frac{ \int f e^{H_N(t,\bvx)} \, d \pP_0^{\otimes N}(\vx) }{ \int e^{H_N(t,\bvx)} \, d \pP_0^{\otimes N}(\vx)}.
\]

The Gaussian terms in \eqref{eq:lowboundderiv} can be simplified by integrating by parts,
\[
\E_W \sum_{i<j} \frac{W_{ij}}{2 \sqrt{t} \sqrt{ N  \Delta_{ij}}} \langle \vx_i \cdot \vx_j \rangle_t = \E_W \sum_{i<j} \frac{1}{2 \Delta_{ij} N} \langle (\vx_i^1 \cdot \vx^1_j)^2 \rangle_t - \E_W \sum_{i<j} \frac{1}{2 \Delta_{ij} N} \langle (\vx_i^1 \cdot \vx^1_j) (\vx_i^2 \cdot \vx^2_j) \rangle_t 
\]
and similarly,
\[
\E_{\vz} \sum_{i = 1}^N  \frac{1}{2 \sqrt{1 - t}  } \big\langle  (\tilde \bQ_i^{1/2}\vz_i)\cdot   \vx_i\big\rangle_t = \E_{\vz} \sum_{i = 1}^N  \frac{1}{2} \big\langle  (\vx^1_i)^\trans \tilde \bQ_i (\vx^1_i)\big\rangle_t - \E_{\vz} \sum_{i = 1}^N  \frac{1}{2} \big\langle  (\vx^1_i)^\trans \tilde \bQ_i \vx^2_i\big\rangle_t
\]
where $\bx^2$ is an independent copy (replica) of $\bx^1$. All the second order terms cancel with the self overlap terms in \eqref{eq:lowboundderiv} leaving us with 
\begin{align}
	\phi'(t) &=\frac{1}{N}\E \bigg(\sum_{i<j} - \frac{1}{2 \Delta_{ij} N} \langle (\vx_i^1 \cdot \vx^1_j) (\vx_i^2 \cdot \vx^2_j) \rangle_t +  \frac{\vx_i^0 \cdot \vx_j^0}{\Delta_{ij} N }  \langle \vx_i \cdot \vx_j \rangle_t  \bigg)
	\nonumber\\&\quad - \frac{1}{N} \E\bigg( \sum_{i = 1}^N - \frac{1}{2} \big\langle  (\vx^1_i)^\trans \tilde \bQ_i \vx^2_i\big\rangle_t + \big\langle (\vx^0_i)^\trans \tilde \bQ_i \vx_i \big\rangle_t  \bigg) + O(\kappa N^{-1/2})\label{eqni}
\end{align}
where the error comes from the diagonal terms of the overlap matrices, which are of order $\kappa$. 
We can now  use the Nishimori property (see for example \cite[Proposition~16]{lelargemiolanematrixestimation})
\begin{equation}\label{nishimori}
	\E \langle f(\bvx^1,\bvx^2,\dots, \bvx^n) \rangle_t = \E \langle f(\bvx^0, \bvx^2, \dots, \bvx^n) \rangle_t
\end{equation}
to replace one replica under the average interpolating Gibbs measure with the signal.  To prove \eqref{nishimori}, it is enough to show that the average $\langle \cdot \rangle_t$ can be interprated as a distribution of $\vx$ conditionally on $\vx^0,W$ and $\vz$.  Indeed, if we let $W,\vz$ to be Gaussian and set 
$$Y_{ij}= \sqrt{t} \vx_{i}^{0} \cdot \vx_{j}^{0}+ \sqrt{\Delta_{ij}} W_{ij}.$$
Then the law of $Y_{ij}$ has density proportional to  $e^{-\frac{1}{2 \Delta_{ij} }( Y_{ij}-\sqrt{t} \vx_{i}^{0} \cdot \vx_{j}^{0} )^{2}}$. Hence the law of $\vx$ such that 
$$Y_{ij}=\sqrt{t} \vx_{i} \cdot \vx_{j}+  \sqrt{\Delta_{ij}} W_{ij}$$ 
conditionally to $\vx^{0}$  and $W$ has density with respect to $\pP(\vx)$  proportional to 
$$\exp\bigg( -\frac{1}{2 \Delta_{ij}}\Big( \sqrt{t} \vx_{i} \cdot \vx_{j}+\Delta_{ij}W_{ij}-\sqrt{t} \vx_{i}^{0} \cdot \vx_{j}^{0} \Big)^{2} \bigg)$$
which is  proportional to 
$$\exp \bigg(  \sum_{i<j} \frac{\sqrt{t} W_{ij}}{\sqrt{\Delta_{ij} N}} (\vx_i \cdot \vx_j) + \frac{t}{\Delta_{ij} N} (\vx_i \cdot \vx_j)(\vx_i^{0} \cdot \vx_j^{0}) - \frac{t}{2 \Delta_{ij} N} ( \vx_i \cdot \vx_j )^2 \bigg).$$
The same is true  for the second term if we write
$$y_{i}=\vz_{i}+\sqrt{1-t} \tilde \bQ_i^{1/2} \vx_{i}^{0}$$
with $\vz$ a standard Gaussian vector,
and condition by $y_{i}=\vz_{i}+\sqrt{1-t} \tilde \bQ_i^{1/2} \vx_{i}$.  Hence 
$ e^{H_N(t,\bvx)} \, d \pP_0^{\otimes N}(\vx) $ is the distribution  of $\vx$ conditioned by  $W,\vz$ and $\vx^0$.
This implies \eqref{nishimori}. Plugging the Nishimori equation \eqref{nishimori} into \eqref{eqni} yields
that
\begin{align*}
	\phi'(t) &=  \E\bigg(\frac{1}{2 N^2} \Big\langle \sum_{i < j} \langle \frac{1}{\Delta_{ij}}  (\vx_i^1 \cdot \vx^1_j) (\vx_i^2 \cdot \vx^2_j) \Big\rangle_t \bigg)  - \E\bigg(\frac{1}{2 N} \Big\langle \sum_{i = 1}^N \Tr(\bQ_i \vx^1_i (\vx^2_i)^\trans \Big\rangle_t  \bigg) + O(\kappa N^{-1/2})\\
	&=  \E\bigg(\frac{1}{4} \bigg\langle \sum_{s,t = 1}^n \frac{\rho_s \rho_t}{\Delta_{s,t}} \Tr\Big( (\bR^s_{1,2})^\trans  \bR^t_{1,2}\Big) \bigg\rangle_t \bigg)  - \E\bigg(\frac{1}{2} \Big\langle \sum_{s =1}^n \frac{\rho_s \rho_t}{\Delta_{s,t}} \Tr\Big( (\tilde \bQ_s)^\trans \bR_{1,2}^t \Big) \Big) \Big\rangle_t  \bigg) + O(\kappa N^{-1/2})
	%\\&\geq - \sum_{s,t \leq n} \frac{\rho_s \rho_t}{4 \Delta_{s,t}^2}\Tr( (\bQ_s)^\trans \bQ_t) + O(N^{-1/2})
\end{align*}
where we denoted by $\bR \in \R^{\kappa \times \kappa}$ the overlap matrix defined for $t \in \{1,\dots, n \}$ defined by:
$$
R^t_{ab}  =\frac{1}{|I_t|}\sum_{i\in I_t} \vx^a_i(\ell)\vx^b_i(k),\quad \ell,k\in [\kappa]\,.$$
Adding and subtracting $\sum_{s,t \leq n} \frac{\rho_s \rho_t}{4 \Delta_{s,t}^2}\Tr( (\bQ_s)^\trans \bQ_t)$, completes the square so the formula simplifies to
\[
\phi'(t)=\E\bigg(\frac{1}{4} \bigg\langle \sum_{s,t =1}^N \frac{\rho_s \rho_t}{\Delta_{s,t}} \Tr\Big( ((\bR^s_{1,2})-\bQ_s)^\trans ( \bR^t_{1,2}-\bQ_t)\Big) \bigg\rangle_t \bigg)
- \sum_{s,t =1}^N \frac{\rho_s \rho_t}{4 \Delta_{s,t}}\Tr( (\bQ_s)^\trans \bQ_t) + O(\kappa N^{-1/2})
\]

Our assumption that $\frac{1}{\Delta}$ is  a non-negative matrix by Hypotheses~\ref{hypDelta} implies that the first term is non-negative, so we arrive at the lower bound
\[
\phi'(t)  \geq - \sum_{s,t =1}^n \frac{\rho_s \rho_t}{4 \Delta_{s,t}}\Tr( (\bQ_s)^\trans \bQ_t) + O(\kappa N^{-1/2})
\]
Integrating this bound implies that
\[
\phi(1) \geq \phi(0) - \sum_{s,t =1}^n \frac{\rho_s \rho_t}{4 \Delta_{s,t}}\Tr( (\bQ_s)^\trans \bQ_t)  + O(\kappa N^{-1/2})
\]
so the conclusion follows. %from \eqref{eq:interpolationboundsOpt}.
\end{proof}

\section{The Upper Bound --- Cavity Computations} \label{sec:upbd}

\subsection{Concentration of the Overlaps}

We will introduce a perturbation of the Hamiltonian that will imply concentration of the Hadamard powers of the overlaps and the generalized Ghirlanda--Guerra identities \cite{panchenko2015free, PPotts,PVS} in each block of the inhomogeneous vector spin models. Given a vector $\vl= (\lambda(1), \dots, \lambda(\kappa))\in \R^\kappa$, consider the $p$-spin Gaussian estimation problem
\begin{equation}\label{eq:inferencepert}
Y_{\iii} =  g_{\iii} +  \frac{s}{ N^{\frac{p-1}{2} }} \sum_{k \leq \kappa} \lambda(k) x^0_{i_1}(k) \cdots x^0_{i_p}(k)
\end{equation}
where $1 \leq i_1, \dots, i_p \leq N$ is some enumeration of the indices. Later on, we will restrict $\iii \in I_s$, but the pertubation Hamiltonian can be defined more generally. Since 
$$Y_{\iii} - \frac{s}{ N^{\frac{p-1}{2} }} \sum_{k \leq \kappa} \lambda(k) x^0_{i_1}(k) \cdots x^0_{i_p}(k)$$
is a standard Gaussian variable, the maximum likelihood estimator of this Gaussian channel is proportional to
\begin{align*}
d\pP(\bvx^0| Y_{\iii})  &= \frac{1}{Z}  \exp \bigg(- \frac{1}{2} ( Y_{\iii} - \frac{s}{ N^{\frac{p-1}{2}} } \sum_{k \leq \kappa} \lambda(k) x_{i_1}(k) \cdots x_{i_p}(k) )^2  \bigg) d\pP^{\otimes N}_0(\vx)
\\&= \frac{1}{Z'}  \exp\bigg(  \frac{s}{ N^{\frac{p-1}{2}} } \sum_{k \leq \kappa} Y_{\iii}  \lambda(k) x_{i_1}(k) \cdots x_{i_p}(k) 
\\&\qquad -\frac{s^2 }{ 2N^{p-1} } \sum_{k,k' \leq \kappa} \lambda(k) x_{i_1}(k) \cdots x_{i_p}(k) \lambda(k') x_{i_1}(k') \cdots x_{i_p}(k') \bigg) d\pP^{\otimes N}_0(\vx)
\end{align*}
where $Z$ and $Z'$ are the partition functions or normalizing constants. 
We denote for $t\leq n$
$$(\bR_{\ell,\ell'}^t)(k,k') =\frac{1}{N_t}\sum_{i \in I_t} x^{\ell}_i(k) x^{\ell'}_i(k')\mbox{ and }( \bR^t_{\ell,\ell'})^{\odot p}(k,k')= ( \bR^t_{\ell,\ell'}(k,k') )^p, k,k'\in [\kappa]$$
where $N_t$ are the proportions of indices in the group with index $t$ as defined in \eqref{eq:nt}. We denote also in short $(\rho_t \bR^t_{\ell,\ell'})^{\odot p}=\rho_t^p( \bR^t_{\ell,\ell'})^{\odot p}$.

If we consider an independent copy for each $i_1, \dots, i_p$ conditionally on $\bvx^0$, then the perturbation Hamiltonian (the log-likelihood) for $t \leq n$ is given by
\begin{align}
H^t_{N,p}(\bvx,g, \lambda, s) %&= \sum_{\iii} \frac{s}{ N^{\frac{p-1}{2}} } \sum_{k \leq r} Y_{\iii}  \lambda(k) x_{i_1}(k) \cdots x_{i_p}(k)  - \frac{s^2}{ N^{p-1} } \sum_{k,k' \leq r} \lambda(k) x_{i_1}(k) \cdots x_{i_p}(k) \lambda(k') x_{i_1}(k') \cdots x_{i_p}(k') \notag
&= \sum_{\iii \in I_t} \frac{s}{ N^{\frac{p-1}{2}} } \sum_{k \leq \kappa} g_{\iii}  \lambda(k) x_{i_1}(k) \cdots x_{i_p}(k) +s^2 N\big( \vl^\trans (\rho_t \bR^t_{1,0})^{\odot p} \vl \big) - \frac{s^2}{2} N\big( \vl^\trans (\rho_t \bR^t_{1,1})^{\odot p} \vl \big)  \label{eq:3paramhamiltonian}
\end{align}
Notice that the covariance of the Gaussian term is
\begin{align*}
&\frac{1}{N} \E \bigg(\sum_{\iii \in I_t} g_{\iii}  \frac{s}{ N^{\frac{p-1}{2}} } \sum_{k \leq \kappa}  \lambda(k) x^1_{i_1}(k) \cdots x^1_{i_p}(k) \bigg) \bigg(\sum_{\iii \in I_t} g_{\iii}  \frac{s}{ N^{\frac{p-1}{2}} } \sum_{k \leq \kappa} \lambda(k) x^2_{i_1}(k) \cdots x^2_{i_p}(k) \bigg) 
\\&= \frac{s^2}{N^{p}}\sum_{\iii \in I_t} \bigg( \sum_{k \leq \kappa}  \lambda(k) x^1_{i_1}(k) \cdots x^1_{i_p}(k) \bigg) \bigg( \sum_{k \leq \kappa}  \lambda(k) x^2_{i_1}(k) \cdots x^2_{i_p}(k) \bigg) 
\\&= s^2 \big( \vl^\trans (\rho_t \bR^t_{1,2} )^{\odot p} \vl \big).
\end{align*}
For $s$ sufficiently large, adding this perturbation to the Gibbs measure will imply concentration of the quadratic forms $(\vl^\trans \bR_{1,2}^{\odot p} \vl)$. For applications, we will need concentration for all $\vl$ and all $p \geq 1$. For any $\vl \in [-1,1]^{\kappa}$ and $p \geq 1$, the overlap is uniformly bounded 
\begin{equation}\label{boundO}
(\vl^\trans \bR_{1,2}^{\odot p} \vl) \leq \kappa^2 C^{2p},
\end{equation}
since the vector spin coordinates $\vx$ are uniformly bounded by some $C \geq 1$. Let $\lambda_{m}$ be a countable enumeration of elements of the dense set $( [-1,1] \cap \mathbb{Q})^\kappa$ and consider the perturbed Hamiltonian
\[
H_N^{\pert}(\bvx, \vl) =  H_N(\bvx) + \sum_{t = 1}^n \sum_{p \geq 1} \sum_{m \geq 1} H^t_{N,p}\Big(\bvx,g_{m,t}, \lambda_{m}, \frac{u_{m,p,t} \epsilon_{N}}{2^{m + p} \kappa C^p}  \Big).
\]
The above sum is infinite but the constants will be chosen so that the covariance of the above Gaussian process is absolutely converging.
The gaussian variables $g_{m,t}$ appearing in different $H_{N,p}^t$ are independent. We choose the scaling coefficient $C_{m,p} := 2^{m + p} \kappa C^{p} $ so that the covariance is of order $O(s^2 N)$
\begin{align}\label{eq:covbound}
&\frac{1}{N}\Cov \bigg(\sum_{t =1}^n \sum_{p \geq 1} \sum_{m \geq 1} H^p_N\Big(\bvx, g_{m,t}, \lambda_m, \frac{u_{m,p,t} \epsilon_{N}}{C_{m,p}} \Big)  , \sum_{t =1}^n \sum_{p \geq 1} \sum_{m \geq 1} H^p_N\Big(\bvx,  g_{m,t}, \lambda_m, \frac{u_{m,p,t} \epsilon_{N}}{C_{m,p}}  \Big) \bigg) \notag
\\&= \epsilon_{N}^2 \sum_{t =1}^n \sum_{p \geq 1} \sum_{m \geq 1} u^2_{m,p,t}  \frac{(\vl_m^\trans (\rho_t \bR_{1,2}^t)^{\odot p} \vl_m  ) }{C_{m,p}^2}  \notag
\\&\leq \epsilon_{N}^2 \max_{m,p,t} ( u_{m,p,t}^2 ).
\end{align}
If $\max_{m,p,t} u_{m,p,t} \epsilon_N \to 0$, then the covariance of the perturbation term will be of lower order than the Hamiltonian, so the limit of the free energy will not change because of the perturbation. To make this precise, consider the perturbed free energy as a function of the infinite sequence $u = (u_{m,p } )$
\begin{equation}\label{eq:pertFE}
F_N^\pert(u) = \frac{1}{N} \log \int e^{H_N(\bvx) + \sum_{t =1}^n\sum_{p \geq 1} \sum_{m \geq 1} H^t_{N,p}(\bvx,g_{m,t}, \lambda_{m}, \frac{u_{m,p,t} \epsilon_{N}}{2^{m + p} r C^p}  )} \, d\pP^{\otimes N}_0(\bvx)
\end{equation}
and denote in short $F_N=F_N^\pert(0) $. 
In this section, we also let $\langle \cdot \rangle_\pert$ denote the Gibbs average with respect the perturbed Hamiltonian,
\[
\langle f(\bvx) \rangle_\pert = \frac{\int f(\bvx) e^{H_N^\pert(\bvx)} \, d\pP^{\otimes N}_0(\bvx) }{ \int e^{H_N^\pert(\bvx)} \, d\pP^{\otimes N}_0(\bvx) }
\]
which corresponds to averages with respect to the probability $\pP( \bvx^0 | Y, (Y_\theta)_{\theta \in \Theta} )$ where $Y_\theta$ in an enumeration various $p$-spin Gaussian interference problems introduced in \eqref{eq:inferencepert}. In particular, the Nishimori property is valid for averages with respect to the random Gibbs measure because it corresponds to a conditional probability.

\begin{lem}[Equivalence of the Perturbed Free Energy]\label{prop:nochangeinFE} Assume Hypothesis \ref{hypcompact}. 
Uniformly over all $u_{m,p,t} \in [ 1/2,1 ]$
\[
| \E F_N^\pert(u) - \E F_N| \leq \epsilon_N^2.
\]
In particular, if $\epsilon_N^2 \to 0$, then the perturbation will not change the limit of the free energy.
\end{lem}
Observe that this estimate holds independently of the choice of the original Hamiltonian $H_N$. 
\begin{proof}
Consider the interpolating free energy,
\[
\phi(\tau) = \E\bigg[\frac{1}{N} \log \int e^{H_N(\bvx) + \sum_{p \geq 1} \sum_{m \geq 1} H^t_{N,p}(\bvx,g_{m,t},\lambda_{m}, \frac{u_{m,p,t} \tau }{2^{m + p} r C^p}  )} \, d\pP_0(\bvx)\bigg].
\]
as a function of the $\epsilon_N$ parameter. Notice that $\phi(0) = \E F_N$ and  $\phi(\epsilon_N) = \E F^\pert_N(u)$. A straightforward integration by parts computation and the Nishimori identity \eqref{nishimori} implies that
\begin{align*}
	\phi'(\tau) &= \E \bigg\langle \sum_{t =1}^n \sum_{p \geq 1} \sum_{m \geq 1} \sum_{\iii \in I_t} \frac{u_{m,p,t}}{ C_{m,p} N^{\frac{p-1}{2} + 1}} \sum_{k \leq \kappa} g_{\iii}  \lambda_{m,p}(k) x_{i_1}(k) \cdots x_{i_p}(k) 
	\\&\qquad + 2 \frac{\tau u^2_{m,p,t}}{C^2_{m,p}} \big( \vl_{m,p}^\trans (\rho_t\bR_{1,1}^t)^{\odot p} \vl_{n,p} \big) - \frac{\tau u^2_{m,p,t} }{C^2_{m,p}} \big( \vl_{m,p}^\trans (\rho_t\bR_{1,1}^t)^{\odot p} \vl_{n,p} \big)  \bigg\rangle_\pert
	\\&= \E \bigg\langle \sum_{t =1}^n \sum_{p \geq 1} \sum_{m \geq 1} \frac{tu^2_{m,p,t}}{C^2_{m,p}} \big( \vl_{m,p}^\trans (\rho_t\bR_{1,1}^t)^{\odot p} \vl_{m,p} \big) -  \frac{t u^2_{m,p,t} }{C^2_{m,p}} \big( \vl_{m,p}^\trans (\rho_t\bR_{1,2}^t)^{\odot p} \vl_{m,p} \big) 
	\\&\qquad+ \frac{2\tau u^2_{m,p,t}}{C^2_{m,p}} \big( \vl_{m,p}^\trans (\rho_t\bR_{1,0}^t)^{\odot p} \vl_{m,p} \big) - \frac{\tau u^2_{m,p,t}}{C^2_{m,p}} \big( \vl_{m,p}^\trans (\rho_t\bR_{1,1}^t)^{\odot p} \vl_{m,p} \big)  \bigg\rangle_\pert
	\\&\leq \tau \max_{m,p}(u_{m,p,t}^2).
\end{align*} because the overlaps are bounded uniformly under Hypothesis \ref{hypcompact}.	Therefore, for $t \in [0,\epsilon_N]$, we have $\phi'(\tau) \leq \epsilon_N $ since $\max_{m,p,t}(u_{m,p,t}^2) \leq 1$, so the result follows. 
\end{proof}

On the other hand, if we take $\epsilon_N$ going to zero sufficiently slowly, then we will be able to regularize the Gibbs measure using this perturbation. We will fix a $t \leq n$, and show that the overlaps within $t \leq n$ will concentrate in the limit.

We define
\begin{equation}\label{eq:FEconcentration}
v_N = \sup_{u} \E (N F^\pert_N( u ) - N \E F^\pert_N( u ) )^2
\end{equation}
where the supremum is taken over all $u_{n,p,t}\in [1/2,1]$ and the expectation $\E$ is over the Gaussian variables $g_{m,t,\iii}$ and the $\bx^0$. In our applications, $v_N$ is usually of order $N$ as we will see by using concentration of measure.

We consider the case when the parameters $u_{m,p,t}$ are random. For each $m,p \geq 1$, consider $u_{m,p,t} \epsilon_N $ where $u_{m,p,t} \in [1/2,1]$ are uniform and  independent and $\epsilon^2_N \to 0$, so that the limit of the free energy is unchanged by Lemma~\ref{prop:nochangeinFE}. We will prove that if  $\epsilon_N = N^{-\gamma}$ for $\gamma < \frac{1}{4}$, then the perturbation can be large enough to regularize the Gibbs measure by implying concentration of the quadratic forms of overlaps and the uniform concentration of the log partition function (or the free energies not normalized by $N$) with respect to the perturbed as a function of $u_{m,p,t}$. 
Taking $u_{m,p,t} \sim U[1/2,1]$ independent for all $m,p,t$, we get the following bound for the concentration of the overlaps on average. 
\begin{theo}[Concentration Bound of the Overlap] \label{thm:concoverlap} Assume Hypotheses \ref{hypcompact} and \ref{hypbayes}.
If $\frac{v_N}{N^2 \epsilon_N} \to 0$, then for any $m,p \geq 1$ there exists a constant $L_{m,p}$ that only depends on $m$ and $p$ such that	\begin{equation}\label{eq:concentraitonquadratic}
	\max_{t}\E_u\E \langle ( (\vl_{m}^\trans (\bR_{1,2}^t)^{\odot p} \vl_{m}) - \E \langle (\vl_{m}^\trans (\bR_{1,2}^t)^{\odot p} \vl_{m}) \rangle_\pert )^2  \rangle_\pert \leq  L_{m,p} \bigg( \bigg(\frac{v_N}{N^2 \epsilon^4_N} \bigg)^{1/3} + \frac{1}{\epsilon_N^2 N} \bigg).
\end{equation}
where $\langle \cdot \rangle = \langle \cdot \rangle_\pert$ is the average with respect to the perturbed Gibbs measure. $\E$ denotes the average with respect to the Gaussian random variables that appear in the  Hamiltonian and the signal variable $\bx^0$. The outer average $\E_{u}$ is with respect to the uniform random variables $(u_{m',p',t'})_{m',p',t'}$.
\end{theo}

\begin{proof}
This proof is a generalization of the case when $p = 1$ and $\kappa = 1$ found in \cite{strongreplicasym}. To simplify notation, we will drop the subscript on the Gibbs average, $\langle \cdot \rangle := \langle \cdot \rangle_\pert$ and the subscripts $m,p,t$ because they are fixed throughout the proof. Furthermore, since the  $(u_{m,p,t})_{m,p,t \geq 1}$ are independent, we can fix the $u_{m',p',t'}$ for $m' \neq n$, $p' \neq p$, and $t \neq t'$ and average with respect to the Gaussian $g$ and $u_{m,p,t}$ first. This restriction will not affect the validity of the Nishimori property \eqref{nishimori} because the Gibbs measure is a conditional probability corresponding to a Gaussian estimation problem for all $u_{m,p,t}$. To also simplify notation, we will abuse notation and define
\[
\bR_{\ell,\ell'}^{\odot p} = (\rho \bR_{\ell,\ell'}^t)^{\odot p}
\]
in this proof because $\rho,p,t$ are fixed, so it will not affect any computations.
\\\\
\textit{Step 1:} We first bound the moments 
and show that
\begin{equation}\label{eq:concstep1}
	\E \langle ( (\vl_{m}^\trans \bR_{1,0}^{\odot p} \vl_{m}) - \E \langle (\vl_{m}^\trans \bR_{1,0}^{\odot p} \vl_{m}) \rangle )^2 \rangle \leq \frac{4C^2_{m,p}}{\epsilon^2_N N^2}  \E \big\langle ( H' - \E \langle H' \rangle)^2 \big\rangle
\end{equation}
where  $H'=\partial_{s} H$ denotes the derivative with respect to the last coordinate of the Hamiltonian defined in \eqref{eq:3paramhamiltonian} which we denote by
\begin{equation}\label{defs}
	s:=s_{m,p,t} = \frac{u_{m,p,t} \epsilon_{N}}{2^{m + p} \kappa C^p} \in \Big[\frac{\epsilon_{N}}{2 \cdot 2^{m + p} \kappa C^p} ,\frac{\epsilon_{N}}{2^{m + p} \kappa  C^p}  \Big].
\end{equation}
We also recall that $C_{m,p} = 2^{m + p} \kappa C^{p}$. During the proof we write in short $u$ for $u_{m,p,t}$.
This inequality comes from an integration by parts argument and the Nishimori property. A crucial observation is that we can integrate by parts with respect to the Gaussian random variables conditionally on all other sources of randomness by independence. By independence, we first do the computation conditionally on $s$ and denote in short $\vl$ for $\vl_{m,p}$. We first prove that
\begin{equation}\label{conc:step1}
	\frac{1}{N} \E \langle  (\vl^\trans \bR_{1,2}^{\odot p} \vl)( H' - \E \langle H' \rangle ) \rangle \leq s\E \langle ( (\vl^\trans \bR_{1,2}^{\odot p} \vl) - \E \langle  (\vl^\trans \bR_{1,2}^{\odot p} \vl) \rangle )^2 \rangle + s\E \langle ( (\vl^\trans \bR_{1,2}^{\odot p} \vl) -  \langle  (\vl^\trans \bR_{1,2}^{\odot p} \vl) \rangle )^2 \rangle.
\end{equation}
Notice that the left hand side simplifies to
\begin{align*}
	&\frac{1}{N} \E \langle (\vl^\trans \bR_{1,0}^{\odot p} \vl)( H' - \E \langle H' \rangle ) \rangle 
	\\&= \frac{1}{ N^{\frac{p+1}{2}} }\sum_{\iii} \E\bigg[ \langle 
	(\vl^\trans \bR_{1,0}^{\odot p} \vl)
	g_{\iii}\sum_k  \lambda(k)  x_{i_1} (k)\cdots x_{i_p}(k)  \rangle + 2s \langle  (\vl^\trans \bR_{1,0}^{\odot p} \vl)  (\vl^\trans \bR_{1,0}^{\odot p} \vl) \rangle - s \langle  (\vl^\trans \bR_{1,0}^{\odot p} \vl)  (\vl^\trans \bR_{1,1}^{\odot p} \vl) \rangle \bigg]
	\\&- \E \Big[ \langle (\vl^\trans \bR_{1,0}^{\odot p} \vl) \rangle \Big]  \bigg( \frac{1}{N^{\frac{p+1}{2}} }\sum_{\iii}  \E \bigg[ \langle g_{\iii}\sum_k \lambda(k) x_{i_1}(k) \cdots x_{i_p}(k)  \rangle + 2s \E \langle (\vl^\trans \bR_{1,0}^{\odot p} \vl) \rangle - s\E \langle  (\vl^\trans \bR_{1,1}^{\odot p} \vl) \rangle \bigg) \bigg].
\end{align*}
Since $\bx^0$ is independent from the Gaussian terms, we can integrate the Gaussian terms using integration by parts conditionally on the $\bx^0$ (and recalling the extra $s$ factor in the exponent) to conclude that the  above right hand side equals
\begin{align}
	&s\E \Big[ \langle (\vl^\trans \bR_{1,0}^{\odot p} \vl) (\vl^\trans \bR_{1,1}^{\odot p} \vl)   \rangle -   \langle (\vl^\trans \bR_{1,0}^{\odot p} \vl) (\vl^\trans \bR_{1,2}^{\odot p} \vl)   \rangle + 2\langle (\vl^\trans \bR_{1,0}^{\odot p} \vl) (\vl^\trans \bR_{1,0}^{\odot p} \vl) \rangle - \langle (\vl^\trans \bR_{1,0}^{\odot p} \vl) (\vl^\trans \bR_{1,1}^{\odot p} \vl) \rangle \Big] \notag
	\\&- \E \Big[ \langle (\vl^\trans \bR_{1,0}^{\odot p} \vl) \rangle \Big] \bigg(  s\E  \bigg[ \langle (\vl^\trans \bR_{1,1}^{\odot p} \vl)   \rangle - s\E \langle (\vl^\trans \bR_{1,2}^{\odot p} \vl)   \rangle + 2s\E \langle (\vl^\trans \bR_{1,0}^{\odot p} \vl) \rangle - s\E \langle (\vl^\trans \bR_{1,1}^{\odot p} \vl) \rangle \bigg) \bigg] \label{eq:concstep1.1}.
\end{align}
Next, we integrate with respect to $\bx^1$ and $\bx^2$ independently since they are independent conditionally on $\bx^0$ and apply the Nishimori property to conclude that
\[
\E \langle (\vl^\trans \bR_{1,0}^{\odot p} \vl) (\vl^\trans \bR_{1,2}^{\odot p} \vl)   \rangle  = \E \langle (\vl^\trans \bR_{1,0}^{\odot p} \vl) (\vl^\trans \bR_{2,0}^{\odot p} \vl)   \rangle  = \E \langle (\vl^\trans \bR_{1,0}^{\odot p} \vl) \rangle^2.
\]
which implies  that \eqref{eq:concstep1.1} can be further simplified to
\begin{align*}
	&- s\E \langle (\vl^\trans \bR_{1,0}^{\odot p} \vl) \rangle^2 +  2 s\E \langle (\vl^\trans \bR_{1,0}^{\odot p} \vl)^2 \rangle- s( \E \langle (\vl^\trans \bR_{1,0}^{\odot p} \vl) \rangle )^2 
	\\&=  s\Big( \E \langle (\vl^\trans \bR_{1,0}^{\odot p} \vl)^2 \rangle - ( \E_\langle (\vl^\trans \bR_{1,0}^{\odot p} \vl) \rangle )^2 \Big) + s\Big( \E \langle (\vl^\trans \bR_{1,0}^{\odot p} \vl)^2 \rangle - \E \langle (\vl^\trans \bR_{1,0}^{\odot p} \vl) \rangle^2 \Big) 
	\\&= s\E \langle ((\vl^\trans \bR_{1,0}^{\odot p} \vl) - \E \langle (\vl^\trans \bR_{1,0}^{\odot p} \vl) \rangle )^2 \rangle + s\E \langle ((\vl^\trans \bR_{1,0}^{\odot p} \vl) -  \langle (\vl^\trans \bR_{1,0}^{\odot p} \vl) \rangle )^2 \rangle\,.
\end{align*}
We can now conclude that
\[
\frac{1}{N} \E \langle ((\vl^\trans \bR_{1,0}^{\odot p} \vl) - \E \langle (\vl^\trans \bR_{1,0}^{\odot p} \vl) \rangle ) ( H' - \E \langle H' \rangle ) \rangle \geq s \E \langle ((\vl^\trans \bR_{1,0}^{\odot p} \vl) - \E \langle (\vl^\trans \bR_{1,0}^{\odot p} \vl) \rangle )^2 \rangle
\]
so the Cauchy--Schwarz inequality implies
\begin{align*}
	s\E \langle ((\vl^\trans \bR_{1,0}^{\odot p} \vl)  - \E \langle (\vl^\trans \bR_{1,0}^{\odot p} \vl)  \rangle )^2 \rangle &\leq 	\frac{1}{N} | \E \langle ((\vl^\trans \bR_{1,0}^{\odot p} \vl)  - \E \langle (\vl^\trans \bR_{1,0}^{\odot p} \vl)  \rangle ) ( H' - \E \langle H' \rangle ) \rangle | 
	\\&\leq \frac{1}{N} \Big( \E \langle ((\vl^\trans \bR_{1,0}^{\odot p} \vl)  - \E \langle (\vl^\trans \bR_{1,0}^{\odot p} \vl)  \rangle )^2 \rangle  \E \langle ( H' - \E \langle H' \rangle )^2 \rangle \Big)^{1/2},
\end{align*}
which simplifies to
\[
s^2\E \langle ( (\vl^\trans \bR_{1,0}^{\odot p} \vl)  - \E \langle (\vl^\trans \bR_{1,0}^{\odot p} \vl) \rangle )^2 \rangle \leq \frac{1}{N^2}  \E \big\langle ( H' - \E \langle H' \rangle)^2 \big\rangle.
\]
Using the fact that $s \geq \frac{\epsilon_N}{2 C_{m,p}}$ by \eqref{defs} implies \eqref{eq:concstep1}.
\\\\
\textit{Step 2:} We now have to control the variance of the derivative of the Hamiltonian. Again, this part of the proof is identical to the one dimensional case, because the variances of $H$ can be expressed in terms of the derivatives of the free energy, which can be bounded using elementary facts about convex functions \cite{strongreplicasym}. We recreate the details for completeness.

We begin by bounding the difference of derivatives in \eqref{eq:concstep1} with its thermal and quenched variances 
\begin{equation}\label{toto}
	\frac{4C^2_{m,p}}{\epsilon^2_N N^2}  \E \big\langle ( H' - \E \langle H' \rangle)^2 \big\rangle \leq
	\frac{4C^2_{m,p}}{\epsilon^2_N N^2}  \E \big\langle ( H' -  \langle H' \rangle)^2 \big\rangle + \frac{4C^2_{m,p}}{\epsilon^2_N N^2}  \E \big\langle ( \langle H' \rangle  - \E \langle H' \rangle)^2 \big\rangle 
\end{equation}
Our goal is to use the convexity and boundedness of the perturbed free energy functions to  bound the right hand side of \eqref{toto} from above.

We first fix the rest of the uniform random variables in the Hamiltonian except for $u_{m,p}$. We begin by bounding the first term in the RHS of \eqref{toto} using a bound on the first derivative of the free energy. We prove that
\begin{equation}
	\frac{4C^2_{m,p}}{\epsilon^2_N N^2} \E_{u} \E \big\langle ( H' -  \langle H' \rangle)^2 \big\rangle \leq \frac{12 C^4_{n,p} 4^{n + p} }{\epsilon_N^2 N}.
\end{equation}
To this end, first  notice that
\begin{align}
	&\frac{d^2}{d s^2}\E \log Z_N^\pert(s)  =  \E \langle (H' - \langle H' \rangle )^2 \rangle + 2N \E \langle \vl \bR_{1,0}^{\odot p} \vl \rangle - N \E \langle \vl \bR_{1,1}^{\odot p} \vl \rangle \label{eq:secondderiv}
\end{align}
where the last term comes from $\partial_{s}^{2}H$.
If we rearrange terms, then 
\begin{align*}
	&\int_{\frac{\epsilon_N}{2C_{n,p}}}^{\frac{\epsilon_N}{C_{m,p}}}  \E \big\langle ( H' -  \langle H' \rangle)^2 \big\rangle \, ds
	\\&= \int_{-\frac{\epsilon_N}{2C_{m,p}}}^{\frac{\epsilon_N}{C_{m,p}}}   \frac{d^2}{d s^2} \E \log Z_N^\pert(s)  \, ds -  N \int_{\frac{\epsilon_N}{2C_{m,p}}}^{\frac{\epsilon_N}{C_{m,p}}}  \E \bigg\langle 2 \vl^\trans \bR^{\odot p}_{1,0} \vl -  \vl^\trans \bR^{\odot p}_{1,1} \vl \bigg\rangle  \, ds 
	\\&\leq \frac{d}{d s} \E \log Z_N^\pert(s) \bigg|_{\frac{\epsilon_N}{2C_{m,p}}}^{\frac{\epsilon_N}{C_{m,p}}} + 3 N \int_{\frac{\epsilon_N}{2C_{m,p}}}^{\frac{\epsilon_N}{C_{m,p}}} \kappa^2 C^{2p} \, ds.
\end{align*} where we used \eqref{boundO}.
Next, to control the first term, notice that
\[
\frac{d}{d s} \E \log Z_N^\pert(s) = \E \langle H' \rangle = 2sN \E \langle \vl^\trans \bR^{\odot p}_{1,0} \vl \rangle - s N \E \langle \vl^\trans \bR^{\odot p}_{1,2} \vl \rangle\leq 3 N s \kappa^2 C^{2p}
\]
so we can conclude that
\[
\int_{\epsilon_N/2C_{m,p}}^{\epsilon_N/C_{m,p}}  \E \big\langle ( H' -  \langle H' \rangle)^2 \big\rangle \, ds  \leq 3 N  \kappa^2 C^{2p}\frac{\epsilon_N}{C_{m,p}} 
\]
Since $s = \frac{u \epsilon_N}{C_{n,p}}$ where $u \sim U[1/2,1]$, the above reads
\begin{equation}\label{b1}
	\E_u \E \big\langle ( H' -  \langle H' \rangle)^2 \big\rangle   \leq 6 N  \kappa^2 C^{2p}
\end{equation}
Eventhough this bound is not great, remember that we will multiply  it by $\frac{4C^2_{n,p}}{\epsilon^2_N N^2}$ which is very small, see \eqref{eq:concstep1}.

\textit{Step 3:} We now bound the second term  in the RHS of \eqref{toto} and show that
\begin{equation}\label{bnvc}
	\frac{4C^2_{m,p}}{\epsilon^2_N N^2}  \E_{u} \E \big\langle ( \langle H' \rangle  - \E \langle H' \rangle)^2 \big\rangle  \leq  1000 C_{m,p}^{\frac{16}{3}} 4^{m + p} \Big( \frac{v_N}{N^2 \epsilon^4_N} \Big) ^{1/3} .
\end{equation}
The proof uses the expectation over $u_{n,p}$ and  the convexity of the modified free energy
\[
\tilde F_N(s) = \frac{1}{N} \log Z^\pert_N(s)  + 3s^2 4^{m + p}C_{m,p}^2\,.
\]
$\tilde F_N$ is a convex function because the computation in \eqref{eq:secondderiv} implies that the second derivative 
\begin{align*}
	\frac{d^2}{ds^2}\tilde F_N &= \frac{1}{N} \langle (H' - \langle H' \rangle)^2 \rangle  + 2  \langle \vl \bR_{1,0}^{\odot p} \vl \rangle -  \langle \vl \bR_{1,1}^{\odot p} \vl \rangle  + 6 \cdot 4^{m + p}C_{m,p}^2
\end{align*}
is non-negative because $C_{m,p} = 2^{m + p} \kappa C^{p}$ and $\langle \vl \bR_{1,2}^{\odot p} \vl \rangle$ is uniformly bounded by $\kappa^2 C^{2p}$. Our goal is to control the difference of the derivatives of this convex function with its expected value
\begin{equation}\label{eq:concentrationestimate}
	\tilde F'_N - \E \tilde F'_N = \frac{1}{N} ( \langle H' \rangle - \E \langle H' \rangle ) .
\end{equation}
We can now use the bounds on the derivatives of convex functions given by the following lemma from \cite[Lemma 3.2]{Pbook}.
\begin{lem}[A Bound for Convex Functions]
	Let  $G$ and $g$ be convex differentiable functions. For $\delta > 0$ and nonnegative functions $C_\delta^-(x) = g'(x) - g'(x - \delta)$ and $C_\delta^+(x) = g'(x + \delta) - g'(x)$ then for all real numbers $x$
	\[
	|G'(x) - g'(x)| =  \frac{1}{\delta} \sum_{u \in \{x - \delta,x,x + \delta \}} |G(u) - g(u)| + C_\delta^+(x)  + C_\delta^-(x) 
	\]
\end{lem}
To control \eqref{eq:concentrationestimate}, we apply this lemma to  $G(s) = \tilde F_N(s)$ and $g(s) = \E \tilde F_N(s)$ to conclude that for any $s \in (\epsilon_N/2 C_{m,p}, \epsilon_N/ C_{m,p})$ and $\delta > 0$ sufficiently small so that $s - \delta > 0$,
\begin{align*}
	\frac{1}{N} | \langle H' \rangle - \E \langle H' \rangle  | &\leq \delta^{-1} \sum_{u \in \{x - \delta,x,x + \delta \}} |F_N(u) - \E F_N(u)| + |\E \tilde F_N'(x) - \E \tilde F_N'(x - \delta) | + |\E \tilde F_N'(x + \delta) - \E \tilde F_N'(x) | 
\end{align*}
We square both sides and apply Jensen's inequality $(\sum_{i = 1}^m a_i )^2 \leq m \sum_{i = 1}^m a_i^2$ to arrive at the bound
\begin{align}
	\frac{1}{5 N^2} \E_{u} \E ( \langle H' \rangle - \E \langle H' \rangle )^2 &\leq \delta^{-2} \sum_{u \in \{x - \delta,x,x + \delta \}} \E_{u} \E |F_N(u) - \E F_N(u)|^2\nonumber
	\\&\quad + \E_{u} |\E \tilde F_N'(s) - \E \tilde F_N'(s - \delta) |^2 + \E_{u} | \E \tilde F_N'(s + \delta) - \E \tilde F_N'(s) |^2 \label{jk}
\end{align}
It now remains to control the two terms in the above upper bound.
\begin{enumerate}
	\item From the definition of $v_N$,
	\[
	\sup_u \E (F_N(u) - \E F_N(u))^2 \leq \frac{v_N}{N^2},
	\]
	we get
	\begin{equation}\label{bnv}
		\delta^{-2} \sum_{u \in \{x - \delta,x,x + \delta \}} \E_{u} \E |F_N(u) - \E F_N(u)|^2\leq \frac{3 v_N}{ N^2 \delta^2}.
	\end{equation}
	\item For the two last  terms, we begin by controlling
	\begin{align}
		&2 \int_{\frac{1}{2}}^1\Bigg( \Big( \E \tilde F_N'\Big( \frac{u \epsilon_N}{C_{m,p}} \Big) - \E \tilde F_N'\Big(\frac{u \epsilon_N}{C_{m,p}}  - \delta \Big)\Big)^2 + \Big( \E \tilde F_N'\Big( \frac{u \epsilon_N}{C_{m,p}} \Big) - \E \tilde F_N'\Big(\frac{u \epsilon_N}{C_{m,p}} +\delta\Big)\Big)^2 \,\Bigg) du \notag
		\\&= \frac{2C_{m,p}}{\epsilon_N} \int_{\frac{\epsilon_N}{2C_{m,p}}}^{\frac{\epsilon_N}{C_{m,p}}} \Bigg(| \E \tilde F_N'(s) - \E \tilde F_N'(s - \delta) |^2  + | \E \tilde F_N'(s + \delta) - \E \tilde F_N'(s) |^2\Bigg) \, ds \label{eq:2b2}
	\end{align}
	By the Nishimori property, we get the uniform bound
	\begin{equation}\label{eq:derivbound}
		\E \tilde F_N'(s) = \frac{1}{N} \E \langle H' \rangle  = 2s \E \langle \vl^\trans \bR^{\odot p}_{1,0} \vl \rangle - s  \E \langle \vl^\trans \bR^{\odot p}_{1,2} \vl \rangle\leq 3 \epsilon_N C^2_{m,p} 4^{m + p}
	\end{equation}
	so
	\[
	\max ( | \E \tilde F_N'(s) - \E \tilde F_N'(s - \delta) |, | \E \tilde F_N'(s + \delta) - \E \tilde F_N'(s) | ) \leq 6 \epsilon_N C^2_{m,p} 4^{m + p}
	\]
	Since $F_N'$ is increasing, $\E \tilde F'_N(s) - \E F_N'(s - \delta)$ is nonnegative so we conclude that \eqref{eq:2b2} is bounded by
	\begin{align*}
		&12  C^3_{m,p} 4^{m + p} \int_{\frac{\epsilon_N}{2C_{m,p}}}^{\frac{\epsilon_N}{C_{m,p}}}  \E \tilde F'_N(s) - \E \tilde F'_N(s - \delta)    +  \E \tilde F'_N(s + \delta) - \E \tilde F'_N(s) \, ds
		\\&= 12  C^3_{m,p} 4^{m + p} \bigg( \E \tilde F_N \Big(\frac{\epsilon_N}{C_{m,p}} + \delta \Big) - \E \tilde F_N \Big(\frac{\epsilon_N}{C_{m,p}} - \delta \Big)  +  \E \tilde F_N \Big(\frac{\epsilon_N}{2C_{m,p}} - \delta \Big) - \E \tilde F_N \Big(\frac{\epsilon_N}{2C_{m,p}} + \delta \Big) \bigg)
		\\&\leq 72 \delta \epsilon_N  C^5_{m,p} 8^{m + p} 
	\end{align*}
	by the mean value theorem and the bound on the derivatives \eqref{eq:derivbound}.
\end{enumerate} 
We deduce from \eqref{jk} and \eqref{bnv} that
\[
\frac{1}{5 N^2} \E_{u} \E ( \langle H' \rangle - \E \langle H' \rangle )^2 \leq \frac{3 v_N}{ N^2 \delta^2} + 72 \delta \epsilon_N  C^5_{n,p} 8^{n + p} .
\]
We finally choose $\delta$ to be given by
\[
\delta^3 = \frac{v_N}{ 24 N^2 \epsilon_N C_{n,p}^5	8^{n + p}}
\]
so
\[
\frac{1}{N^2} \E_{u} \E ( \langle H' \rangle - \E \langle H' \rangle )^2 \leq 2 \bigg( \frac{3v_N}{N^2} \bigg( \frac{ 24 N^2\epsilon_N C_{m,p}^5	8^{m + p}}{v_N } \bigg)^{2/3} + 72  \bigg( \frac{ v_N}{ 24 N^2 \epsilon_N C_{m,p}^5	8^{m + p}} \bigg)^{1/3} \epsilon_N  C^5_{m,p} 8^{m + p} \bigg)
\]
which gives \eqref{bnvc} . 

\textit{Step 4:} From steps 1 and 3, we conclude that conditionally on $u_{m',p'}$ for $(m',p') \neq (m,p)$ that
\[
\E_{u_{m,p}} \E \langle ( (\vl_{m,p}^\trans \bR_{1,2}^{\odot p} \vl_{m,p}) - \E \langle (\vl_{n,p}^\trans \bR_{1,2}^{\odot p} \vl_{m,p}) \rangle )^2 \rangle \leq   \bigg(  1000 C_{n,p}^{\frac{16}{3}} 4^{n + p} \Big( \frac{v_N}{N^2 \epsilon^4_N} \Big) ^{1/3}  + \frac{12 C^4_{n,p} 4^{m + p} }{\epsilon_N^2 N} \bigg).
\]
Since this bound is uniform in  $u_{m',p'}$, Theorem \ref{thm:concoverlap} follows. 

\end{proof}

\begin{rem}
In the proof above, we assumed that $C_{m,p} \geq 1$ and $\epsilon_N \leq 1$. This is not a problem, because $\epsilon_N = N^{-\gamma}$ for $0 < \gamma < 1/4$ in our applications, and we can assume $C_{m,p} \geq 1$ at the cost of a less sharp upper bound. The constant $L_{m,p}$ is also of order $C_{m,p}^{16/3}$. 
\end{rem}

Using Gaussian concentration to explicitly estimate $v_N$, we can conclude concentration of all quadratic forms associated to the Hadamard powers of the overlap at rational vectors.
\begin{cor}[Overlap Quadratic Form  Concentration]\label{OC1}
If $\epsilon_N = N^{- \gamma}$ for $0 < \gamma < 1/4$, then for any $p \geq 1$ rational valued $\vl \in [-1,1]^\kappa$, and $t \leq n$ 
\begin{equation}\label{eq:concentrationonaverage}
	\E_u\E \langle ( (\vl^\trans (\bR^t_{1,2})^{\odot p} \vl) - \E \langle (\vl^\trans (\bR^t_{1,2})^{\odot p} \vl) \rangle_\pert )^2  \rangle_\pert \to 0.
\end{equation}
\end{cor}

\begin{proof}
We explicitly compute the rate $v_N$ defined in \eqref{eq:FEconcentration} for this model. We fix our parameter $u$, and let $\phi = N F_N^\pert(u)$. By independence, we can split the expected values into a statement about the concentration of $\bx^0$ and the Gaussian terms,
\begin{equation}\label{eq:decompositionindicator}
	\E  | \phi - \E \phi| \leq \E  | \phi - \E_{W,g}  \phi| + \E_{\bx^0}  | \E_{W,g} \phi - \E\phi|
\end{equation}
The average $\E_W$ is with respect to the `$W$' Gaussian terms in $H_N$, the average $\E_{\bx^0}$ is with respect to `$\bx^0$' terms in the approximate indicator and $\E_g$ is with respect to the `$g$' Gaussian terms $g_N(\bx)$, and $\E$ is the average with respect to all sources of randomness.

We start by proving the first term satisfies
\begin{equation}\label{eq:firsttermGGIconc}
	\E | \phi - \E_{W,g} \phi| \leq O( \sqrt{N + N \epsilon_N^2} ).
\end{equation}
By independence, we can compute this upper bound conditionally on $\bx^0$. Since the Gaussian Hamiltonian has variance of order $N$, the classical Gaussian concentration inequality for Lipschitz functions \cite[Theorem 1.2]{Pbook} implies that 
\[
\E | \phi - \E_{W,g} \phi|^2
\leq 8\gamma^2 C^{2p} N + N \epsilon_N^2.
\]

We now focus on the second term and prove it satisfies
\begin{equation}\label{eq:secondtermGGIconc}
	\E_{\bx^0}  | \E_{W,g}  \phi - \E  \phi|  \leq  O( N^{\frac{1}{2}} ).
\end{equation}
We use the bounded difference property and a consequence of the Efron--Stein inequality. For $N$ sufficiently large
\begin{align*}
	|\E_{W,g} \partial_{x_i^0(k)} \phi | &\leq \frac{1}{\|\bD_\infty\|} \Big| \E_{W,g} \Big\langle  x_j^0(k) (x_i \cdot x_j) + \sum_{p \geq 1} \frac{p}{C_{m,p}} R_{1,0}^{p-1} x_i(k) \Big\rangle_\pert \Big| 
	\leq L
\end{align*}
for some universal constant $L$ that only depends on $C$ and $\|\bD\|_\infty$.
Since the $\E_{W,g} \phi$ has a bounded derivative for each coordinate $x_i^0(k)$ and $x_i^0(k)$ is almost surely bounded by $C$, it satisfies the bounded difference inequality, 
\[
|\E_{W,g} \phi(x^0_1(1),\dots,x^0_i(k),\dots,x_N^0(\kappa) ) - \E_{W,g} \phi(x^0_1(1),\dots,\tilde x^0_i(k),\dots,x_N^0(\kappa) ) | \leq  CL
\]
so $\phi$ satisfies the bounded difference property for some universal constant $L$ that only depends on the maximal value of the support. From the concentration inequality for functions of independent random variables satisfying the bounded differences property \cite[Corollary~3.2]{boucheronConc} if follows that
\[
\E_{x^0} ( \E_{W,g} \phi - \E \phi )^2 \leq  C^2 L^2 \kappa N.
\]
Combining the estimates in \eqref{eq:firsttermGGIconc} and \eqref{eq:secondtermGGIconc}  proves \eqref{eq:secondtermGGIconc} after applying Jensen's inequality.

From the overlap concentration estimate Lemma~\ref{thm:concoverlap}, we deduce
\begin{equation}
	\E_u	\E \langle ( (\vl_{m,p}^\trans (\bR^t_{1,2})^{\odot p} \vl_{m,p}) - \E \langle (\vl_{m,p}^\trans (\bR^t_{1,2})^{\odot p} \vl_{m,p}) \rangle_\pert )^2  \rangle_\pert \leq L_{m,p} \bigg( \bigg(\frac{ N + N \epsilon_N^2}{N^2 \epsilon^4_N} \bigg)^{1/3} + \frac{1}{\epsilon_N^2 N} \bigg).
\end{equation}
In particular, if $\epsilon_N = N^{- \gamma}$ for $0 < \gamma < 1/4$, then $N \epsilon^4_N \to \infty$ so \eqref{eq:concentrationonaverage} holds for any $m,p$ and $t$. Since the $\vl_{m,p}$ are dense in the sphere, the result follows. 
\end{proof}

\begin{rem}
The condition that $\gamma > 0$ implies that the limit of the free energy is unchanged by Lemma~\ref{prop:nochangeinFE}. The condition that $\gamma < 1/4$ is required to ensure that the perturbation is large enough to regularize the Gibbs measure. The exponent 1/4 is not expected to be optimal.
\end{rem}

The concentration of the quadratic forms in Corollary \ref{OC1} is insufficient to determine concentration of the overlaps (see Corollary \ref{OC3}) because the limiting arrays may not be a priori symmetric.  Indeed, even though the convergence in Corollary \ref{OC1} can be turned into an almost sure  uniform convergence in $\lambda$ 
\[
(\vl^\trans \bR^t_{1,2} \vl) - \E\langle ( \vl^\trans \E\bR^t_{1,2} \vl) \rangle_\pert \to 0 \qquad a.s.
\]
%$$\sup_{\|\v|\|=1}|\langle \v|,  (\bR_{1,2}^{\odot p}-\E \bR_{1,2}^{\odot p})\v|\langle \to 0\qquad a.s.$$
this allows to conclude only that 
\[
((\bR^t_{1,2})^{\odot p}-\E \langle (\bR^t_{1,2})^{\odot p} \rangle_\pert +((\bR_{1,2}^t)^{\odot p}-\E \langle (\bR^t_{1,2})^{\odot p}\rangle_\pert )^{\trans}
\]
goes to zero. It is not hard to see that $\E\langle (\bR^t_{1,2})^{\odot p} \rangle_\pert$ is symmetric, but at  this point nothing guarantees  that $(\bR^t_{1,2})^{\odot p}$ is as well symmetric.
To prove asymptotic symmetry of the overlaps, we invoke the synchronization property first observed for overlap matrix arrays in the vector spin glass models. 

The entries of the overlap matrix are bounded, the matrix of overlaps is tight and there exists a subsequence such that $\bR_N^t = (R^t(\bvx^\ell, \bvx^{\ell'}) )_{\ell,\ell' \geq 1} $ converges in distribution to some matrix $\bar \bR^t = (\bar R(\bvx^\ell, \bvx^{\ell'}) )_{\ell,\ell' \geq 1} $.  The concentration of the quadratic forms \eqref{eq:concentrationonaverage} proves that $(\vl_k^\trans ( \bR^t_{\ell,\ell'})^{\odot p} \vl_k)$ is close to $\E\langle (\vl_k^\trans ( \bR^t_{\ell,\ell'} )^{\odot p} \vl_k) \rangle_\pert$ for $\ell \neq \ell'$ with probability going to one in the limit. This means that the off diagonal entries of the array of quadratic forms are constant
\begin{equation}\label{selfav}
( (\vl_k^\trans ( \bR^t_{\ell,\ell'})^{\odot p} \vl_k) )_{\ell \neq \ell'} = L^{t,p}_k
\end{equation}
for some function $L$ that only depends on $t$, $p$ and $\vl_k$. The fact that the off diagonals of the quadratic forms are constant is sufficient to proving that the limiting array $\bar \bR$ must have symmetric entries. In particular, $\bar\bR$ satisfies the synchronization property proved for vector spin models in \cite[Theorem 4]{PVS}. 
\begin{theo}[Synchronization]\label{synch}
Any infinite array  $(\bar\bR^t_{\ell,\ell'})_{\ell,\ell'\ge 1}$  of $\kappa\times \kappa $ matrices that satisfy \eqref{eq:concentrationonaverage}  implies that $\bar\bR^t_{\ell,\ell'}$ is  almost surely symmetric for all $\ell,\ell' \ge 1$.
\end{theo}

\begin{proof}
We adapt the general proof of synchronization to this simpler setting where we have concentation of the quadratic forms of the powers of the overlaps. The proof is essentially identical to the general case \cite[Section~6]{PPotts}, but avoids relying on the Ghirlanda--Guerra identities because they are trivially satisfied by the replica symmetric array.
\\\\\textit{Step 1:} We first recover the values of the overlaps. By \eqref{selfav}, we observe that there exists  some constants $L^t_{k},L^t_{k,k'}$ etc that depend only on $t \leq n$ and entries $k,k'$ such that :
\begin{enumerate}
	\item{\it  First Order Diagonal Elements:} For all $1 \leq k \leq \kappa$, if we take $p = 1$ and $\lambda = e_k$ then
	\[
	\bar R_{\ell,\ell'}^t(k,k) = L^t_{k,k}
	\]
	for some constant $L_k^t = \E \bar R_{1,2}^t(k,k)$.
	\item {\it First Order Off-Diagonal Elements:} For all $1 \leq k \neq k' \leq \kappa$, if we take $p = 1$ and $\lambda_1 = e_{k} + e_{k'}$ and $\lambda_2 = e_k - e_{k'}$ then \eqref{selfav} implies 
	\[
	\bar R^t_{\ell,\ell'} (k,k) + \bar R^t_{\ell,\ell'}(k',k') + \bar R^t_{\ell,\ell'} (k,k') + \bar R^t_{\ell,\ell'}(k',k) = L^t_{e_{k} + e_{k'}}
	\]
	\[
	\bar R^t_{\ell,\ell'} (k,k) + \bar R^t_{\ell,\ell'}(k',k') - \bar R^t_{\ell,\ell'} (k,k') - \bar R^t_{\ell,\ell'}(k',k) = L^t_{e_{k} - e_{k'}}
	\]
	so
	\begin{equation}\label{eq:offdiag1}
		\bar R^t_{\ell,\ell'} (k,k') + \bar R^t_{\ell,\ell'}(k',k) = L^t_{k,k'} = \frac{1}{2} (  L^t_{e_{k} + e_{k'}} -  L^t_{e_{k} - e_{k'}}  ).
	\end{equation}
	\item {\it Second Order Off-Diagonal Elements:} For $1 \leq k \neq k' \leq k$, if we take $p = 2$ and $\lambda_1 = e_{k} + e_{k'}$ and $\lambda_2 = e_k - e_{k'}$ then \eqref{selfav} gives a $\tilde L^t_{k,k'}$ such that for all $\ell \neq \ell'$
	
	\begin{equation}\label{eq:offdiag2}
		(\bar R^t_{\ell,\ell'} (k,k'))^2 + (\bar R^t_{\ell,\ell'}(k',k))^2 = \tilde L^t_{k,k'}.
	\end{equation}
\end{enumerate} 
From the formulas \eqref{eq:offdiag1} and \eqref{eq:offdiag2} we can explicitly solve the system  to conclude that $\bar R_{\ell,\ell'}(k,k')$ and $\bar R_{\ell,\ell'}(k,k')$ can take one of two possible values
\[
\bar R^t_{\ell,\ell'} (k,k') , \bar R^t_{\ell,\ell'}(k',k) = \frac{L^t_{k,k'} \pm \sqrt{ 2\tilde L^t_{k,k'}  - (L^t_{k,k'} )^2} }{2}.
\]
Note that all the concentration results above only apply for the offdiagonal elements of the overlap array, because Theorem~\ref{thm:concoverlap} only holds when $\ell \neq \ell'$. Obviously $L^t$ and $\tilde L^t$ are symmetric matrices. Moreover, by Corollary \ref{OC1}, the indices does not depend on the choice of $\ell,\ell'$ as the distribution of any limit point should be symmetric in the replicas. 
\\\\
\textit{Step 2:} It remains to show that
\[
\bar R^t_{\ell,\ell'} (k,k') = \bar R^t_{\ell,\ell'}(k',k) = \frac{L^t_{k,k'}}{2}.
\]
We proceed by contradiction by showing that if $\bar R^t_{\ell,\ell'} (k,k') \neq \bar R^t_{\ell,\ell'}(k',k)$ with positive probability, then we can always examine a large enough subarray such that the diagonals are identical. We consider a $2n \times 2n$ array of $2 \times 2$ blocks of the form
\[
\begin{bmatrix}
	\bar\bR_{\ell,\ell'}({k,k}) & \bar\bR_{\ell,\ell'}({k,k'})\\
	\bar\bR_{\ell,\ell'}({k',k}) & \bar\bR_{\ell,\ell'}({k',k'})
\end{bmatrix}_{\ell,\ell'}
\]	
By the computations in step 1, for $\ell \neq \ell'$ the offdiagonal blocks must be of the form
\[
\begin{bmatrix}
	a & b\\
	c & d
\end{bmatrix} \text{ or } \begin{bmatrix}
	a & c\\
	b & d
\end{bmatrix}
\]
where 
\[
a= L^t_{k,k} \quad d = L^t_{k',k'} \quad b = \frac{L^t_{k,k'} + \sqrt{ 2\tilde L^t_{k,k'}  - (L^t_{k,k'} )^2} }{2}  \quad c= \frac{L^t_{k,k'} - \sqrt{ 2\tilde L^t_{k,k'}  - (L^t_{k,k'} )^2} }{2}.
\]
We will show that $b \neq c$ is impossible, by examining the barycenters of the infinite arrays, namely
using the fact that we have $n$, the number of replicas, as large as we wish. For large enough $n$, we can find arbitrary large disjoint set of indices $\ell \in V_1$ and $\ell' \in V_2$ each of cardinality $m$ \cite[Theorem 3]{Tournament} such that the corresponding matrix array is only of the first form
\[
\begin{bmatrix}
	a & b\\
	c & d
\end{bmatrix}.
\]
We now restrict ourselves to indices from $V = V_1 \sqcup V_2$. Because the arrays of each of the entries restricted to $V$ is positive semidefinite, we can find vectors $u_\ell$ and $w_\ell$  in a Hilbert space  such that
\begin{equation}\label{eq:matrixconstant}
	\begin{bmatrix}
		\langle u_\ell,u_{\ell'} \rangle & \langle u_\ell,w_{\ell'} \rangle\\
		\langle w_\ell,u_{\ell'} \rangle & \langle w_\ell,w_{\ell'} \rangle
	\end{bmatrix}_{\ell, \ell' \in V} = \begin{bmatrix}
		\bar\bR_{\ell,\ell'}({k,k}) & \bar\bR_{\ell,\ell'}({k,k'})\\
		\bar\bR_{\ell,\ell'}({k',k}) & \bar\bR_{\ell,\ell'}({k',k'})
	\end{bmatrix}_{\ell,\ell' \in V}.
\end{equation}
Notice that for $\ell \neq \ell' \in V$ we have $\langle u_\ell,u_{\ell'} \rangle = a$ and $\langle w_\ell,w_{\ell'} \rangle = d$. Furthermore, by construction if $\ell \in V_1$ and $\ell' \in V_2$, we see that 
\[
\begin{bmatrix}
	\langle u_\ell,u_{\ell'} \rangle & \langle u_\ell,w_{\ell'} \rangle\\
	\langle w_\ell,u_{\ell'} \rangle & \langle w_\ell,w_{\ell'} \rangle
\end{bmatrix}_{\ell \in V_1, \ell' \in V_2} = 	\begin{bmatrix}
	a & b\\
	c & d
\end{bmatrix}.
\]

For $i = 1,2$, if we examine the barycenters
\[
U_i = \frac{1}{m} \sum_{\ell \in V_i} u_\ell \quad \text{and} \quad W_i = \frac{1}{m} \sum_{\ell \in V_i} w_\ell,
\]
then \eqref{eq:matrixconstant} readily gives
\[
\begin{bmatrix}
	\langle U_1,U_{2} \rangle & \langle U_1,W_{2} \rangle\\
	\langle W_1,U_{2} \rangle & \langle W_1,W_{2} \rangle
\end{bmatrix} = \begin{bmatrix}
	a & b\\
	c & d
\end{bmatrix}.
\]
However, if we look at the differences between the barycenters, \eqref{eq:matrixconstant} implies that

\[
\| U_1 - U_2 \|^2 = \frac{1}{m^2} \bigg\| \sum_{\ell \in V_1} u_\ell - \sum_{\ell \in V_2} u_\ell  \bigg\|^2 \leq \frac{2 C^2 + 2a}{m}
\]
and
\[
\| W_1 - W_2 \|^2 = \frac{1}{m^2} \bigg\| \sum_{\ell \in V_1} w_\ell - \sum_{\ell \in V_2} w_\ell  \bigg\|^2 \leq \frac{2C^2 + 2d}{m}\,.
\]
Indeed, the diagonal entries of the overlap arrays are bounded by some universal constant $C^2$ and the offdiagonals are fixed for $\ell \in V_1$ and $\ell' \in V_2$ by \eqref{eq:matrixconstant}. We used the fact that $\langle u_{\ell},u_{\ell'} \rangle = a$ and $\langle w_\ell, w_{\ell'} \rangle = d$ for any $\ell \neq \ell'$ to cancel off all the offdiagonal terms. 
If we take $m \to \infty$, then $U_1 \approx U_2$ and $W_1 \approx W_2$, so $\langle U_1,W_{2} \rangle = \langle W_1,U_{2} \rangle$ which implies that $b = c$, so the overlap array must be symmetric. 
\end{proof}

In particular,  the limiting matrix overlaps $\bar\bR^t$ are almost surely symmetric. This allows us to conclude the concentration of the overlap entries instead of its quadratic forms as previously found in Theorem~\ref{thm:concoverlap}, see \eqref{eq:concentraitonquadratic}.

\begin{theo}[Concentration of the Overlaps]\label{OC3}
If $\epsilon_N = N^{-\gamma}$ for some $0 < \gamma < 1/4$, then for all $t \leq n$
\begin{enumerate}
	\item 
	\[
	\E \langle \| \bR^t_{1,2} - \E \langle \bR^t_{1,2} \rangle_\pert \|_2^2 \rangle_\pert \to 0. 
	\]
	\item
	\[
	\E \langle \| \bR^t_{1,0} - \E \langle \bR^t_{1,0} \rangle_\pert \|_2^2 \rangle_\pert \to 0. 
	\]
\end{enumerate}
Furthermore,
\[
\bQ_t:= \E \langle \bR^t_{1,2} \rangle_\pert = \E \langle \bR^t_{1,0} \rangle_\pert,
\]
so the limit points are the same. 
\end{theo}

\begin{proof} The first point is a direct consequence of Corollary \ref{OC1} and Theorem \ref{synch}. For the second point we use  the Nishimori property to  get concentration of the overlap with respect to the planted signal from the first case
\[
\E \langle \| \bR_{1,0} - \E \langle \bR_{1,0} \rangle_\pert \|_2^2 \rangle_\pert \to 0. 
\]
and 
\[
\E \langle \bR^t_{1,2} \rangle_\pert = \E \langle \bR^t_{1,0} \rangle_\pert.
\]
\end{proof}

\subsection{Cavity Computations via the Aizenman--Sims--Starr Scheme}

We will now apply the concentration of the overlaps to prove the matching upper bound of the free energy using the Aizenman--Sims--Starr scheme. We have for any $M \geq 1$ 
\[
F(\Delta):=\limsup_{N \to \infty} F_N(\Delta) = \limsup_{N \to \infty} \frac{1}{M} ( (N + M) F_{N + M}(\Delta) - N F_N(\Delta) ).
\]
We can partition the cavity coordinates into groups such that 
\[
I^+ = \{N + 1, \dots , N+M \} = \bigcup_{s \leq n} I_s^+,
\]
and the proportions in each of these groups converge as $M \to \infty$ for every fixed $N$ (at least on average)
\[
\lim_{M \to \infty} \rho_s^{M} = \frac{| I_s^+|}{M} = \rho_s
\]

We decompose the Hamiltonians into its common part and its cavity fields,
\[
H_{N + M}(\bvx, \bvy) = H'_{N}(\bvx) + \sum_{i = 1}^M (\vy_i^\trans \vz_i(\bvx) + \vy_i^\trans m_i(\bvx) + \vy_i^\trans s_i(\bvx) \vy_i)+ o(N) 
\]
and
\[
H_N(\bvx) \stackrel{d}{=} H_{N}'(\bvx) + y(\bvx)
\]
Here, the common Hamiltonian is given by
\[
H_{N}'(\bvx) =  \sum_{1\le i < j\le N} \frac{1}{\sqrt{\Delta_{ij}}\sqrt{N+M}} g_{ij} \vx_i \cdot \vx_j + \sum_{1\le i < j\le N} \frac{1}{\Delta_{ij}(N+M)}  (\vx_i^0 \cdot \vx^0_j) (\vx_i \cdot \vx_j) -\sum_{1\le i < j\le N} \frac{1}{2 \Delta_{ij} (N + M) }  (\vx_i \cdot \vx_j)^2 
\]
The cavity fields in this model are of the form :
\begin{eqnarray*}
\vz_i(\bvx) &=& \sum_{j = 1}^N g_{j,N + 1}\frac{1}{\sqrt{\Delta_{ij}} \sqrt{N + M}}  \vx_j\\
\vm_i(\bvx) &=& \sum_{j = 1}^N  \frac{1}{\Delta_{ij} (N+M)} \vx^0_j \cdot \vx^0_{N + i} \vx_j = \sum_{s \leq n}  \frac{\rho^M_t}{\Delta_{st}} \bR^t_{1,0}\vx^0_{N + i} + O(N^{-1}) = \tilde \bR^i_{1,0} \vx_i^0 +  O(N^{-1}) \qquad i \in I_s^+
\\
s_i(\bvx) &=& -  \sum_{j = 1}^N \frac{1}{2 \Delta_{ij} (N+M) } \vx_j \vx_j^\trans =- \frac{1}{2} \sum_{s \leq n}  \frac{\rho^M_t}{\Delta_{st}} \bR^t_{1,1} + O(N^{-1}) = - \frac{1}{2} \tilde \bR^i_{1,1} +  O(N^{-1}) \qquad i \in I_s^+
\end{eqnarray*}
where we used the notation
\[
\tilde \bR^i_{1,1} = \sum_{s \leq n}  \frac{\rho^M_t}{\Delta_{st}} \bR^t_{1,1} = \sum_{s \leq n}  \frac{\rho^M_t}{\Delta_{st}} \bR^t(\bx, \bx)  \qquad \tilde \bR^i_{1,0} = \sum_{s \leq n}  \frac{\rho^M_t}{\Delta_{st}} \bR^t_{1,0} = \sum_{s \leq n}  \frac{\rho^M_t}{\Delta_{st}} \bR^t(\bx, \bx^0) .
\]
Noting that we can write that the Gaussian variables with variance $1/N$ decompose as $g_{ij}/ \sqrt{(N+M)} +g_{ij}' \sqrt{ M/N(N+M)}$ with independent Gaussian variables $g_{ij}$ and $g'_{ij}$
\begin{align*}
y(\bvx) &=\sum_{1\le i < j\le N} \frac{\sqrt{M}}{\sqrt{\Delta_{ij} }\sqrt{ N(N+M)}}  g'_{ij} \vx_i \cdot \vx_j + \sum_{i < j} \frac{M}{\Delta_{ij} N(N+M)}  (\vx^0_i \cdot \vx^0_j) (\vx_i \cdot \vx_j) -\sum_{i < j}  \frac{M}{2 \Delta_{ij} N(N + M) }  (\vx_i \cdot \vx_j)^2 
\\&= \sum_{i < j} \frac{\sqrt{M}}{\Delta_{ij} N}  g'_{ij} \vx_i \cdot \vx_j + \frac{M}{2} \sum_{st} \frac{\rho^M_s \rho^M_t}{ \Delta_{st}} \langle\bR^s_{1,0} \bR^t_{1,0} \rangle -\frac{M}{4} \sum_{st} \frac{\rho^M_s \rho^M_t}{ \Delta_{st}} \langle\bR^s_{1,1} \bR^t_{1,1} \rangle + O(N^{-1})
\end{align*}
where $g'_{ij}$ is independent of $g_{ij}$.
Notice that the covariance of the Gaussian fields are given by
\begin{equation}\label{eq:covz}
\E \vz_i(\bvx^1) \vz_i(\bvx^2)^\trans =  \sum_{t} \frac{\rho^M_t}{\Delta_{s,t}^2}  \bR^t_{1,2} := \tilde \bR_{1,2}^i \qquad i \in I_s^+
\end{equation}
and
\[
\E \bigg( \bigg( \sum_{i < j} \frac{1}{\sqrt{\Delta_{ij}}\sqrt{(N+M)}} g_{ij} \vx^1_i \cdot \vx^1_j \bigg) \bigg( \sum_{i < j} \frac{1}{\sqrt{\Delta_{ij}}\sqrt{(N+M)}} g_{ij} \vx^2_i \cdot \vx^2_j \bigg) \bigg) = \frac{M}{2} \sum_{s,t} \frac{\rho^M_s \rho^M_t}{\Delta_{st}} \langle \bR_{1,2}^s, \bR_{1,2}^t \rangle.
\]

By adding and subtracting the normalization terms $H_{N}'$, we need to compute
\[\Delta F_{N,M}:=
\frac{1}{M} \bigg( \E \log \bigg\langle \int \exp\bigg( \sum_{i = 1}^M \vy_i^\trans \vz_i(\bvx) + \vy_i^\trans \vm_i(\bvx) + \vy_i^\trans s_i(\bvx) \vy_i  \bigg) \, d\pP_0(\bvy) \bigg\rangle' - \E \log \langle \exp(y(\bvx))  \rangle' \bigg)
\]
where $\langle \cdot \rangle'$ denotes the average with respect to the Gibbs measure $H_{N}'$. Of course, this procedure works as well if we added the perturbation terms in Theorem~\ref{thm:concoverlap}, resulting in 
\begin{equation}\label{eq:pertHamil}
\Delta F^{\pert}_{N,M}=	\frac{1}{M} \bigg( \E \log \bigg\langle \int \exp\bigg( \sum_{i = 1}^M \vy_i^\trans \vz_i(\bvx) + \vy_i^\trans \vm_i(\bvx) + \vy_i^\trans s_i(\bvx) \vy_i  \bigg) \, d\pP_0(\bvy) \bigg\rangle_{\pert} - \E \log \langle \exp(y(\bvx))  \rangle_{\pert} \bigg)+o(1)
\end{equation}
after a straightforward interpolation argument (see for example \cite[Chapter~3.5]{Pbook}). It suffices to compute this quantity in the limit. 

We now recall a modification of a general result \cite[Theorem 1.4]{Pbook} that implies that the functional $\Delta F_{N,M}$ is continuous with respect to the distribution of the off-diagonal elements of the array $(\bR_{\ell,\ell'})_{\ell \neq \ell'}$. We now state precisely what we mean.
\begin{lem}[Continuity with Respect to Overlap Distribution] \label{lem:contfunc}
For every $\epsilon > 0$, there exists a function $F_\epsilon( \bR^n_{\neq} )$ of finitely many elements of  $\bR^n_{\neq} = (R_{\ell,\ell'})_{1 \leq \ell \neq \ell' \leq n}$  such that
\[
| \Delta F^\pert_{N,M} - F_\epsilon ( \bR^n_{\neq} ) | < \epsilon.
\]
\end{lem}

\begin{proof}
Recall that without loss of generality, we can assume that $\pP_0$ is supported on finitely many points. The proof is essentially identical to the proof of in \cite[Theorem 1.4]{Pbook}. The key difference is that the covariance structures of the cavity fields in these models imply that the diagonal self overlap terms cancel, leaving us with a functional of the off diagonal terms of the overlaps. 

By Gaussian concentration, we will be able to restrict the logarithms in  $\Delta F_{N,M}$ to a compact set without introducing a large error. We first define the truncation
\[
f_a(x) = \begin{cases}
	-a & x < -a\\
	x  & -a \leq x \leq a\\
	a  & x > a
\end{cases}
\]
and the corresponding truncated Hamiltonians
\[
H^a_Z(\bvx, \bvy) = f_a\bigg(  \sum_{i = 1}^M \vy_i^\trans \vz_i(\bvx) + \vy_i^\trans \tilde\bR^i_{1,0} \vx^0_i -\frac{1}{2} \vy_i^\trans \tilde\bR^i_{1,1} \vy_i  \bigg)
\]
\[
H^a_Y(\bvx) = f_a(  \sqrt{M} y(\bvx^\ell) + \frac{M}{2} \sum_{st} \frac{\rho^M_s \rho^M_t}{ \Delta_{st}} \langle\bR^s_{1,0} \bR^t_{1,0} \rangle -\frac{M}{4} \sum_{st} \frac{\rho^M_s \rho^M_t}{ \Delta_{st}} \langle\bR^s_{1,1} \bR^t_{1,1} \rangle  ).
\]
Notice that $H^\infty_Z$ and $H_Y^\infty$ are the corresponding untruncated cavity field Hamiltonians appearing in $\Delta F_{N,M}$.
To simplify notation, we use $\langle \cdot \rangle = \langle \cdot \rangle_\pert$. By a standard Gaussian concentration argument, we will prove that there is a constant $c$ such that
\begin{equation}\label{eq:concentrationZ}
	\bigg| \frac{1}{M} \E \log \bigg\langle \int e^{H^\infty_Z(\bvx,\bvy)} \, d\pP_0(\bvy) \bigg\rangle - \frac{1}{M} \E \log \bigg\langle \int e^{H^a_Z(\bvx,\bvy)} \, d\pP_0(\bvy)  \bigg\rangle \bigg| \leq e^{-c a^2}
\end{equation}
and
\begin{equation}\label{eq:concentrationY}
	\bigg| \frac{1}{M} \bigg( \E \log \bigg\langle \int \exp\bigg( \sum_{i = 1}^M \vy_i^\trans \vz_i(\bvx) + \vy_i^\trans \vm_i(\bvx) + \vy_i^\trans s_i(\bvx) \vy_i  \bigg) \, d\pP_0(\bvy) \bigg\rangle -  \frac{1}{M} \E \log \langle e^{H^a_Y(\bvx)} \rangle \bigg| \leq e^{-c a^2}
\end{equation}
for all $a$ sufficiently large. We will defer the proof of this fact to the end of the proof, and first explain the rest of the logic. 

The key observation is that by Weierstrass Theorem, we can approximate the logarithm by a polynomial uniformly on a compact set, so we can approximate $\Delta F_{N,M}$ as a linear combination of the moments of the overlaps. Since $H_Z(x)$ and $H_Y(x)$ is bounded, we can restrict the logarithm to the set $[e^{-a}, e^{a}]$ and approximate it by polynomials. The expectations of these polynomials are of sums of terms of the form
\[
\E \bigg\langle \int e^{H_Z(\bvx,\bvy)} \, d\pP_0(\bvy)  \bigg\rangle^r \quad\text{and}\quad \E \langle e^{H_Y(\bvx)}  \rangle^r = \bigg\langle \E \prod_{\ell \leq r} e^{H_Y(\bvx)}  \bigg\rangle 
\]
We start by computing the first of these terms. These moments can be computed explicitly. Because of \eqref{eq:concentrationZ} and \eqref{eq:concentrationZ} it suffices to compute the moments with respect to the untruncated Hamiltonians. In particular, for any $r > 0$, the moments simplify to
\begin{align*}
	&\quad \E \bigg\langle \int \exp\bigg( \sum_{i = 1}^M \vy_i^\trans \vz_i(\bvx) + \vy_i^\trans \tilde\bR^i_{1,0} \vx^0_i -\frac{1}{2} \vy_i^\trans \tilde\bR^i_{1,1} \vy_i  \bigg) \, d\pP_0(\bvy) \bigg\rangle^r
	\\&=\E \bigg\langle \prod_{\ell \leq r} \int \exp\bigg( \sum_{i = 1}^M (\vy^\ell_i)^\trans \vz_i(\bvx^\ell) + (\vy_i^\ell)^\trans \tilde\bR_{\ell,0} \vx^0_i -\frac{1}{2} (\vy^\ell_i)^\trans \tilde\bR_{\ell,\ell} \vy^\ell_i  \bigg) \, d\pP_0(\bvy^\ell) \bigg\rangle
	\\&= \bigg\langle \E  \prod_{\ell \leq r} \int  \exp\bigg( \sum_{i = 1}^M (\vy^\ell_i)^\trans \vz_i(\bvx^\ell) + (\vy_i^\ell)^\trans \tilde\bR_{\ell,0} \vx^0_i -\frac{1}{2} (\vy^\ell_i)^\trans \tilde\bR_{\ell,\ell} \vy^\ell_i  \bigg) \, d\pP_0(\bvy^\ell) \bigg\rangle
\end{align*}
where $\langle\cdot \rangle = \langle \cdot \rangle_{\pert}$ is the average on with respect to the perturbed Gibbs measure. Taking expectations with respect to $z_i$
\begin{align*}
	&\quad \E \exp\bigg(\sum_{\ell \leq r} \sum_{i = 1}^M \bigg(  (\vy^\ell_i)^\trans \vz_i(\bvx^\ell) + (\vy_i^\ell)^\trans \tilde\bR_{\ell,0} \vx^0_i -\frac{1}{2} (\vy^\ell_i)^\trans \tilde\bR_{\ell,\ell} \vy^\ell_i \bigg) \bigg)
	\\&= \E \exp\bigg(\sum_{\ell \neq \ell'} \sum_{i = 1}^M  \bigg( \frac{1}{2} (\vy^\ell_i)^\trans \tilde \bR^i_{\ell,\ell'} \vy^{\ell'}_i + (\vy^\ell_i)^\trans \tilde \bR_{\ell,0} \vx^0_i \bigg) \bigg) 
\end{align*}
since
\[
\E \bigg(\sum_{\ell \leq r}\sum_{i = 1}^M (\vy^\ell_i)^\trans \vz_i(\bvx^\ell) \bigg) \bigg(\sum_{\ell \leq r} \sum_{i = 1}^M (\vy_i^{\ell'})^\trans \vz_i(\bvx^{\ell'}) \bigg) =  \frac{1}{2} \sum_{\ell, \ell' = 1}^r \sum_{i,j = 1}^M (\vy^\ell_i)^\trans \tilde \bR^i_{\ell,\ell'} \vy^{\ell'}_j.
\]
In particular, these terms only depend on the offdiagonal elements of $\bR_{\ell,\ell'}$ and $\bR_{0,\ell}$, because the diagonal terms canceled off.  A similar computation  works for the second term in $\Delta F_{N,M}$ since again
\begin{align*}
	&\E \exp\bigg( \sum_{\ell \leq r} \sqrt{M} y(\bvx^\ell) + \frac{M}{2} \sum_{st} \frac{\rho^M_s \rho^M_t}{ \Delta_{st}} \langle\bR^s_{1,0} \bR^t_{1,0} \rangle -\frac{M}{4} \sum_{st} \frac{\rho^M_s \rho^M_t}{ \Delta_{st}} \langle\bR^s_{1,1} \bR^t_{1,1} \rangle \bigg)
	\\&=  \E \exp\bigg( \sum_{\ell \neq \ell'} -\frac{M}{4} \sum_{st} \frac{\rho^M_s \rho^M_t}{ \Delta_{st}} \langle\bR^s_{\ell,\ell'} \bR^t_{\ell,\ell'} \rangle + \frac{M}{2} \sum_{st} \frac{\rho^M_s \rho^M_t}{ \Delta_{st}} \Tr(\bR^s_{\ell,0} \bR^t_{\ell,0} ) ) \bigg)
\end{align*}
is a function of the off diagonal elements of the array $(\bR_{\ell,\ell'})_{\ell, \ell' \geq 0}$. Therefore, we can approximate $\Delta F_{N,M}$ with a continuous function of finitely many off diagonal elements of the array. 

All that remains is to prove that the logarithm can be approximated by polynomials by proving the bounds \eqref{eq:concentrationZ} and \eqref{eq:concentrationY}. The proofs of both inequalities follow from the same argument, so we only show the first one. We prove this by showing that
\begin{equation}\label{eq:truncatelog}
	\bigg|\frac{1}{M}\E \log \bigg\langle \int e^{H^\infty_Z (\bvx,\bvy) } \, d\pP(\bvy) \bigg\rangle - \frac{1}{M}\E \log_a \bigg\langle \int e^{H^\infty_Z(\bvx,\bvy)} \, \pP_X(\bvy) \bigg\rangle  \bigg| \leq e^{-c a^2} 
\end{equation}
and
\begin{equation}\label{eq:truncateZham}
	\bigg|\frac{1}{M}\E \log_a \bigg\langle \int e^{H^\infty_Z (\bvx,\bvy) } \, d\pP(\bvy) \bigg\rangle - \frac{1}{M} \E \log_a \bigg\langle \int e^{H^a_Z(\bvx,\bvy)} \, \pP_X(\bvy) \bigg\rangle  \bigg| \leq e^{-c a^2} 
\end{equation}
where $\log_a(x) = f_a( \log (x))$ is the truncated logarithm. We begin by proving the first inequality. Because the support of $\pP_X$ is bounded, the covariance
\[
\E \vz_i(\bvx^1) \vz_i(\bvx^2)^\trans = \tilde \bR^i_{1,2}
\]
defined in \eqref{eq:covz} is almost surely bounded by some universal constant. Let
\[
F_Z = \frac{1}{M}\log \bigg\langle \int e^{H^\infty_Z (\bvx,\bvy) } \, d\pP(\bvy) \bigg\rangle = \frac{1}{M} \log \bigg\langle \int \exp\bigg( \sum_{i = 1}^M \vy_i^\trans \vz_i(\bvx) + \vy_i^\trans \vm_i(\bvx) + \vy_i^\trans s_i(\bvx) \vy_i  \bigg) \, d\pP_0(\bvy) \bigg\rangle.
\]
By Gaussian concentration \cite[Theorem 1.2]{Pbook},
\begin{equation}\label{eq:GaussianconcZ}
	\pP(| F_Z - \E F_Z| \geq a ) \leq 2 e^{-c a^2}
\end{equation}
where $c$ only depends on $\sup_{t \leq n} \sup_{\bvx \in \supp \pP_X} \|\tilde \bR_{1,2}^t\|$. By Jensen's inequality, and the concavity of the logarithm
\[
0 \leq \E F_Z \leq  \frac{1}{M} \log \bigg\langle \E \int \exp\bigg( \sum_{i = 1}^M \vy_i^\trans \vz_i(\bvx) + \vy_i^\trans \vm_i(\bvx) + \vy_i^\trans s_i(\bvx) \vy_i  \bigg) \, d\pP_0(\bvy) \bigg\rangle \leq C
\]
for some constant that only depends on the upper bound of the covariance $\sup_{t \leq n} \sup_{\bvx \in \supp \pP_X} \|\tilde \bR_{1,2}^t\|$. This implies that $\{ |F_Z| \geq a \} \subset \{ |F_Z - \E F_Z| \geq a/2 \}$ for $a \geq C$. The Gaussian concentration inequality \eqref{eq:GaussianconcZ} implies that there exists a constant $c$ such that for all $a$ sufficiently large
\[
\pP( |F_Z| \geq a ) \leq e^{-c a^2}
\]
and the bounds on the tails of a sub-Gaussian random variable \cite[Equation 1.80]{Pbook} implies that
\[
\E( |F_Z| \1( |F_Z| \geq a ) ) \leq e^{-c a^2}
\]
for some universal constant $c$. This proves \eqref{eq:truncatelog}.

To prove \eqref{eq:truncateZham}, notice that $|\log_a(x) - \log_a(y)| \leq e^a |x - y| $ so \eqref{eq:truncateZham} is bounded by
\[
e^a \bigg| \bigg\langle \int e^{H_Z^\infty(\bvx,\bvy)} \, \pP(\bvy) - \int e^{H_Z^a(\bvx,\bvy)} \, \pP(\bvy) \bigg\rangle \bigg| \leq e^a \E \langle | H_Z^\infty(\bvx,\bvy) | \1( |H_Z^\infty(\bvx,\bvy)| \geq a ) \rangle.
\]
Since the $H_Z$ is a Gaussian process, the exponential decay of the tails implies that we can also bound this by $e^{-ca^2}$ for $a$ sufficiently large. 
\end{proof}
\begin{rem}
The fact that the functionals do not depend on the diagonal terms is essential to understanding the limiting behavior, because concentration Corollary~\ref{OC3} only applies to the offdiagonal terms of the array. 
\end{rem}

In the limit, we use synchronization of vector spins and multispecies models to study the limiting behavior of the overlaps under the asymptotic Gibbs measure. In our case, Theorem~\ref{OC3} implies that
\[
\bR_{\ell,\ell'}^s = \bQ^s_{\ell,\ell'} = \bQ^s_{1,0} = \bQ^s \qquad \text{ for all }  s \leq n, \ell \neq \ell'
\]
where $\bQ_{\ell,\ell'}$ is the values of the limiting overlap matrix. One can check that the free energy functional \eqref{eq:pertHamil} in the limit is equivalent to
\begin{align}
&\frac{1}{M} \bigg( \E \log \int \exp\bigg( \sum_{i = 1}^M \vy_i^\trans \tilde\bQ_{i}^{1/2}\vz_i + \vy_i^\trans  \tilde\bQ_i \vx^0_{i}   - \frac{1}{2} \vy_i^\trans  \tilde\bQ_i \vy_i  \bigg) \, d\pP_0(\bvy) \notag
\\&- \E \log \exp\bigg( y + \frac{M}{4} \sum_{st} \frac{\rho^M_s \rho^M_t}{ \Delta_{st}} \langle\bQ_s, \bQ_t \rangle  \bigg)  \label{eq:aymHamil}.
\end{align}
where
\[
\tilde \bQ_i =\tilde Q_s= \sum_{t \leq n} \frac{1}{\Delta_{s,t}} \rho^M_t \bQ_t \qquad i \in I_s.
\]\b
and the Gaussians $\vz$ and $y$ have covariance
\[
\Cov( \vz_i, \vz_j ) = \delta_{j = i} \bQ_i \quad \Var( y) =  \frac{M}{2} \sum_{st} \frac{\rho^M_s \rho^M_t}{ \Delta_{st}} \langle\bQ^s, \bQ^t \rangle  
\]

\begin{lem}[Equivalence of the Functionals of the Overlaps]\label{lem:equiv}
If the overlaps concentrate, then \eqref{eq:pertHamil} and \eqref{eq:aymHamil} can be approximated by the same functionals $F_\epsilon(\bR^n_{\neq})$ of finitely many off diagonal entries of the overlap array.
\end{lem}

\begin{proof} In our setting, we can replace the off diagonal entries of the array $(\bR_{\ell \neq \ell'})$ and $\bQ$, since the elements of the off diagonal array take only one value by concentration. The original cavity fields appear to have some terms that depend on the diagonal terms $\bR_{\ell,\ell}$, but we will show that the moments of the cavity field are independent of this term by cancellation. 

Since $\pP_0$ can be assumed to have finite support by Lemma~\ref{lem:finitesupport}, the computation in the proof of Lemma~\ref{lem:contfunc} implies that for fixed $\vy$ and $\vx^0$ the moments of the first cavity field can be written as a weighted sum of terms of the form
\begin{align*}
	&\E \exp\bigg( \sum_{i = 1}^M (\vy^\ell)_i^\trans \vz_i(\bvx) + (\vy^\ell)_i^\trans \vm_i(\bvx) + (\vy^\ell)_i^\trans s_i(\bvx) (\vy^\ell)_i  \bigg)  \bigg) = \exp\bigg( \sum_{\ell \neq \ell'}^r \sum_{i = 1}^M \frac{1}{2} \vy_i^\trans \tilde  \bR_{\ell,\ell'}^i \vy_j + \vy_i^\trans \tilde \bR^i_{0,\ell} \vx^0_{i}  \bigg)
\end{align*}
where
\[
\tilde \bR_{\ell,\ell'}^i =  \sum_{t} \frac{\rho^M_t}{\Delta_{s,t}}  \bR^t_{1,2} \qquad \text{ for $i \in I_s$}.
\]
When compared to
\begin{align*}
	&\E \exp\bigg( \sum_{\ell = 1}^r \sum_{i = 1}^M (\vy^\ell)_i^\trans \vz_i + (\vy^\ell)_i^\trans \tilde \bQ_i \vx^0_{i}   - \frac{1}{2} (\vy^\ell)_i^\trans \tilde \bQ_i (\vy^\ell)_i  \bigg) 
	\\&= \exp\bigg( \sum_{\ell, \ell' = 1}^r \sum_{i = 1}^M \frac{1}{2} (\vy^\ell)_i^\trans \tilde \bQ_i (\vy^\ell)_i + \vy_i^\trans \tilde \bQ_i \vx^0_{i} - \frac{1}{2} (\vy^\ell)_i^\trans \tilde \bQ_i (\vy^\ell)_i \bigg)
	\\&= \exp\bigg( \sum_{1 \leq \ell \neq \ell' \leq r} \sum_{i = 1}^M \frac{1}{2} (\vy^\ell)_i^\trans \tilde \bQ_i (\vy^\ell)_i + \vy_i^\trans \tilde \bQ_i \vx^0_{i} - \frac{1}{2} (\vy^\ell)_i^\trans \tilde \bQ_i (\vy^\ell)_i \bigg)
\end{align*}
where
\[
\tilde \bQ_i = \sum_{t \leq n} \frac{\rho^M_t}{\Delta_{s,t}}  \bQ^t \qquad i \in I_s
\]
we can conclude that the moments are functionals of the same values of the overlap. 

For the second cavity field, it is easy to see that all moments are the same of the overlaps. A direct computation of the moments using the moment generating function of a Gaussian gives
\begin{align*}
	\E \Big( \exp(y(\bvx)) \Big)^r &= \exp\bigg( \sum_{\ell =1}^r y(\bvx^\ell) + \frac{M}{2} \sum_{st} \frac{\rho^M_s \rho^M_t}{ \Delta_{st}} \langle\bR_{\ell,0}^s, \bR_{\ell,0}^t \rangle - \frac{M}{4} \sum_{st} \frac{\rho^M_s \rho^M_t}{ \Delta_{st}} \langle\bR_{\ell,\ell}^s, \bR_{\ell,\ell}^t \rangle  \bigg)
	\\&= \exp\bigg( \sum_{1 \leq \ell \neq \ell' \leq r} \frac{M}{4} \sum_{st} \frac{\rho^M_s \rho^M_t}{ \Delta_{st}} \langle\bR_{\ell,\ell'}^s, \bR_{\ell,\ell'}^t \rangle  + \frac{M}{2} \sum_{st} \frac{\rho^M_s \rho^M_t}{ \Delta_{st}} \langle\bR_{\ell,0}^s, \bR_{\ell,0}^t \rangle   \bigg).
\end{align*}
and when compared to
\begin{align*}
	&\E \exp\bigg( \sum_{\ell = 1}^r  y + \frac{M}{4} \sum_{st} \frac{\rho^M_s \rho^M_t}{ \Delta_{st}} \langle\bQ^s, \bQ^t \rangle \bigg) 
	\\&= \exp\bigg( \sum_{1 \leq \ell \neq \ell' \leq r} \frac{M}{8} \sum_{st} \frac{\rho^M_s \rho^M_t}{ \Delta_{st}} \langle\bQ^s, \bQ^t \rangle  + \frac{M}{4} \sum_{st} \frac{\rho^M_s \rho^M_t}{ \Delta_{st}} \langle\bQ^s, \bQ^t \rangle  \bigg)
\end{align*}
so both functionals are the same functions of the overlaps $(\bR_{\ell,\ell'})_{\ell \neq \ell' \geq 0}$ and $\bQ$.

\end{proof}

By overlap concentration Theorem~\ref{OC3}, we know $(R^t_{\ell,\ell'})_{\ell\neq \ell'} \to  (\bQ^t)_{\ell \neq \ell'}$ for some constant matrix $\bQ^t$ in the limit. Lemma~\ref{lem:contfunc} implies that the convergence of the offidagonal elements of the array is sufficient and Lemma~\ref{lem:equiv} implies that the functional \eqref{eq:aymHamil} characterizes the limiting behavior of $\Delta F_{N,M}$ defined in \eqref{eq:pertHamil}. That is,
\begin{align}
&\lim_{N \to \infty} \Delta F^\pert_{N,M} = F_\epsilon(\bR^n_{\neq}) + O(\epsilon) \notag
\\&= \frac{1}{M} \bigg( \E \log \int \exp\bigg( \sum_{i = 1}^M \vy_i^\trans \notag \tilde\bQ_{i}^{1/2}\vz_i + \vy_i^\trans  \tilde\bQ_i \vx^0_{i}   - \frac{1}{2} \vy_i^\trans  \tilde\bQ_i \vy_i  \bigg) \, d\pP_0(\bvy) - \E \log \exp\bigg( y + \frac{M}{4} \sum_{st} \frac{\rho^M_s \rho^M_t}{ \Delta^2_{st}} \langle\bQ^s, \bQ^t \rangle  \bigg)  + O(\epsilon) \notag
\\&= \frac{1}{M} \bigg( \sum_{i \in I_s^+} \E \log \int \exp\bigg(\vy_i^\trans \tilde\bQ_{s}^{1/2}\vz_s + \vy_i^\trans  \tilde\bQ_s \vx^0_{i}   - \frac{1}{2} \vy_i^\trans  \tilde\bQ_s \vy_i  \bigg) \, d\pP_0(\bvy) - \frac{M}{4} \sum_{st} \frac{\rho^M_s \rho^M_t}{ \Delta^2_{st}} \langle\bQ^s, \bQ^t \rangle  \bigg) + O(\epsilon)\label{eq:lowboundcavity}
\end{align}
The parameters $Q_s = Q_s^M$ appearing above depend on our choice of $M$. 

The lower bound \eqref{eq:lowboundcavity} holds for all $M$. By compactness, there exists a subsequence such that $Q_s^M$ converges to a limiting object $Q$. We may take $M \to \infty$ along a subsequence such that the proportions converge $\frac{|I^+_s|}{M} \to \rho_s$ and $Q_s^M$ converges, then $\epsilon \to 0$ and use the continuity of our functional in $Q$ to conclude that
\[
\lim_{M \to \infty} \lim_{N \to \infty} \Delta F_{N,M} = \varphi( \bQ )
\]
for some sequence $Q = (\bQ^s)_{s \leq n}$. This gives  the matching upper bound.

\begin{theo}[Bayes Optimal Upper Bound of the Free Energy]
We have
\[
\limsup_{N \to \infty} \frac{1}{N} \E_Y \ln Z_X(Y) = \limsup_{N \to \infty} \frac{1}{N} \E_Y \ln Z^\pert_X(Y) \leq \sup_{\bQ}\phi( \bQ ).
\]
\end{theo}

\section{Solving the Variational Problem} \label{sec:phase}

In this section, we examine the stability of the critical points of the functional describing the limit of the free energy
\begin{equation}\label{eq:functional}
\phi(\bQ) = - \sum_{s,t \leq n} \frac{\rho_s \rho_t}{4 \Delta_{s,t}}\Tr( \bQ_s \bQ_t) +  \sum_{s \leq n} \rho_s \E_{\vz,\vx^0} \ln \bigg[ \int \exp \bigg( \bigg( \tilde \bQ_s\vx^0 + \sqrt{\tilde \bQ_s} \vz \bigg)^\trans \vx - \frac{\vx^\trans \tilde \bQ_s \vx}{2}  \bigg) \, d \pP_0(\vx) \bigg]
\end{equation}
where 
\[
\tilde \bQ_s = \sum_{t \leq n} \frac{1}{\Delta_{s,t}} \rho_t \bQ_t.
\]
We begin with the case when the prior distribution is Gaussian, which will allow us to explicitly compute a closed form of the solution. Later, we will generalize this analysis to general bounded priors. 

\subsection{Standard Gaussian Prior}
Suppose that the prior is a standard Gaussian on $\R^\kappa$, 
\[
d\pP_0(\vx) = \frac{1}{(2 \pi)^{\frac{\kappa}{2}}} e^{- \frac{\vx \cdot \vx}{2}} .
\]
In this case, the second term in the functional \eqref{eq:functional} can be computed explicitly,
\begin{align*}
&\quad \sum_{s =1}^n \rho_s \E_{\vz,\vx^0} \ln \bigg[ \int \exp \bigg( \bigg( \tilde \bQ_s\vx^0 + \sqrt{\tilde \bQ_s} \vz \bigg)^\trans \vx - \frac{\vx^\trans \tilde \bQ_s \vx}{2}  \bigg) \, d \pP_0(\vx) \bigg]
\\&= \sum_{s =1}^n \rho_s \E_{\vz,\vx^0} \ln \bigg[ \frac{1}{(2\pi)^{\kappa /2}} \int \exp \bigg( \bigg( \tilde \bQ_s\vx^0 + \sqrt{\tilde \bQ_s} \vz \bigg)^\trans \vx - \frac{\vx^\trans \tilde \bQ_s \vx}{2} - \frac{x^\trans \bI x}{2} \bigg) \, d\vx \bigg]
\\&= \sum_{s =1}^n \rho_s \E_{\vz,\vx^0} \ln \bigg[ \frac{1}{(2\pi)^{\kappa /2}} \int \exp \bigg( - \frac{1}{2} \Big(\vx -  (\tilde\bQ_s + \bI)^{-1}( \tilde \bQ_s \vx^0 + \sqrt{\tilde \bQ_s} \vz ) \Big)^\trans (\tilde\bQ_s + \bI) \Big(\vx - (\tilde\bQ_s + \bI)^{-1}(\tilde \bQ_s \vx^0 + \sqrt{\tilde \bQ_s} \vz)\Big)  
\\&\qquad+ \frac{1}{2}  \Big( \tilde \bQ_s \vx^0  + \sqrt{\tilde \bQ_s} \vz\Big)^\trans (\tilde\bQ_s + \bI)^{-1} \Big(\tilde \bQ_s \vx^0 + \sqrt{\tilde \bQ_s} \vz\Big)  \bigg) \, d\vx \bigg]
\\&= \sum_{s =1}^n \rho_s \bigg(- \frac{1}{2} \ln \det(\tilde \bQ_s + \bI) + \frac{1}{2} \E_{\vz,\vx^0} \Big( \tilde \bQ_s \vx^0  + \sqrt{\tilde \bQ_s} \vz\Big)^\trans (\tilde\bQ_s + \bI)^{-1} \Big(\tilde \bQ_s \vx^0 + \sqrt{\tilde \bQ_s} \vz\Big)  \bigg)
\\&= \sum_{s =1}^n \rho_s \bigg(-\frac{1}{2} \ln \det (\tilde \bQ_s + \bI ) + \frac{1}{2} \Tr \Big( \tilde \bQ_s (\tilde\bQ_s + \bI)^{-1}\tilde \bQ_s  \Big) + \frac{1}{2} \Tr \Big( \sqrt{\tilde \bQ_s} (\tilde\bQ_s + \bI)^{-1} \sqrt{\tilde \bQ_s} \Big) \bigg).
\end{align*}
In the last computation, we used the fact that $\vx^0,\vz$ are independent centered standard Gaussians in $\R^\kappa$.
By adding and subtracting the identity matrix and using the fact that the trace is invariant under cyclic permutations,
\begin{align*}
&\frac{1}{2} \Tr \Big( \tilde \bQ_s (\tilde\bQ_s + \bI)^{-1}\tilde \bQ_s  \Big) + \frac{1}{2} \Tr \Big( \sqrt{\tilde \bQ_s} (\tilde\bQ_s + \bI)^{-1} \sqrt{\tilde \bQ_s} \Big) 
\\&= \frac{1}{2} \Tr \Big( \tilde \bQ_s (\tilde\bQ_s + \bI)^{-1} (\tilde \bQ_s + \bI - \bI)  \Big) + \frac{1}{2} \Tr \Big(\tilde \bQ_s (\tilde\bQ_s + \bI)^{-1} \Big)
\\&= \frac{1}{2} \Tr ( \tilde \bQ_s ).
\end{align*}
Therefore, the general replica symmetric functional $\phi_g$ in the standard Gaussian case is
\begin{align}
\phi_g( \bQ ) &= - \sum_{s,t = 1}^n \frac{\rho_s \rho_t}{4 \Delta_{s,t}}\Tr( \bQ_s \bQ_t)  + \sum_{s =1}^n \frac{\rho_s}{2} \bigg( \Tr(\tilde \bQ_s)  -\ln \det (\tilde \bQ_s + \bI) \bigg).
\end{align}
where 
\[
\tilde \bQ_s = \sum_{t =1}^n \frac{1}{\Delta_{s,t}} \rho_t \bQ_t 
\]
We next investigate the maximizers of $\phi_g$ on the set of symmetric positive semidefinite matrices. Observe that $\phi_g$ is a $C^\infty$ function that goes to $-\infty$ when $\sup_{s} \| \bQ_s\|_{op} \to \infty$ hence the supremum is attained. We are interested in finding conditions on $\Delta_{s,t}$ and $\rho_s$ to determine when $\phi(\bQ)$ has a maximizer at the origin $\bQ =0$. We define the following block matrices indexed by $s \leq n$,
\begin{equation}\label{eq:covarianceparameters}
\tilde \bQ = (\tilde \bQ_s)_{s \leq n} \quad \frac{1}{\Delta} = \Big( \frac{1}{\Delta_{s,t} } \Big)_{s,t \leq n} \quad \rho = \diag( \rho_1, \dots, \rho_n ).
\end{equation}

\begin{lem}[Phase Transition with Gaussian Prior]
\begin{enumerate}
	\item The functional $\phi$ has a unique maximizer at $\bQ = 0$ if
	\[
	\Big\|\sqrt{ \rho} \frac{1}{\bD} \sqrt{\rho} \Big\|_{op} < 1
	\]
	\item The functional $\phi$ achieves its maximum value away from $\bQ = 0$  if
	\[
	\Big\|\sqrt{ \rho} \frac{1}{\bD} \sqrt{\rho} \Big\|_{op} > 1
	\]
\end{enumerate}
\end{lem}
\begin{rem}
In the case when $n = 1$, we have that the phase transition happens when $\frac{1}{\Delta} = 1$, which is precisely the phase transition in that model \cite{lelargemiolanematrixestimation}.
\end{rem}
\begin{proof}\hfill\\\\
\textit{First and Second Variations:} Consider an arbitrary perturbation of $\bM = (\bM_1, \dots, \bM_n)$ of $\bQ$. The first and second variation is denoted by
\begin{equation}\label{eq:1stand2ndvariations}
	\nabla_{\bM} \phi_g(\bQ) = \partial_{\epsilon} \phi_g(\bQ + \epsilon \bM) \Big|_{\epsilon = 0} \quad\text{and}\quad \nabla_{\bM}^2 \phi_g(\bQ) = \partial^2_{\epsilon} \phi_g(\bQ + \epsilon \bM) \Big|_{\epsilon = 0}.
\end{equation}
We can compute these directly to see that
\[
\nabla_{\bM} \phi_g(\bQ) = \bigg( -\sum_{s,u} \frac{\rho_s \rho_u}{2\Delta_{s,u} } \Tr( \bQ_u \bM_s ) + \sum_{s,u} \frac{\rho_s \rho_u}{2\Delta_{s,u} } \Tr( (\bI -(\tilde \bQ_u + \bI)^{-1}) \bM_s ) \bigg)
\]
and
\[
\nabla_{\bM}^2 \phi_g(\bQ) = -\sum_{s,t} \frac{\rho_s \rho_t}{2\Delta_{s,t} } \Tr( \bM_s \bM_t ) + \sum_{s,t} \sum_{u}  \frac{\rho_s \rho_t \rho_u}{2\Delta_{s,u} \Delta_{u,t}} \Tr((\tilde \bQ_u + \bI)^{-1} \bM_s (\tilde \bQ_u + \bI)^{-1} \bM_t ).
\]
\textit{Small Operator Norm:} We first find a sufficient condition for $\phi_g$ to have a unique maximizer at the origin. At the critical point, the derivative is equal to $0$ for all symmetric $\bM$, so we can conclude that
\begin{equation}\label{eq:critptGaussian}
	\tilde \bQ_s = \sum_{u} \frac{\rho_u}{\Delta_{s,u} }  \bQ_u  = \sum_{u} \frac{\rho_u}{\Delta_{s,u} } \big(\bI - (\tilde \bQ_u + \bI)^{-1} \big)  \qquad \forall s.
\end{equation}
Consider the following vectors of $\kappa \times \kappa$ matrices
\[
\tilde \bQ = (\tilde \bQ_s)_{s \leq n} \quad f(\tilde \bQ) = ( \bI - (\tilde \bQ_s + \bI)^{-1} )_{s \leq n} = ( \tilde \bQ_s (\tilde \bQ_s + \bI)^{-1} )_{s \leq n}.
\]
The critical point condition can be simplified to
\begin{equation}\label{eq:gaussiancritpt}
	\tilde \bQ =  \frac{1}{ \Delta} \rho  f(\tilde \bQ) \implies \sqrt{\rho} \tilde \bQ = \Big( \sqrt{\rho} \frac{1}{\Delta} \sqrt{\rho} \Big) ( \sqrt{\rho} f(\tilde \bQ) ) = \bA  ( \sqrt{\rho} f(\tilde \bQ) ).
\end{equation}
where $\bA = \sqrt{\rho} \frac{1}{\Delta} \sqrt{\rho}$. We therefore can compute the $L^2$ norm
\begin{equation}\label{eq:contraction}
	\| \sqrt{\rho} \tilde \bQ \|_2^2 = \Tr\Big( (\sqrt{\rho} \tilde \bQ)^\trans (\sqrt{\bm{\rho}} \tilde \bQ ) \Big) = \Tr\Big( ( \bA ( \sqrt{\rho} f(\tilde \bQ) )^\trans (\bA ( \sqrt{\rho} f(\tilde \bQ)) \Big) \leq \|\bA \|_{op} \|\sqrt{\rho}f(\tilde \bQ)\|_2^2.
\end{equation}
Because $\|\sqrt{\rho} f(\tilde \bQ)\|_2 \leq \|\sqrt{\rho} \tilde \bQ\|_2 \| (\tilde \bQ + \bI)^{-1}\|_2 \leq \|\sqrt{\rho} \tilde \bQ\|_2$, we arrive at
\[
\| \sqrt{\rho} \tilde \bQ \|_2^2 \leq  \|\bA \|_{op} \|\sqrt{\rho}\tilde \bQ\|_2^2.
\]
In particular, when $\|\bA \|_{op} < 1$ there exists a unique solution to the critical point equation at $\tilde \bQ = 0$.\\\\ 
\noindent\textit{Large Operator Norm:} We examine the Taylor expansion of $\phi$ around the origin. The first and second variation simplify greatly at the origin $\bQ = 0$ to give
\begin{equation}\label{eq:hessianGaussian}
	\nabla_{\bM} \phi_g(0) = 0 \quad \nabla_{\bM}^2 \phi_g(0) = \sum_{s,t} \bigg( -\frac{\rho_s \rho_t}{2\Delta_{s,t} } \Tr( \bM_s \bM_t ) + \sum_{u}  \frac{\rho_s \rho_t \rho_u}{2\Delta_{s,u} \Delta_{u,t}} \Tr(\bM_s \bM_t ) \bigg).
\end{equation}
The second variation will always be negative for any choice of $\bM$ if and only if
\[
\rho \frac{1}{\bD}\rho \frac{1}{\bD}\rho-\rho \frac{1}{\bD}\rho = \rho \sqrt{\frac{1}{\bD}}  \bigg( \sqrt{\frac{1}{\bD}} \rho \sqrt{\frac{1}{\bD}} - \bI  \bigg) \sqrt{\frac{1}{\bD}} \rho
\]
has only negative eigenvalues. By Sylvester's law of inertia, it suffices to study the eigenvalues of the matrix
\[
\sqrt{\frac{1}{\bD}} \rho \sqrt{\frac{1}{\bD}} - \bI 
\]
and by the invariance of the operator norm of cyclic permutations of the matrices, we get the condition that the Hessian is negative semidefinite when
\[
\|\bA\|_{op} = \Big\|\sqrt{ \rho} \frac{1}{\bD} \sqrt{\rho} \Big\|_{op} < 1
\]
which is precisely the condition for a unique maximizer at $0$. We now claim that when the Hessian has a strictly positive eigenvalue, then there exists a maximizer away from the origin.

Suppose now that $\|\bA\|_{op} > 1$. Then there exists a unit eigenvector $v\in \R^n$ such that $\bA v = \|\bA\|_{op} v$. Furthermore, the entries of $\bA$ are non-negative so the entries of the eigenvector $v$ are also non-negative by the Perron--Frobenius Theorem. Therefore, if we define $\bM = (u_1\bI, \dots, u_n \bI ) $ where $u_i = \frac{1}{\sqrt{\rho_i}} v_i \geq 0$ then the Hessian at $0$ defined in \eqref{eq:hessianGaussian} is given by
\[
\nabla_{\bM}^2 \phi_g(0) = n \bigg(	u^\trans \rho \frac{1}{\bD_{s,t}}\rho \frac{1}{\bD_{s,t}}\rho u - u^\trans\rho \frac{1}{\bD_{s,t}}\rho u \bigg) = n ( v^\trans\bA^2 v - v^\trans \bA v ) \geq n (\|\bA \|_{op}^2 - \|\bA\|_{op})
\]
which is strictly non-negative if $\|\bA\|_{op} > 1$. 

Therefore, the function $g(\epsilon) = \phi_g(\epsilon \bM)$ satisfies $g''(0) > 0$, so it is convex on $[0,\delta]$ for some $\delta > 0$. Furthermore, $g'(0) = 0$, so we can conclude that  for $\delta$ small enough 
\[
\phi_g(\delta \bM) = g(\delta) > g(0) = \phi_g(0)
\]
by convexity. Lastly, notice that $\bM \geq 0$ since the entries of $u$ are non-negative, so $\bM$ is in the domain of the optimization problem.  Observe also that $0$ is no longer a maximizer.
\end{proof}

\subsection{General Centered Prior}

Now suppose that we are in the scenario that $\pP_X$ is a centered prior measure on $\R^\kappa$ with compact support. We want to study the maximizers of the functional \eqref{eq:functional}. Let $\langle \cdot \rangle_{ \bQ}$ denote the average with respect the Gibbs measure associated with the Hamiltonian $(  \bQ\vx^0 + \sqrt{ \bQ} \vz )^\trans \vx - \frac{\vx^\trans  \bQ \vx}{2}$.  The partial derivatives of $\phi$ in the direction $\bM_s$ can be computed using Gaussian integration by parts and the Nishimori property \eqref{nishimori}
\begin{align}
\partial_{\epsilon}  \phi(\bQ + \epsilon \bM_s) |_{\epsilon = 0} &=  \bigg( - \sum_{t = 1}^n \frac{\rho_s \rho_t}{2\Delta_{s,t}} \Tr( \bQ_t \bM_s)  + \sum_{t =1 }^n \frac{\rho_s \rho_t}{ \Delta_{s,t} } \E \bigg\langle \vx^\trans \bM_s \vx^0  - \frac{\vx^\trans \bM_s \vx}{2} \bigg\rangle_{\tilde \bQ_t} \bigg) 
\\&\quad + \sum_{t = 1}^n \rho_t \E \bigg\langle   \vz^\trans ( \partial_{\varepsilon} \sqrt{ \tilde \bQ_t+\varepsilon \tilde \bM_{t} }|_{\varepsilon=0}  ) \vx \bigg\rangle \notag
\\&= \frac{\rho_s}{2} \bigg( - \sum_{t = 1}^n \frac{\rho_t}{\Delta_{s,t}} \Tr( \bQ_t \bM_s)  + 2\sum_{t =1 }^n \frac{\rho_t}{ \Delta_{s,t} } \E \bigg\langle \vx^\trans \bM_s \vx^0 - \frac{\vx^\trans \bM_s \vx^2}{2} \bigg\rangle_{\tilde \bQ_t} \bigg) \notag
\\&= \frac{\rho_s}{2} \bigg( - \sum_{t = 1}^n \frac{\rho_t}{\Delta_{s,t}} \Tr( \bQ_t \bM_s )  + \sum_{t =1 }^n \frac{\rho_t}{ \Delta_{s,t} } \E \langle \vx^\trans \bM_s \vx^0  \rangle_{\tilde \bQ_t} \bigg) \label{eq:firstvargeneral}.
%\\&= \frac{\rho_s}{2} \bigg( - \sum_{t = 1}^n \frac{\rho_t}{\Delta_{s,t}^2} \Tr(bQ_s \bM_s)  + \sum_{t =1 }^n \frac{\rho_t}{ \Delta_{s,t}^2 } \E \langle \vx \rangle^2  \bigg)
\end{align}
We dealt with the square root that appeared in the first equality using the identity 
\begin{equation}\label{eq:derivsqrt}
\Tr\big(\bA \sqrt{\tilde \bQ_t} \partial_{\epsilon} \sqrt{\tilde \bQ_t + \epsilon \tilde \bM_t} |_{\epsilon = 0} \big) = \Tr\big(\bA \sqrt{\tilde \bQ_t} \partial_{\epsilon} \sqrt{\tilde \bQ_t + \frac{\epsilon \rho_s}{\Delta_{st}} \bM_s} |_{\epsilon = 0} \big)= \frac{\rho_s}{2\Delta_{st}} \Tr( \bA \bM_s ),	
\end{equation}
for any symmetric matrix $\bA$, which implies that for standard Gaussian vectors $\vz$,
\[
\E \big( \vx^\trans \sqrt{\tilde \bQ_t}  \vz \big) \big( \vz^\trans ( \partial_{\varepsilon} \sqrt{ \tilde \bQ_t+ \frac{\varepsilon \rho_s}{\Delta_{st}} \bM_{s} }|_{\varepsilon=0} ) \vx^2 \big) = \Tr\Big( \vx^2 \vx^\trans \sqrt{\tilde \bQ_t} \partial_{\varepsilon} \sqrt{ \tilde \bQ_t+ \frac{\varepsilon \rho_s}{\Delta_{st}} \bM_{s} }|_{\varepsilon=0}  \Big) = \frac{\rho_s}{2 \Delta_{st}} \Tr\Big( \vx^2 \vx^\trans \bM_s \Big) .
\]
Therefore, Gaussian integration by parts implies that
\begin{align*}
\E \Big\langle  \vz^\trans ( \partial_{\varepsilon} \sqrt{ \tilde \bQ_t+\varepsilon \tilde \bM_{s} }|_{\varepsilon=0} 
) \vx \Big\rangle_{\tilde \bQ_t} &= \E \Big\langle \E_z ( \vx^\trans \sqrt{\tilde \bQ_t}  \vz ) ( \vz^\trans (\partial_{\varepsilon} \sqrt{ \tilde \bQ_t+ \frac{\varepsilon \rho_s}{\Delta_{st}} \bM_{s} }|_{\varepsilon=0} 
) \vx )   \Big\rangle_{\tilde \bQ_t} \\
&\qquad - \E \Big\langle \E_z ( (\vx^2)^\trans \sqrt{\tilde \bQ_t}  \vz ) (\vz^\trans ( \partial_{\varepsilon} \sqrt{ \tilde \bQ_t+ \frac{\varepsilon \rho_s}{\Delta_{st}} M_{s} }|_{\varepsilon=0} 
) \vx) \Big\rangle_{\tilde \bQ_t} \notag 
\\&= \frac{\rho_s}{\Delta_{st}} \E \Big\langle \frac{\vx^\trans \bM_s \vx}{2} - \frac{\vx^\trans \bM_s \vx^2}{2} \Big\rangle_{\tilde \bQ_t} \label{eq:IBPsqrt}
\end{align*}
where we used the same convention for the partial derivatives given in \eqref{eq:1stand2ndvariations}.
At the critical point, the first derivative must vanish for all $\bM_s$, so the critical point equation simplifies to
\begin{equation}\label{eq:generalcrticalpointcondition}
\tilde \bQ_s = \sum_{t =1 }^n \frac{\rho_t}{ \Delta_{s,t} } \E \langle \vx^0 \vx^\trans \rangle_{\tilde \bQ_t} = \sum_{t =1 }^n \frac{\rho_t}{ \Delta_{s,t} } \E \langle \vx  \rangle_{\tilde \bQ_t} \langle \vx \rangle_{\tilde \bQ_t}^\trans.
\end{equation}
by the Nishimori property. We first prove that the average $f(\bQ) = \E \langle \vx^0 \vx^\trans \rangle_{\bQ}$ is Lipschitz.

\begin{lem}[Lipschitz Continuity] \label{lem:Lip}
The functional $f: \mathbb S_{\kappa}^{+} \to \mathbb S_{\kappa}^{+}$ is Lipschitz in the space $\mathbb S_{\kappa}^{+}$ of non-negative $\kappa\times \kappa$ symmetric matrices: for any $\bQ,\bQ'\in\mathbb S_{\kappa}^{+}$,
\[
\| f( \bQ) - f(\bQ') \|_2 \leq 3\kappa^2 C^3 \| \bQ - \bQ' \|_2
\]
where $C$ is the bound on the support of $\pP_X$ and $\kappa$ is the dimension.
\end{lem}
\begin{proof}
The function $f: \mathbb{S}_\kappa \to \R^{\kappa \times \kappa}$ is also given by
\[
f(\bQ) = \E \frac{ \int \vx^0\vx^\trans 	\exp \bigg( \bigg( \bQ \vx^0 + \sqrt{\bQ} \vz \bigg)^\trans \vx - \frac{\vx^\trans  \bQ \vx}{2}  \bigg) \, d\pP_X(\vx) }{  \int \exp \bigg( \bigg( \bQ\vx^0 + \sqrt{\bQ} \vz \bigg)^\trans \vx - \frac{\vx^\trans  \bQ \vx}{2}  \bigg) \, d\pP_X(\vx) } = \E \langle \vx^0 \vx^\trans \rangle_{\bQ}.
\]
Given arbitrary $\bQ_1, \bQ_2 \geq 0$ we define $ \bQ_\alpha = \alpha\bQ_1 + (1-\alpha)\bQ_2 = \bQ_2 + \alpha(\bQ_1 - \bQ_2)$ to be the interpolation between the matrices. We will show that the function
\[
g_{k,k'}(\alpha) := f_{k,k'}(\alpha\bQ_1 + (1-\alpha)\bQ_2) =  f_{k,k'}(\bQ_2 + \alpha (\bQ_1 - \bQ_2 )) =  \E \langle \vx^0 (k)\vx(k') \rangle_{\bQ_\alpha}
\]
has uniformly bounded derivative for $t \in [0,1]$. Let $\bM = \bQ_1 - \bQ_2$. A similar computation as the derivation of the first variation in \eqref{eq:firstvargeneral} via the Gaussian integration by parts computation on \eqref{eq:IBPsqrt} implies that
\begin{align*}
	|g'_{k,k'}(\alpha)| &= \bigg| \E\bigg\langle \vx^0(k)\vx(k') \bigg(\vx^\trans \bM \vx^0 + \vz^\trans ( \partial_{\alpha} \sqrt{ \bQ + \alpha \bM }  )|_{\alpha = 0} \vx - \frac{\vx^\trans \bM \vx}{2}\bigg) \bigg\rangle_{\bQ} 
	\\&\quad -  \E\bigg\langle \vx^0(k)\vx(k') \bigg((\vx^2)^\trans \bM \vx^0 + \vz^\trans ( \partial_{\alpha} \sqrt{ \bQ + \alpha \bM } ) |_{\alpha = 0} \vx^2 - \frac{(\vx^2)^\trans \bM (\vx^2)}{2}\bigg) \bigg\rangle_{\bQ}  \bigg|
	\\&\leq \bigg| \E \bigg\langle \vx^0(k)\vx(k') \bigg( \vx^\trans \bM \vx^0 - \frac{\vx^\trans \bM \vx^2}{2} \bigg) \bigg\rangle_{\bQ} \bigg| + \bigg| \E \bigg\langle \vx^0(k)\vx(k') \bigg( (\vx^2)^\trans \bM \vx^0 - \frac{(\vx^2)^\trans \bM \vx^3}{2} \bigg) \bigg\rangle_{\bQ} \bigg|
	\\&\leq 2C^2 \E \bigg\langle  \bigg( \|\bM\|_2 \| \vx^0 \vx^\trans \|_2 + \frac{1}{2} \|\bM\|_2 \| \vx^2 \vx^\trans \|_2 \bigg) \bigg\rangle_{\bQ}
	\\&\leq 3 \kappa C^3 \| \bM \|_2.
\end{align*}
In the last second line, we used the Cauchy--Schwarz inequality on the Frobenius inner product and in the last line we used the fact that $\E \langle  \cdot \rangle_{\bQ}$ is the average with respect to a probability measure.  

By the mean value theorem, we can conclude that
\[
| f_{k,k'}(\bQ_1) - f_{k,k'}(\bQ_2)| = |g_{k,k'}(1) - g_{k,k'}(0)| \leq 3 \kappa C^3 \| \bQ_1 - \bQ_2 \|_2.
\]
Lastly, we get our required estimate
\[
\| f(\bQ_1) - f(\bQ_2) \|_2 =  \bigg( \sum_{k,k' = 1}^\kappa ( f_{k,k'}(\bQ_1) - f_{k,k'}(\bQ_2)  )^2 \bigg)^{1/2} \leq 3\kappa^2 C^3 \| \bQ_1 - \bQ_2 \|_2.
\]
Taking $\bQ_1 = \tilde \bQ_t$ and $\bQ_2 = \tilde \bQ'_t$ gives our estimate. 
\end{proof}

Using this Lipschitz continuity, we can do a fixed point argument to show that there exists a unique maximizer at $0$ if the covariances $\frac{1}{\Delta^2_{s,t}}$ are sufficiently small.

\begin{lem}[Uniqueness of a Maximizer at $0$]\label{prop:gentransition1}
Consider the model parameters defined in \eqref{eq:covarianceparameters}. The functional $\phi(\bQ)$ has a unique maximizer at $0$ if
\[
\Big\|\sqrt{ \rho} \frac{1}{\bD} \sqrt{\rho} \Big\|_{op} < \frac{1}{9 \kappa^4 C^6}.
\]

\end{lem}
\begin{proof}

Following the computations leading to \eqref{eq:contraction} applied to the general critical point equation \eqref{eq:generalcrticalpointcondition}, we see that
\begin{equation}\label{eq:fixedpointgeneral}
	\| \sqrt{\rho} \tilde \bQ \|_2^2 = \Tr\Big( (\sqrt{\rho} \tilde \bQ)^\trans (\sqrt{\rho} \tilde \bQ ) \Big) = \Tr\Big( ( \bA ( \sqrt{\rho} f(\tilde \bQ) )^\trans (\bA ( \sqrt{\rho} f(\tilde \bQ)) \Big) \leq \|\bA \|_{op} \|\sqrt{\rho}f(\tilde \bQ)\|_2^2
\end{equation}
where
\[
f(\tilde \bQ)  = ( \E \langle \vx^0 \vx^\trans \rangle_{\tilde \bQ_1},\dots, \E \langle \vx^0 \vx^\trans \rangle_{\tilde \bQ_n} )\qquad \bA = \sqrt{ \rho} \frac{1}{\bD} \sqrt{\rho}
\]
We proved in Lemma~\ref{lem:Lip} that $f(\tilde \bQ_t) = \E \langle  \vx^0 \vx^\trans\rangle_{\tilde \bQ_t}$
is Lipschitz in $\tilde \bQ_t$ with a Lipschitz constant that does not depend on $t$
\[
\|f(\tilde \bQ_t) - f(\tilde \bQ_t')\|_2 \leq 3\kappa^2 C^3 \| \tilde \bQ_t - \tilde \bQ'_t\|_2.
\]
Recall that the prior is centered $f(0) = 0$, so
\[
\|\sqrt{\rho}f(\tilde \bQ)\|_2^2 = \sum_{t} \rho_t \|f(\tilde \bQ_t) - f(0)\|_2^2 \leq 9 \kappa^4 C^6 \sum_{t} \rho_t \| \tilde \bQ_t \|_2^2. \leq 9 \kappa^4 C^6  \| \sqrt{\rho}\tilde \bQ \|_2^2.
\]
Applying this bound to \eqref{eq:fixedpointgeneral} yields the result.
\end{proof}

This bound is clearly not tight, because it does not depend on the measure $\pP_0$ except through the support contained in $[-C,C]^\kappa$. We will show below that if the operator norm is sufficiently large, then a maximizer away from $0$ will appear. 

\begin{lem}[Phase Transition with Centered Prior]\label{prop:gentransition2} Let $\vx \sim \pP_X$. The functional $\phi$ has a maximizer away from $\bQ = 0$ if
\[
\bigg\| \sqrt{\rho} \frac{1}{\bD} \sqrt{\rho} \bigg\|_{op} > \frac{1}{ \| \E \vx \vx^\trans \|^2_{op}} =  \frac{1}{\| \Cov(\vx) \|^2_{op}}.
\]
\end{lem}

\begin{rem}
If we have a Gaussian prior, then $\vx \sim N(0,\bI)$, so $\| \E  \vx \vx^\trans \|_{op} = \| \bI \|_{op} = 1$. Furthermore, it also agrees with condition (194) in \cite{lenkarmt} in the case when $\kappa = 1$.
\end{rem}

\begin{proof}
We adapt the proof of the Gaussian prior to the general scenario. Consider an arbitrary perturbation $\bM= (\bM_1, \dots, \bM_n)$ of $\bQ$ such that $\bQ + \epsilon \bM \geq 0$ for all $\epsilon$ sufficiently small. The first and second variations are denoted by
\[
\nabla_{\bM} \phi(\bQ) = \partial_{\epsilon} \phi(\bQ + \epsilon \bM) \Big|_{\epsilon = 0} \quad \text{and} \quad \nabla^2_{\bM} \phi(\bQ) = \partial^2_{\epsilon} \phi(\bQ + \epsilon \bM) \Big|_{\epsilon = 0} .
\]
This can be computed explicitly using integration by parts and the Nishimori property,
\[
\nabla_{\bM} \phi(\bQ) = \sum_{s,u} \frac{\rho_s \rho_u}{ 2 \Delta_{s,u}} \E \big\langle \Tr( \vx^\trans \bM_{s} \vx_0 ) \big\rangle_{\tilde \bQ_u} - \sum_{s,u} \frac{\rho_s \rho_u}{2\Delta_{s,u}} \Tr(\bQ_u \bM_s) 
\]
and
\begin{align*}
	\nabla^2_{\bM} \phi(\bQ) &=\sum_{s,t,u} \frac{\rho_s \rho_t\rho_u}{ 4 \Delta_{s,u} \Delta_{u,t}} \Big( \E \big\langle ( \vx^\trans \bM_{s} \vx_0 )( \vx^\trans \bM_{t} \vx_0 ) \big\rangle_{\tilde \bQ_u} -  \E \big\langle ( \vx_1^\trans \bM_{s} \vx_0 )( \vx_2^\trans \bM_{t} \vx_0 ) \big\rangle_{\tilde \bQ_u} \Big)
	\\&\quad - \sum_{s,t} \frac{\rho_s \rho_t}{2\Delta_{s,t}} \Tr(\bM_s \bM_t) .
\end{align*}

When $\bQ = 0$ the Hessian simplifies because $\langle \cdot \rangle_{0}$ is simply the average with respect to $\pP_X$ which is centered,
\begin{align}
	\nabla^2_{\bM} \phi(0) &=  \sum_{s,t,u} \frac{\rho_s \rho_t\rho_u}{ 4 \Delta_{s,u}^2 \Delta_{u,t}^2}  \E \big\langle \Tr( \bM_{s} \vx_0\vx^\trans  )^2 \big\rangle_{0} - \sum_{s,t} \frac{\rho_s \rho_t}{2\Delta_{s,t}^2} \Tr(\bM_s \bM_t).\label{eq:Hessiangeneral}
\end{align}
For $\vx \sim \pP_X$, let $\bC = \E \vx \vx^\trans = \Cov(\vx)$. Let $b$ denote the unit eigenvector corresponding to the largest eigenvalue of $\bC$. In particular, we have $b^\trans \bC b = \| \bC\|_{op}$. Suppose that
\begin{equation}\label{eq:bignormgeneral}
	\|\bA\|_{op} = \bigg\| \sqrt{\bm{\rho}} \frac{1}{\bD} \sqrt{\bm{\rho}} \bigg\|_{op} > \frac{1}{\E \| \vx_1 \vx_2^\trans \|_2^2} = \frac{1}{\|\bC\|_{op}^2}.
\end{equation}

The matrix $\bA = \sqrt{\bm{\rho}} \frac{1}{\Delta^2} \sqrt{\bm{\rho}}$ has non-negative entries, the eigenvector $v$ associated with the largest eigenvalue has non-negative entries by the Perron--Frobenius Theorem. Therefore, we can take 
\[
\bM = ( u_1 \bB, \dots, u_n \bB ) \qquad \text{where $\bB = b b^\trans$ and $u_i = \frac{1}{\sqrt{\rho_i}} v_i$}
\]
Notice that
\[
\E \big\langle \Tr( \bB \vx_0\vx^\trans  )^2 \big\rangle_{0} = \E \bigg\langle \sum_{i,j,k,l}b_i b_j b_k b_\ell x_i x^0_j x_k x_\ell^0  \bigg \rangle_0  = (b^\trans \bC b )^2 = \| \Cov(\vx) \|_{op}^2.
\]
since $\vx$ and $\vx^0$ have the same distribution in the Bayes optimal case. The Hessian at $0$ defined in \eqref{eq:Hessiangeneral} simplifies greatly for this choice of $\bM$,
\[
\nabla_{\bM}^2 \phi(0) = \bigg(	\|\bC \|_{op}^2  {u}^\trans \rho \frac{1}{\bD_{s,t}}\rho \frac{1}{\bD_{s,t}}\rho u - u^\trans\rho \frac{1}{\bD_{s,t}}\rho {u} \bigg) =  \|\bC\|_{op}^2 v^\trans\bA^2 v -  {v}^\trans \bA {v}  \geq \|\bC\|_{op}^2 \|\bA \|_{op}^2 -  \|\bA\|_{op}.
\]
By our assumption \eqref{eq:bignormgeneral}, $\nabla_{\bM}^2 \phi(0) > 0$. Therefore, the function $g(\epsilon) = \phi(\epsilon \bM)$ satisfies $g''(0) > 0$, so it is convex on $[0,\delta]$ for some $\delta > 0$. Furthermore, $g'(0) = 0$, so we can conclude that 
\[
\phi(\delta \bM) = g(\delta) \geq g(0) = \phi(0)
\]
by convexity. Lastly, notice that $\bM \geq 0$ since the entries of $\vec{u}$ are non-negative, so $\bM$ is in the domain of the optimization problem. 
\end{proof}
\begin{rem}
Our choice of $\bM$ in the proof is the choice of directional derivative that maximizes the Hessian at $0$. This gives us the sharp condition when the Hessian of $\phi$ is no longer negative semidefinite.
\end{rem}

\subsection{Comparison with the Naive BBP Transition}

In this section, we show that the threshold Lemma~\ref{prop:gentransition2} is stronger than the transition computed by examining the BBP transition of a spiked matrix with homogeneous noise and inhomogenous spike. We will see that the transitions are equal if and only if the models are homogeneous.

We first consider the model with homogeneous noise, but inhomogeneous signal. We want to find the BBP transition of the matrix 
\[
\frac{1}{\sqrt{N}} Y^\Delta \odot \frac{1}{\Delta^{\odot \frac{1}{2}}} = \frac{x^0 (x^0)^\trans}{N} \odot \frac{1}{\Delta^{\odot \frac{1}{2}}} + \frac{1}{\sqrt{N}} G.
\]
where $G$ is a Gaussian Wigner matrix.
\begin{lem}[BBP Transition]
	The matrix $\frac{1}{\sqrt{N}} Y^\Delta \odot \frac{1}{\Delta^{\odot \frac{1}{2}}}$ has an outlier iff
	\[
	\bigg\| \sqrt{\bp} \frac{1}{\bD^{\odot\frac{1}{2}} }\sqrt{\bp} \bigg\|_{op} > \frac{1}{ \| \E \vx \vx^\trans \|_{op} }.
	\]
\end{lem}
\begin{proof} It is well known that since $G$ follows the GOE the largest eigenvalue of $\frac{1}{\sqrt{N}} Y^\Delta \odot \frac{1}{\Delta^{\odot \frac{1}{2}}}$ is given by
	$\lambda^{*}=\max_{i:\gamma_{i}>1}\{\gamma_{i}+\gamma_{i}^{-1}\}$ if $(\gamma_{i})_{1\le i\le n}$ are the eigenvalues of
	$$R= \frac{x^0 (x^0)^\trans}{N} \odot \frac{1}{\Delta^{\odot \frac{1}{2}}}$$
	But exactly as in the proof of Theorem~\ref{univspec} and Remark~\ref{rem:finiterank}, we can see that the eigenvalues of $R$ are the same as those of
	$$M_{st}^{N}=\frac{1}{N} \sum_{j = 1}^\kappa \|v^{j}_{s}\|\|v^{j}_{t}\|\frac{1}{\Delta_{st}^{\odot \frac{1}{2}}}$$
	a finite matrix whose top eigenvalue converges entrywise when $N$ goes to infinity, by the law of large numbers, towards the top eigenvalue of the matrix
	$$M_{st}^{\infty}= \| \E \vx \vx^\trans \|_{op}  \sqrt{\rho_{s}\rho_{t}}\frac{1}{\Delta_{st}^{\odot \frac{1}{2}}}$$
	by the definition $v^1$ in Remark~\ref{rem:finiterank}. 
	We conclude that $\frac{1}{\sqrt{N}} Y^\Delta \odot \frac{1}{\Delta^{\odot \frac{1}{2}}}$ has an outlier iff
	$$\bigg\|\sqrt{\rho} \frac{1}{\Delta^{\odot\frac{1}{2}}}\sqrt{\rho} \bigg\|_{op}>\frac{1}{ \| \E \vx \vx^\trans \|_{op} }$$
\end{proof}

We now prove that the BBP transition is strictly weaker than the spin glass transition except in the homogeneous models. That is, if $\Delta$ is such that the BBP transition happens
\[
\bigg\| \sqrt{\bp} \frac{1}{\bD^{\odot \frac{1}{2}}} \sqrt{\bp} \bigg\|_{op} \geq \frac{1}{ \| \E \vx \vx^\trans \|_{op} }
\]
then we are also in the information theoretically feasible detectable region
\[
\bigg\| \sqrt{\bp} \frac{1}{\bD}\sqrt{\bp} \bigg\|_{op} \geq \frac{1}{  \| \E \vx \vx^\trans \|^2_{op} }.
\]
In particular, if it is possible to detect the signal using the spectral method, then it is also possible to detect is using any other method. The converse is false, so there exists some algorithms that beat the naive spectral ones. 

This statement is true, and in fact a stronger statement holds, which clearly implies Proposition~\ref{prop:gap}.
\begin{lem}[Gap in Thresholds]
We have
\[
\bigg\| \sqrt{\bp} \frac{1}{\bD^{\odot \frac{1}{2}}} \sqrt{\bp} \bigg\|^2_{op} \leq \bigg\| \sqrt{\bp} \frac{1}{\bD} \sqrt{\bp} \bigg\|_{op},
\] 
with equality holding if and only if there exists a constant $c$ such that for all $s,t$, $\bD_{s,t} =c$. In particular, if $\Delta$ satisfies
\[
\bigg\| \sqrt{\bp} \frac{1}{\bD^{\odot \frac{1}{2}}} \sqrt{\bp} \bigg\|_{op} \geq \frac{1}{\| \E \vx \vx^\trans \|_{op} }
\]
then
\[
\bigg\| \sqrt{\bp} \frac{1}{\bD} \sqrt{\bp} \bigg\|_{op} \geq 	\bigg\| \sqrt{\bp} \frac{1}{\bD^{\odot \frac{1}{2}}} \sqrt{\bp} \bigg\|^2_{op} \geq \frac{1}{ \| \E \vx \vx^\trans \|_{op}^2 }.
\]
\end{lem}
\begin{proof}
First of all, observe that $ \sqrt{\bp} \frac{1}{\bD^{\odot \frac{1}{2}}} \sqrt{\bp}$ and $\sqrt{\bp} \frac{1}{\bD} \sqrt{\bp}$ have non-negative entries, so the Perron--Frobenius theorem implies that the largest eigenvector has non-negative entries. Therefore, we can restrict ourselves to eigenvectors with non-negative entries to compute the maximum eigenvalues. 
But for any vector $u$ with non-negative entries, by Cauchy-Schwartz's inequality
\begin{eqnarray*}\langle u,  \sqrt{\bp} \frac{1}{\bD^{\odot \frac{1}{2}}} \sqrt{\bp} u\rangle 
	&=&\sum_{s,t}\frac{\rho_{s}^{\frac{1}{4}}u_{s}^{\frac{1}{2}}\rho_{t}^{\frac{1}{4}}u_{t}^{\frac{1}{2}}}{\Delta_{s,t}^{1/2}} \rho^{\frac{1}{4}}_{s}u_{s}^{\frac{1}{2}}\rho^{\frac{1}{4}}_{t}u_{t}^{\frac{1}{2}}\\
	&\le& \left(\langle u, \sqrt{\bp} \frac{1}{\bD} \sqrt{\bp} u\rangle\right)^{\frac{1}{2}}\left(\sum \rho_{s}^{\frac{1}{2}}u_{s}\right)^{\frac{1}{2}}\\
	&\le &\left(\langle u, \sqrt{\bp} \frac{1}{\bD} \sqrt{\bp} u\rangle\right)^{\frac{1}{2}}\|u\|_{2}^{\frac{1}{2}}\\
\end{eqnarray*} where we used again Cauchy--Schwartz's inequality and the fact that $\sum \rho_{s}=1$ to see that
$\sum \rho_{s}^{\frac{1}{2}}u_{s}\le \|u\|_{2}$. We deduce by taking the supremum over $u$ so that $\|u\|_{2}=1$ that

$$\bigg\|\sqrt{\bp} \frac{1}{\bD^{\odot \frac{1}{2}}} \sqrt{\bp} \bigg\|_{op}\le \bigg\| \sqrt{\bp} \frac{1}{\bD} \sqrt{\bp} \bigg\|_{op}^{1/2}\,.$$
Furthermore, equality holds only if there is equality in the above Cauchy-Schwartz inequality which happens when we take $u$ to be the largest eigenvector of the matrix $\sqrt{\bp} \frac{1}{\bD^{\odot \frac{1}{2}}} \sqrt{\bp}$. For this to happen, the largest eigenvector must be $u=\sqrt{\rho}$ so there exists $c$ such that
	$$ \frac{\rho_{s}^{\frac{1}{2}}\rho_{t}^{\frac{1}{2}}}{\Delta_{s,t}^{\frac{1}{2}}}=c  \rho^{\frac{1}{2}}_{s}\rho^{\frac{1}{2}}_{t},\quad \forall \, s,t$$
	implying that $\Delta_{s,t}$ must be constant.

\end{proof}

\subsection{Proof of  Corollary~\ref{prop:limitingMMSE} and Lemma~\ref{prop:recovery}}

We now explain how one can use the phase transition to recover the recovery transitions. Recall that the minimal matrix mean squared error is given by
\[
\mathrm{MMSE}(N) = \frac{2}{N (N-1)} \sum_{i < j} \E  ( \vx_i^0 \cdot \vx_j^0 - \E [ \vx_i^0 \cdot \vx_j^0 \given Y] )^2 
\]
The limit of the MMSE will follow from the following property of the maximizers of the replica symmetric functional $\phi$ defined in \eqref{eq:RSformula}.
\begin{lem}[Limit of the Overlaps]\label{lem:critpt}
If $\frac{1}{\Delta}$ is positive definite, the maximizers of \eqref{eq:RSformula} satisfy
\[
\lim_{N \to \infty} \sum_{s,t = 1}^n \rho_s \rho_t \langle \Tr(\bR^s_{10}, \bR^t_{10} ) \rangle = \sum_{s,t = 1}^n \rho_s \rho_t  \Tr(\bQ_s \bQ_t ) .
\]
Furthermore, for any maximizer, the values of $\Tr(\bQ_s, \bQ_t ) \geq 0$ are unique for fixed $s,t$. 
\end{lem}
\begin{proof}
Using the change of variables $\eta_{s,t} = \frac{1}{\Delta_{s,t}}$ we consider our Gaussian free energy
\[
F_N(\eta) = \frac{1}{N} \E_Y \log \int e^{\sum_{s,t} \sum_{i \in I_s, j \in I_t} \frac{\sqrt{\eta_{s,t}}}{ \sqrt{ N} } g_{ij} (\vx_i \cdot \vx_{j} ) + \frac{\eta_{st}}{N} (\vx^0_i \cdot \vx^0_{j} )(\vx_i \cdot \vx_{j} )   -   \frac{\eta_{st}}{2 N} (\vx_i \cdot \vx_{j} )^2.  } \, d \pP_0^{\otimes N}(x)
\]
where $\eta_{i,j} = \eta_{s,t}$ for $i \in I_s$ and $j \in I_t$, which is equivalent to \eqref{eq:FEGauss}. Notice that differentiating the free energy with respect to $\eta_{u,v} := \frac{1}{\Delta_{u,v}}$ and an application of the Nishimori property recovers the Gibbs averages,
\begin{align*}
	\partial_{\eta_{u,v}} F_N(\eta) &= \E  \bigg\langle \frac{1}{2N} \sum_{i\in I_u, j \in I_v} \frac{g_{ij} (x_i \cdot x_j) }{\sqrt{N \eta_{s,t}}   } + \frac{\rho_u \rho_v}{2} \Tr(\bR_{1,0}^u \bR_{1,0}^v) - \frac{ \rho_u \rho_v }{4} \Tr(\bR_{1,1}^u \bR_{1,1}^v)  \bigg\rangle
	\\&=  \E \bigg\langle \frac{\rho_u \rho_v}{4} \Tr(\bR^u_{1,1} \bR_{1,1}^v) - \frac{\rho_u \rho_v}{4} \Tr(\bR^u_{1,2} \bR_{1,2}^v) + \frac{\rho_u \rho_v}{2} \Tr(\bR_{1,0}^u \bR_{1,0}^v) - \frac{ \rho_u \rho_v }{4} \Tr(\bR_{1,1}^u \bR_{1,1}^v)  \bigg\rangle + o(1)
	\\&= \E \bigg\langle \frac{\rho_u \rho_v}{4} \Tr(\bR^u_{1,0} \bR_{1,0}^v)   \bigg\rangle + o(1)
\end{align*}
where $\langle \cdot \rangle$ is the Gibbs average with respect to the Hamiltonian \eqref{eq:HamilGauss}. 
On the other hand, for fixed $1 \leq u,v \leq n$, consider the functional
\[
\psi(\Delta) = \sup_{\bQ} \bigg[ - \sum_{s,t =1}^n \frac{\rho_s \rho_t}{4 \Delta_{st}}\Tr( \bQ_{s} \bQ_{t}) +  \sum_{s =1}^n \rho_s \E_{\vz,\vx^0} \ln \bigg[ \int e^{ ( \tilde \bQ_{s}\vx^0 + \sqrt{\tilde \bQ_{s}} \vz )^\trans \vx - \frac{\vx^\trans \tilde \bQ_{s} \vx}{2}} \, d \pP_0(\vx) \bigg] . \bigg]
\]
When $\frac{1}{\bD}$ is invertible, $\tilde \bQ = \frac{1}{\bD} \bp \bQ$ implies that $\bQ = \bp^{-1} (\frac{1}{\bD})^{-1} \tilde \bQ$. In this notation, it is understood that for $\bQ, \tilde \bQ \in (\R^{\kappa \times \kappa} )^n$ and the entries fo the vector
\[
\tilde \bQ_s = \sum_{t} \frac{1}{\Delta_{st}} \rho_t \bQ_t \in \R^{\kappa \times \kappa} 
\]
are non-negative because $\bQ_1, \dots, \bQ_n$ are. Under this change of variables, we have
\begin{equation}\label{eq:psi}
	\psi(\eta) = \sup_{\tilde \bQ} \bigg[ - \frac{1}{4}\Tr( \tilde \bQ^\trans \bigg( \frac{1}{ \bD} \bigg)^{-1} \tilde \bQ) +  \sum_{s =1}^n \rho_s \E_{\vz,\vx^0} \ln \bigg[ \int e^{ ( \tilde \bQ_{s}\vx^0 + \sqrt{\tilde \bQ_{s}} \vz )^\trans \vx - \frac{\vx^\trans \tilde \bQ_{s} \vx}{2}} \, d \pP_0(\vx) \bigg] . \bigg]
\end{equation}
where the supremum is taken over non-negative matrices. Observe that the function we are optimizing is continuous. Moreover,  the second term goes to infinity at most like $\left(\sum_{s}\Tr (\tilde Q_{s}^{2})\right)^{1/2}$, whereas the first term goes to $-\infty$   like $\sum_{s}\Tr (\tilde Q_{s}^{2})$ since we assumed that the smallest eigenvalue of $\frac{1}{\Delta}$ is positive. Hence,  we can restrict the supremum to non-negative matrices such that $\sum_{s}\Tr (\tilde Q_{s}^{2})$ is bounded by some finite $M$, which is a compact space. 
	Using the change of variables $\eta_{s,t} = \frac{1}{\Delta_{s,t}}$, the envelope theorem \cite[Corollary~4]{envelope} implies that at any point where $\psi$ is differentiable,
	\[
	\partial_{\eta_{u,v}} \psi(\eta):= \frac{d}{d\epsilon}  \psi(\eta_{u,v} + \epsilon ) \bigg|_{\epsilon = 0}  =
	\frac{\rho_u \rho_v}{4} \Tr( \bQ^*_u \bQ^*_v) 
	\]
	where $\bQ^*$ is any maximizer of the right hand side of \eqref{eq:psi}. Furthermore, if $\bQ^\dagger$ is another maximizer, then
	\[
	\frac{\rho_u \rho_v}{4} \Tr( \bQ^*_u \bQ^*_v)  = \frac{\rho_u \rho_v}{4} \Tr( \bQ^\dagger_u \bQ^\dagger_v) 
	\]
	so the value of $\Tr( \bQ^*_u \bQ^*_v)$ is unique even though the maximizers $\bQ^*$ and $\bQ^\dagger$ may not be.

But $\psi$ is convex in $\eta$ so that $(\partial_{\eta_{u,v}} \psi(\eta))_{u,v}$
exists for almost every $\eta$, so does the limit of the derivative of the free energy Theorem~\ref{thm:main} implies
\[
\lim_{N \to \infty}  F_N(\eta) =  \psi(\eta) \qquad \text{and} \qquad \lim_{N \to \infty} \partial_{\eta_{u,v}} F_N(\eta) = \partial_{\eta_{u,v}} \psi(\eta)
\]
for almost every $\eta$ so
\[
\lim_{N \to \infty} \sum_{u,v = 1}^n \E \bigg\langle \frac{\rho_u \rho_v}{4} \Tr(\bR^u_{1,0} \bR_{1,0}^v)   \bigg\rangle =   \sum_{u,v = 1}^n \rho_u \rho_v  \Tr(\bQ_u^{*}\bQ_v^{*} ) ,
\]
finishing the proof.
\end{proof}

We now prove the statements about the MMSE.
\begin{proof}[Proof of Corollary~\ref{prop:limitingMMSE}]
Recall that $\langle \cdot \rangle$ is the average with respect to $\pP(x \given Y)$. The mean squared error can be simplified to
\begin{align*}
	&\quad \frac{2}{N (N-1)} \sum_{i < j} \E  ( \vx_i^0 \cdot \vx_j^0 - \E [ \vx_i^0 \cdot \vx_j^0 \given Y] )^2 
	\\&= \frac{2}{N (N-1)} \sum_{i < j} \E \Big( ( \vx_i^0 \cdot \vx_j^0 )^2 - 2(\vx_i^0 \cdot \vx_j^0)\E [ \vx_i^0 \cdot \vx_j^0 \given Y]  + \E [ \vx_i^0 \cdot \vx_j^0 \given Y]^2 \Big) 	
	\\&= \frac{2}{N (N-1)} \sum_{i < j} \E \Big\langle ( \vx_i^0 \cdot \vx_j^0 )^2 - 2(\vx_i^0 \cdot \vx_j^0) (\vx_i^1 \cdot \vx_j^1)  + (\vx_i^1 \cdot \vx_j^1)(\vx_i^2 \cdot \vx_j^2) \Big\rangle 	
	\\&= \E \Big\langle \Tr(\bR_{00},\bR_{00}) - 2\Tr(\bR_{10}, \bR_{10} )  + \Tr(\bR_{12}, \bR_{12} ) \Big\rangle
	\\&= \E \Tr(\bR_{00},\bR_{00}) - \E \langle \Tr(\bR_{10}, \bR_{10} ) \rangle
	\\&= \E \| x x^\trans \|_2^2 - \sum_{s,t = 1}^n \rho_s \rho_t \E \langle \Tr(\bR^s_{10}, \bR^t_{10} ) \rangle.
\end{align*}
By Lemma~\ref{lem:critpt} it follows that
\[
\mathrm{MMSE}(\bD) = \E \| x x^\trans \|_2^2 - \sum_{s,t = 1}^n \rho_s \rho_t \Tr( \bQ_s, \bQ_t ).
\]
\end{proof}

Lemma~\ref{prop:recovery} is now immediate.

\begin{proof}[Proof of Lemma~\ref{prop:recovery}]
By Lemma~\ref{prop:gentransition1} and Corollary~\ref{prop:limitingMMSE}, if $\|\sqrt{ \rho} \frac{1}{\bD} \sqrt{\rho} \|_{op} < \frac{1}{9 \kappa^4 C^6}$ then $\bQ_s = 0$ for all $s$, so $\lim_{N \to \infty} \mathrm{MMSE}(N)  =  \E_{\pP_0} \| x x^\trans \|_2^2$. On the other hand, if $\| \sqrt{\rho} \frac{1}{\bD} \sqrt{\rho} \|_{op} > \frac{1}{ \| \E \vx \vx^\trans \|^2_{op}} =  \frac{1}{\| \Cov(\vx) \|^2_{op}}$ then there exists a $s$ such that $\bQ_s > 0$, so $\lim_{N \to \infty} \mathrm{MMSE}(N)  < \E_{\pP_0} \| x x^\trans \|_2^2$.
\end{proof}

\bibliographystyle{plain}

\end{document}